\documentclass[leqno,final]{siamltex}
\usepackage{graphicx} 
\usepackage{amsmath,amstext,amssymb,bm}
\usepackage{leftidx}

\usepackage{xcolor} 
\usepackage{soul} 
\usepackage{tikz}
\usetikzlibrary{shapes,arrows}
\usepackage{mathrsfs}
\usepackage{epstopdf}
\usepackage{color}
\usepackage{multirow}
\usepackage{tabularx}
\usepackage[shortlabels]{enumitem}

\usepackage{colortbl}
\usepackage{booktabs}

\numberwithin{equation}{section}
\newtheorem{remark}{Remark}[section]

\allowdisplaybreaks[4]

\usepackage{hyperref}
\hypersetup{
	colorlinks=true,   
	linkcolor=blue,    
	citecolor=blue,    
	filecolor=magenta, 
	urlcolor=cyan      
}

\def\u{\mathbf{u}}

\def\X{\textbf{X}}
\def\v{\textbf{v}}
\def\w{\textbf{w}}

\def\V{\textbf{V}}
\def\H{\textbf{H}}

\def\a{\textbf{a}}

\def\g{\textbf{g}}

\def\nuu{\boldsymbol{\nu}}
\def\sigmaa{{\boldsymbol{\sigma}}}
\def\Y{\textbf{Y}}

\begin{document}
	
	\title{Unconditionally stable Gauge--Uzawa finite element schemes  for the
		chemo-repulsion-Navier-Stokes system\thanks{This work was supported by National Natural Science Foundation of China (No. 12471406) and the Science
			and Technology Commission of Shanghai Municipality (Grant Nos. 22JC1400900, 22DZ2229014).} 
		}
	\markboth{CHENYANGLI, PING LIN AND HAIBIAO ZHENG}{GU-FEM for the
		chemo-repulsion-Navier-Stokes system }

	\author{Chenyang Li
		\thanks{School of Mathematical Sciences, East China Normal University, Shanghai, 200241, China. \texttt{(chenyangli1004@yeah.net)}.}
			\and  Ping Lin
		\thanks{Division of Mathematics, University of Dundee, Dundee, DD1 4HN, UK. \texttt{(P.Lin@dundee.ac.uk)} 
		}
		\and  Haibiao Zheng
		\thanks{ \textbf{Corresponding author}. School of Mathematical Sciences, Ministry of Education Key Laboratory of Mathematics and Engineering Applications, Shanghai Key Laboratory of PMMP,  East China Normal University, Shanghai, 200241, China. \texttt{(hbzheng@math.ecnu.edu.cn)} }
	}

	\maketitle
	
	\begin{abstract} 
This paper investigates a Gauge--Uzawa finite element method (GU-FEM) for the two-dimensional chemo--repulsion--Navier--Stokes (CRNS) system.
The proposed approach establishes a fully discrete projection framework that integrates the advantages of both the canonical and Uzawa--type formulations, while preserving variational consistency. 
The proposed GU-FEM possesses two notable advantages:
(1) it requires no initial pressure value;
(2) it avoids artificial pressure boundary conditions, thereby reducing computational overhead.
Furthermore, the scheme is shown to be unconditionally energy stable, unique solvability and optimal error estimates are derived for the cell density, chemical concentration, and fluid velocity. Finally, several numerical experiments are presented to demonstrate the accuracy, stability, and efficiency of the proposed consistent projection finite element method.
	\end{abstract}
	
	
	
\section{Introduction}	\label{introduction}
%

Chemotaxis refers to the directed movement of microorganisms in response to chemical stimuli, where they move toward favorable (chemo-attraction) or away from unfavorable (chemo-repulsion) chemical environments \cite{filbet2006}. This phenomenon plays a crucial role in various biological processes, including biofilm formation, nitrogen fixation, microbial migration in soil, and oil recovery \cite{karmakar2021}.
The Keller--Segel system is a classical mathematical model for describing chemotactic behavior. While its theoretical analysis has been extensively studied, the corresponding numerical analysis remains relatively limited.
A key challenge lies in constructing numerical schemes that preserve essential physical properties such as positivity, mass conservation, and energy dissipation.
Recent studies have proposed several energy-stable or positivity--preserving schemes including finite volume and finite element approaches \cite{chatard2014,guillen1,guillen2,guillen3,guillen4,shen2020,wang2022}.

As we know, the microorganisms often do not exist in isolation in nature, which often live in the viscous fluid so that microorganisms and
chemoattractant are transported with the fluid.
In this work, we consider a chemo-repulsion--Navier--Stokes system with a quadratic production term, 
which describes the interaction between the motion of microorganisms and the surrounding incompressible fluid under chemical repulsion effects.
The governing equations are given by \cite{wang2023}
\begin{align}\label{biosav-model}
	\begin{cases}
		\eta_t + \nabla \cdot (\u \eta) - \mu_1 \Delta \eta = \nabla \cdot (\eta \nabla c), & \text{in } \Omega \times (0, T), \\[3pt]
		c_t + \nabla \cdot (\u c) - \mu_2 \Delta c + c = \dfrac{1}{2} \eta^2, & \text{in } \Omega \times (0, T), \\[3pt]
		\u_t + (\u \cdot \nabla) \u - \mu_3 \Delta \u + \nabla p = -\eta \nabla \eta - \nabla \cdot (\nabla c \otimes \nabla c), & \text{in } \Omega \times (0, T), \\[3pt]
		\nabla \cdot \u = 0, & \text{in } \Omega \times (0, T),
	\end{cases}
\end{align}
where $\Omega \subset \mathbb{R}^2$ denotes a bounded domain with Lipschitz continuous boundary $\partial \Omega$, $\eta(x,t)$ and $c(x,t)$ denote the cell or microorganism density and chemical concentration, respectively, 
while $\u(x,t)$ and $p(x,t)$ represent the velocity and pressure of the incompressible fluid.
The positive constants $\mu_1$, $\mu_2$, and $\mu_3$ correspond to the diffusion and viscosity coefficients. 

The advection terms $\nabla \cdot(\u \eta)$ and $\nabla \cdot(\u c)$ describe the transport of microorganisms and chemical substances 
caused by the background flow field. In contrast, the nonlinear forcing terms $-\eta \nabla \eta$ and $-\nabla \cdot(\nabla c \otimes \nabla c)$ 
characterize the feedback of biological effects on the fluid motion. 
The former represents the self-interaction of the cell density, while the latter arises from the gradient-induced 
chemical stress on the surrounding fluid. These coupling mechanisms lead to a complex interplay between chemotaxis and fluid convection.

Up to now, numerous analytical results have been established concerning the existence and uniqueness of solutions to the classical chemotaxis--fluid system \cite{chae2013,kozono2016,winkler2012,winkler2014,zhang2014}.
In parallel, several studies have investigated numerical methods for the chemotaxis--Navier--Stokes system. For instance, in \cite{lee2015}, the authors numerically examined the formation of falling bacterial plumes induced by bioconvection in a three-dimensional chamber. A high-resolution vorticity-based hybrid finite volume/finite difference scheme for the chemotaxis--fluid model was proposed in \cite{chertock2012}, while an upwind finite element method was developed in \cite{deleuze2016} to analyze pattern formation and hydrodynamic stability.
However, despite these advances, rigorous numerical analysis for the fully coupled chemotaxis--fluid system remains relatively scarce.
In \cite{feng2021}, the authors designed a linear, positivity--preserving finite element scheme for the chemotaxis--Stokes system, yet the fluid convection term was neglected, limiting its applicability to more general flow regimes.
In \cite{guillen5}, finite element error estimates were derived for the chemotaxis--Navier--Stokes equations, though the pressure approximation was only partially addressed.
A linear, decoupled fully-discrete finite element scheme by combining the
scalar auxiliary variable (SAV) approach, implicit-explicit (IMEX) scheme and pressure-projection method is presented in \cite{lijian2023}, and another  linear and decoupled structure,
which is constructed by using the discontinuous Galerkin (DG) method for spatial discretization, the implicit--explicit (IMEX) approach for the highly nonlinear and coupling terms, and a pressure-projection method for the
Navier-Stokes equations is done in \cite{wang2023}.

In the late 1960s, projection methods were introduced by Chorin~\cite{chorin1968} and Temam~\cite{temam1969} to decouple the velocity $\u$ and pressure $p$, thereby reducing computational cost. However, these classical methods suffer from drawbacks, such as yielding momentum equations inconsistent with the original  equation and degrading numerical accuracy.  
To address these issues, the Gauge method was developed by Oseledets~\cite{oseledets1989} and later refined in \cite{liu2003}. While the Gauge formulation is theoretically equivalent to the momentum equation and performs well under various boundary conditions, it still has disadvantages--particularly, the tangential boundary derivative is difficult to implement in a finite element framework, often leading to accuracy loss.  
To overcome these drawbacks, Nochetto and Pyo \cite{pyo2005} proposed the Gauge--Uzawa method for the Navier--Stokes equations, which was later extended to the Boussinesq equations \cite{pyo2006}. Subsequent works \cite{pyo2009,pyo2013} provided optimal error estimates and developed a second-order version of this method. In this article, we put forward a first
order projection method which combines the Gauge and Uzawa methods for the  Chemo-repulsion-Navier-Stokes system.

It is worth noting that most existing studies on the chemotaxis--fluid system have focused on the chemo-attraction behavior of bacteria moving toward chemoattractants, whereas only a few works have addressed energy-stable schemes for the chemotaxis--fluid coupling.
In this paper, we consider a chemo--repulsion--Navier--Stokes system to describe the repulsive movement of microorganisms in a fluid environment.
A remarkable property of this model is that the chemo--repulsion--fluid system inherently satisfies an energy dissipation law, which provides a natural foundation for the construction of stable numerical schemes.
We introduce an auxiliary variable $\boldsymbol{\sigma} = \nabla c$ to handle the cross-diffusion term, thereby controlling the strong regularity requirements of the original system.
A Gauge-Uzawa technique is employed for the Navier--Stokes equations to decouple the velocity and pressure, 
 leading to a fully discrete scheme.
Other contribution of this work is the rigorous proof of the unconditional energy stability, unique solvability and the optimal error estimates of the proposed method, which rely on several technical tools.
Finally, a series of numerical experiments are performed to verify the stability, accuracy, and efficiency of the proposed scheme.

The structure of this paper is organized as follows. In Section 2, we review some preliminary knowledge and derive the chemo--repulsion--Navier--Stokes system. Section 3 introduces a fully discrete Gauge--Uzawa finite element scheme for the chemotaxis--fluid system, and we establish its unconditional energy stability and unique solvability at the discrete level. In Section 4, we focus on deriving the optimal error estimates for the proposed scheme. Section 5 presents numerical experiences to demonstrate the accuracy and efficiency of the developed method. Finally, Section 6 concludes the paper with some remarks.

\section{Preliminary}
For $k\in N^+$ and $1\leq p\leq +\infty$, we denote $L^p(\Omega)$ and $W^{k, p}(\Omega)$ as the classical Lebesgue space and Sobolev space, respectively. The norms of these spaces are denoted by
\begin{align*}
	||u||_{L^p(\Omega)}&=\left(\int_{\Omega}|u(\mathbf{x})|^p dx \right)^\frac{1}{p},\\
	||u||_{W^{k,p}(\Omega)}&=\left(\sum\limits_{|j|\leq k}||D^ju||_{L^p(\Omega)}^p\right)^\frac{1}{p},
\end{align*}
within this context, $W^{k, 2}(\Omega)$ is also known as the Hilbert space and can be expressed as $H^k(\Omega)$.  $||\cdot||_{L^\infty}$ represents the norm of the space  $L^\infty(\Omega)$ which is defined as
\begin{equation*}
	||u||_{L^\infty(\Omega)}=ess\sup\limits_{x\in \Omega}|u(x)|.
\end{equation*}

For simplicity, we denote the inner products of both
$L^2(\Omega)$ and $\textbf{L}^2(\Omega)$
by $(\cdot,\cdot)$,  namely,
\begin{align*}
	\begin{split}
		&(u,v)=\int_\varOmega u(x)v(x) d x \quad \forall  \, u,v\in L^2(\Omega),\\
		&({\bf u},{\bf v})=\int_\Omega {\bf u}(x)\cdot {\bf v}(x) dx \quad \forall \, {\bf u},{\bf v}\in {\bf L}^2(\Omega) .\\
	\end{split}
\end{align*}

We use the following function spaces
\begin{align}
	Q &:= H_0^1(\Omega) := \{\varphi \in H^1(\Omega) : \varphi = 0 \text{ on } \partial\Omega\}, \\
	\X &:= \H_{\sigmaa}^1(\Omega) := \{\sigmaa \in [\H^1(\Omega)]^2 : \sigmaa \cdot \nuu = 0 \text{ on } \partial\Omega\}, \\
	\Y &:= \H_0^1(\Omega) := \{\v \in [\H^1(\Omega)]^2 : \v= 0 \text{ on } \partial\Omega\}, \\
	M &:= L_0^2(\Omega) := \{q \in L^2(\Omega) : \int_{\Omega} q \, dx = 0\}.
\end{align}

We introduce appropriately defined skew-symmetric trilinear forms, which facilitate the stability analysis and the derivation of error estimates.
\begin{align}
	\begin{split}
		b(\textbf{u},\textbf{v},\textbf{w}) &= \int_\Omega (\u \cdot \nabla) \textbf{v} \cdot\textbf{w} dx + \frac{1}{2} \int_{\Omega} (\nabla \cdot \u) \v \cdot \w dx \\
		&=   \frac{1}{2}\int_\Omega (\u \cdot \nabla) \textbf{v} \cdot\textbf{w} dx -  \frac{1}{2} \int_{\Omega} (\u \cdot \nabla) \w \cdot\v dx  \quad \forall ~\u,\v,\w\in\textbf{Y},\\
	\end{split}
\end{align}
which has the following properties \cite{heyinnian2005,heyinnian2022}
\begin{align}
	& b(\textbf{u},\textbf{v},\textbf{v})=0, \quad b(\textbf{u},\textbf{v},\textbf{w}) = - b(\textbf{u},\textbf{w},\textbf{v}),\label{biogu-trilinear1}\\
	&b(\textbf{u},\textbf{v},\textbf{w}) \leq C \| \nabla \textbf{u} \|_{L^2} \| \nabla \textbf{v} \| _{L^2} \| \nabla \textbf{w} \|_{L^2},\quad b(\textbf{u},\textbf{v},\textbf{w})  \leq C \| \textbf{u} \| ^{\frac{1}{2}}_{L^2} \| \nabla \textbf{u}\| ^{\frac{1}{2}}_{L^2} \| \nabla \textbf{v} \| _{L^2} \| \nabla \textbf{w} \|_{L^2}.
\end{align}

If $\nabla \cdot \u=0$, there holds $b(\textbf{u},\textbf{v},\textbf{w}) = ((\textbf{u} \cdot \nabla) \textbf{v},\textbf{w})= \int_\Omega (\u \cdot \nabla) \textbf{v} \cdot\textbf{w} dx $.

In what follows, we establish two fundamental lemmas concerning the discrete Gronwall inequality, which will serve as key analytical tools in the subsequent theoretical development.
\begin{lemma}\cite{heyinnian2015}
	Let $C_0, a_n, b_n, d_n$ be nonnegative numbers with integer $n \ge 0$ such that
	\begin{equation}
		a_m + \tau \sum_{n=1}^{m} b_n 
		\le \tau \sum_{n=0}^{m-1} d_n a_n + C_0, 
		\quad \forall m \ge 1,
		\label{eq:2.7}
	\end{equation}
	then there holds that
	\begin{equation}
		a_m + \tau \sum_{n=1}^{m} b_n 
		\le C_0 \exp\!\left(\tau \sum_{n=0}^{m-1} d_n\right), 
		\quad \forall m \ge 1.
		\label{eq:2.8}
	\end{equation}
\end{lemma}

\begin{lemma}\cite{heywood1990,thomee2006}
	Let $a_k, b_k, c_k, \gamma_k$ be sequences of nonnegative numbers such that 
	$\tau \gamma_k < 1$ for all $k$, and let $g_0 \ge 0$ be such that the following inequality holds:
	\begin{equation}
		a_K + \tau \sum_{k=0}^{K} b_k 
		\le \tau \sum_{k=0}^{K} \gamma_k a_k 
		+ \tau \sum_{k=0}^{K} c_k + g_0.
		\label{eq:2.9}
	\end{equation}
	Then
	\begin{equation}
		a_K + \tau \sum_{k=0}^{K} b_k 
		\le \left( \tau \sum_{k=0}^{K} c_k + g_0 \right)
		\exp\!\left(\tau \sum_{k=0}^{K} \sigma_k \gamma_k \right),
		\label{eq:2.10}
	\end{equation}
	where $\sigma_k = (1 - \tau \gamma_k)^{-1}$.
\end{lemma}

The following Sobolev inequalities in 2D will be used frequently \cite{evans1949}.
\begin{align}\label{biogu-sobolev}
	\begin{cases}
		&W^{2,4}(\Omega)  \hookrightarrow W^{1,\infty}(\Omega),\\ 
		&H^2(\Omega) \hookrightarrow W^{1,q}(\Omega), \quad 2 \leq q < \infty,\\
		&H^2(\Omega) \hookrightarrow L^{\infty}(\Omega).
	\end{cases}
\end{align}

\begin{lemma}
	[Div-grad inequality] \cite{nochetto2004,nochetto2005,pyo2002,temam1977}
	If $\mathbf{v} \in \mathbf{H}^1_0(\Omega)$, then
	$$
	\|\operatorname{div} \mathbf{v}\|_{L^2} \le \|\nabla \mathbf{v}\|_{L^2}.
	$$
\end{lemma}

We need the following regularities assumption for the convergence analysis.

\textbf{A1}: Assume the exact solutions satisfy the following regularities
\begin{subequations}\label{biogu-regularity}
	\begin{align}
&\u \in L^\infty (0,T; H^2), \quad \u_t\in L^2(0,T;H^2\cap L^2)\quad \u_{tt}\in L^2(0,T;L^2), \\
&p\in L^2 (0,T; L^2\cap H^1),\\
&\eta \in L^{\infty}(0,T; W^{2,4}), \quad \eta_t\in L^2(0,T;H^2\cap L^2),\quad \eta_{tt}\in L^2(0,T;L^2),\\
&\sigmaa\in L^{\infty} (0,T;H^2), \quad \sigmaa_t\in L^2(0,T;H^2\cap L^2),\quad \sigmaa_{tt}\in L^2(0,T;L^2).
	\end{align}
\end{subequations}

\subsection{An equivalent system}
For analytical convenience, the chemotactic stress tensor can be decomposed as \cite{wang2023,lin2011,lin2014}
\begin{align}
	\nabla \cdot (\nabla c \otimes \nabla c)
	= \Delta c \, \nabla c + \frac{1}{2} \nabla |\nabla c|^2.
\end{align}

The gradient term $\frac{1}{2}\nabla|\nabla c|^2$ can be absorbed into the pressure term 
by defining a modified pressure
\begin{align}
	\tilde{p} = p + \frac{1}{2}|\nabla c|^2.
\end{align}

Hence, the system \eqref{biosav-model} can be equivalently rewritten in a more compact form as
\begin{subequations}\label{biosav-remodel}
	\begin{align}
		\eta_t + \nabla \cdot (\u \eta) - \mu_1 \Delta \eta &= \nabla \cdot (\eta \nabla c), && \text{in } \Omega \times (0,T), \label{biosav-3}\\
		c_t + \nabla \cdot (\u c) - \mu_2 \Delta c + c &= \dfrac{\eta^2}{2}, && \text{in } \Omega \times (0,T), \label{biosav-1}\\
		\u_t + (\u \cdot \nabla)\u - \mu_3 \Delta \u + \nabla p &= -\eta \nabla \eta - \Delta c \nabla c, && \text{in } \Omega \times (0,T), \\
		\nabla \cdot \u &= 0, && \text{in } \Omega \times (0,T), \label{biosav-4}
	\end{align}
\end{subequations}
where, for simplicity, we still denote the modified pressure $\tilde{p}$ by $p$.

The system \eqref{biosav-remodel} is supplemented with the following boundary and initial conditions:
\begin{align}\label{eq:2.4}
	\frac{\partial \eta}{\partial \nuu} = \frac{\partial c}{\partial \nuu} = 0, 
	\quad \u = 0, 
	\quad \text{on } \partial \Omega \times [0,T],
\end{align}
and
\begin{align}\label{eq:2.5}
	\eta(x,0) = \eta_0(x) \ge 0, \quad 
	c(x,0) = c_0(x) \ge 0, \quad 
	\u(x,0) = \u_0(x), 
	\quad \text{in } \Omega,
\end{align}
where $\nuu$ is the outward unit normal vector to $\partial \Omega$. 
The conditions in \eqref{eq:2.4} impose the no-flux boundary constraint for both $n$ and $c$, 
while the velocity field $\u$ satisfies the classical no-slip boundary condition.
The initial data $\eta_0$, $c_0$, and $\u_0$ represent the prescribed distributions of microorganism density, 
chemical concentration, and fluid velocity, respectively.

 The system (\ref{biosav-remodel}) satisfies the following energy dissipation law \cite{wang2023}:
$$
\mathcal{E}^1(\u, n, c) \leq \mathcal{E}^1(\u_0, n_0, c_0).
$$
where the energy functional is defined as
$$
\mathcal{E}^1(\u, n, c) = \frac{1}{2} \int_{\Omega} |\u|^2  dx + \frac{1}{2} \int_{\Omega} |n|^2  dx + \frac{1}{2} \int_{\Omega} |\nabla c|^2  dx.
$$

 It is worth noting that, the above energy law enables one to prove the existence and uniqueness of weak solution with certain
smoothness by a standard Galerkin procedure \cite{cai2008,evans1949}.

In order to numerically control the second-order nonlinear term arising in the cell  density equation, we employ the technique proposed by \cite{guillen3}, where the auxiliary variable $\boldsymbol{\sigma} = \nabla c$ is introduced. This reformulation enables the construction of a finite element scheme that imposes weaker constraints on the choice of discrete parameters. The same technique has been used in \cite{zhang2016,tang2023,duarte2}

First, by applying the gradient operator to (\ref{biosav-1}), we obtain
\begin{align}\label{biosav-2}
	\nabla c_t + \nabla(\nabla \cdot (\u c)) - \mu_2 \nabla \Delta c + \nabla c = \eta \nabla \eta.
\end{align}

Next, we introduce the auxiliary variable $\boldsymbol{\sigma} = \nabla c$ and add the term $\mu_2 \nabla \times (\nabla \times \boldsymbol{\sigma})$ using that  the vector identity $\nabla \times \sigmaa= \nabla \times (\nabla c) = 0$, we can equivalently rewrite equation~(\ref{biosav-2}) as
\begin{align}
	\boldsymbol{\sigma}_t + \nabla (\u \cdot \boldsymbol{\sigma}) - \mu_2 \Delta \boldsymbol{\sigma} + \boldsymbol{\sigma} - \eta \nabla \eta = 0,
\end{align}
where we have employed the identities 
\begin{align}
-\nabla(\nabla \cdot \boldsymbol{\sigma}) + \nabla \times (\nabla \times \boldsymbol{\sigma}) = -\Delta \boldsymbol{\sigma}, 
\qquad 
\nabla (\nabla \cdot (\u c)) = \nabla (\u \cdot \boldsymbol{\sigma}),
\end{align}
the latter of which follows from the incompressibility condition $\nabla \cdot \u = 0$.

Therefore, the system (\ref{biosav-remodel}) can be equivalently represented as
\begin{subequations}\label{biosav-rremodel}
\begin{align}
	\eta_t +  \nabla \cdot (\u \eta)   - \mu_1 \Delta \eta - \nabla \cdot (\eta \boldsymbol{\sigma}) &= 0, \\
	\boldsymbol{\sigma}_t + \nabla (\u \cdot \boldsymbol{\sigma}) - \mu_2 \Delta \sigmaa+ \boldsymbol{\sigma} - \eta \nabla \eta &= 0, \\
	\u_t + (\u \cdot \nabla)\u - \mu_3 \Delta u + \nabla p + \eta \nabla \eta + (\nabla \cdot \boldsymbol{\sigma}) \boldsymbol{\sigma} &= 0, \\
	\nabla \cdot \u &= 0.
\end{align}
\end{subequations}
where we have used the relation $(\nabla \cdot \boldsymbol{\sigma}) \boldsymbol{\sigma} = (\Delta c)\nabla c$.

The associated initial and boundary conditions for system (\ref{biosav-rremodel}) are prescribed as
\begin{align}
	\frac{\partial \eta}{\partial \nuu} = 0, 
	\quad \boldsymbol{\sigma} \cdot \nuu = 0,
	\quad 
\nabla \times \sigmaa \times \nuu|_{tang} =0 ,
	\quad \u = 0, 
	\quad \text{on } \partial \Omega \times [0, T], \\
	\eta(x,0) = \eta_0, \quad 
	\boldsymbol{\sigma}(x,0) = \nabla c_0=\sigmaa_0, \quad 
	\u(x,0) = \u_0, 
	\quad \text{in } \Omega.
\end{align}

\begin{remark}
	Since $\nabla\times\sigmaa = \nabla\times(\nabla c)=0$, the boundary condition 
	$(\nabla\times\sigmaa)\times\nuu\big|_{\mathrm{tang}}=0$ is automatically satisfied. 
	In other words, the vanishing of the curl of a gradient field ensures that the tangential component of $(\nabla\times\sigmaa)\times\nuu$ on the boundary is identically zero.
\end{remark}

This is the principal model considered in this paper. The chemo--repulsion--Navier--Stokes (CRNS) system provides a fundamental mathematical framework for describing bioconvective phenomena and chemically mediated flow patterns in active suspensions. Moreover, it serves as a prototypical model for the analysis of coupled reaction--diffusion--convection systems that arise in various biological and fluid--mechanical contexts.

\begin{remark}
	Once the transformed system~(\ref{biosav-rremodel}) is solved, the scalar field $c$ can be recovered from $\boldsymbol{\sigma}$ by solving
	\begin{align*}
		\begin{cases}
			c_t + \u \cdot \boldsymbol{\sigma} - \mu_2 \nabla \cdot \boldsymbol{\sigma} + c = \dfrac{\eta^2}{2}, & \text{in}\, \Omega, \\[6pt]
			c(x,0) = c_0, &  \text{in} \,\Omega.
		\end{cases}
	\end{align*}
\end{remark}

The weak  formulation of the system~(\ref{biosav-rremodel}) can be stated as follows:  
Find $(\eta, \boldsymbol{\sigma}, \u, p) \in Q \times \X \times \Y \times M$ such that for all $(r, \w, \v, q) \in Q \times \X \times \Y \times M$, the following relations hold:
\begin{align}
	(\eta_t, r) - (\u  \eta,\nabla r)+ \mu_1 (\nabla \eta, \nabla r) + (\eta \boldsymbol{\sigma}, \nabla r) &= 0, \label{biogu-23}\\
	(\boldsymbol{\sigma}_t, \w) - (\u \cdot  \boldsymbol{\sigma}, \nabla \cdot\w) + \mu_2 (\nabla\cdot \boldsymbol{\sigma}, \nabla \cdot  \w)+ \mu_2 (\nabla \times \boldsymbol{\sigma}, \nabla \times  \w)
	+ (\boldsymbol{\sigma}, \w) - (\eta \nabla \eta, \w) &= 0, \label{biogu-24}\\
	(\u_t, \v) + B(\u, \u, \v) + \mu_3 (\nabla \u, \nabla \v) - (p, \nabla \cdot \v)
	+ (\eta \nabla \eta, \v) + ((\nabla \cdot \boldsymbol{\sigma}) \boldsymbol{\sigma}, \v) &= 0, \label{biogu-25}\\
	(\nabla \cdot \u, q) &= 0.\label{biogu-26}
\end{align}

The transformed system (\ref{biosav-rremodel}) satisfies the following energy law:
\begin{align}\label{biogu-30}
\frac{d}{dt} \mathcal{E}^2(\u, \sigmaa, \eta) = -\mu_1 \|\nabla \eta\|_{L^2}^2 - \mu_2 \|\nabla \cdot \sigmaa\|_{L^2}^2 - \mu_2 \|\nabla \times\sigmaa\|_{L^2}^2- \mu_3 \|\nabla \u\|_{L^2}^2 - \|\sigmaa\|_{L^2}^2 \leq 0,
\end{align}
where we define
\begin{align}
	\mathcal{E}^2(\u, \sigmaa, \eta) = \frac{1}{2} \|\u\|_{L^2}^2 + \frac{1}{2} \|\sigmaa\|_{L^2}^2 + \frac{1}{2} \|\eta\|_{L^2}^2.
\end{align}

\begin{proof}
	Taking $r = \eta$ in (\ref{biogu-23}), we have
\begin{align}
	\frac12 \frac{d}{dt}\|\eta\|_{L^2}^{2} - (\u \eta, \nabla \eta) + \mu_{1}\|\nabla \eta\|_{L^2}^{2} + (\eta \sigmaa, \nabla \eta) = 0.\label{biogu-28}
\end{align}

	Next, we choose $\w = \sigmaa$ in (\ref{biogu-24}) and get
\begin{align}
	\frac12 \frac{d}{dt}\|\sigmaa\|_{L^2}^{2} - (\u \cdot \sigmaa, \nabla \cdot \sigmaa)
+ \mu_{2}\|\nabla \cdot \sigmaa\|_{L^2}^{2}+ \mu_{2}\|\nabla \times \sigmaa\|_{L^2}^{2}  + \|\sigmaa\|_{L^2}^{2} - (\eta\sigmaa, \nabla \eta) = 0.
\end{align}

	Setting $\v = \u$ in (\ref{biogu-25}) and $q = p$ in (\ref{biogu-26}), noticing that (\ref{biogu-trilinear1}) and adding them together, we arrive at
\begin{align}
	\frac12 \frac{d}{dt}\|\u\|_{L^2}^{2} + \mu_3 \|\nabla \u\|_{L^2}^{2}
+ (\u\eta, \nabla \eta) + (\u \cdot \sigmaa, \nabla \cdot \sigmaa) = 0.\label{biogu-29}
\end{align}
	
	Combining with (\ref{biogu-28})–(\ref{biogu-29}) together, the desired result (\ref{biogu-30}) can be obtained. 
\end{proof}

\section{Gauge--Uzawa finite element method}
We consider a uniform temporal partition of the time interval $[0, T]$, denoted by  
$0 = t_0 < t_1 < \cdots < t_N = T$, where $t_n = n\tau$ $(n = 0, 1, \ldots, N)$,  
$\tau = \dfrac{T}{N}$ is the uniform time-step size, and $N \in \mathbb{N}$ denotes the total number of time steps.  
For any time-dependent function \(\xi(x,t)\), denote \(\xi^n := \xi(x, t_n)\),
given a sequence of functions $\{\xi^n\}_{n=0}^N$, the backward difference operator $D_\tau$ is defined by  
\begin{equation}
	D_\tau \xi^n = \frac{\xi^n - \xi^{n-1}}{\tau}, \quad \bar{\xi}^n=2\xi^n-\xi^{n-1}.
	\label{eq:backward_difference}
\end{equation}

Let $\mathcal{T}_h$ be a family of quasi-uniform triangulations of the domain $\Omega$. 
The ordered elements are denoted by $\mathcal{K}_1, \mathcal{K}_2, \ldots, \mathcal{K}_{m_1}$, where $h_i = \operatorname{diam}(\mathcal{K}_i)$ for $i=1,2,\ldots,{m_1}$, and we define $h = \max\{h_1,h_2,\ldots,h_{m_2}\}$. 
For each $\mathcal{K} \in \mathcal{T}_h$, let $\mathbb{P}_r(\mathcal{K})$ denote the space of all polynomials on $\mathcal{K}$ of degree at most $r$. 

We employ the mini element (P1b--P1) to approximate the velocity and pressure, and the piecewise linear (P1) element to approximate cell density and the concentration. 
The corresponding finite element spaces are defined as follows \cite{scott2008}:
\begin{align*}
	Q_h &= \{ \psi_h \in L^2(\Omega) : \forall \mathcal{K} \in \mathcal{T}_h,\ \psi_h|_\mathcal{K} \in \mathbb{P}_1(\mathcal{K}) \}, \\
	\textbf{X}_h &= \{ \sigmaa_h \in [L^2(\Omega)]^2 : \forall \mathcal{K} \in \mathcal{T}_h,\ \sigmaa_h|_\mathcal{K} \in [\mathbb{P}_1(\mathcal{K})]^2 \}, \\
	\textbf{Y}_h &= \{ \v_h \in [L^2(\Omega)]^2 : \forall \mathcal{K} \in \mathcal{T}_h,\ \v_h|_\mathcal{K} \in [\mathbb{P}_1(\mathcal{K})\oplus b(\mathcal{K}]^2 \}, \\
	M_h &= \{ q_h \in L^2_0(\Omega) : \forall \mathcal{K} \in \mathcal{T}_h,\ q_h|_\mathcal{K} \in \mathbb{P}_{1}(\mathcal{K}) \}.\\
\end{align*}
where $b(\mathcal{K})$ denotes the bubble function associated with each element $\mathcal{K} \in \mathcal{T}_h$.
Moreover, the spaces $\textbf{Y}_h$ and $M_h$ satisfy the inf--sup condition \cite{babuvska1973, brezzi1974}:
\begin{align*}
	\beta \|q_h \|_{L^2} 
	\leq  \sup_{\mathbf{v}_h \in \Y_h}  
	\frac{(\nabla \cdot \mathbf{v}_h,\, q_h)}{\|\nabla \mathbf{v}_h\|_{L^2}}, 
	\quad \forall q_h \in M_h.
\end{align*}

The inverse inequality will be used frequently \cite{boffi2013}.
\begin{subequations}\label{inverse-inequality}
	\begin{align}
		\| \u_h \|_{W^{m,q}} \leq C h^{  l-m + n( \frac{1}{q} -\frac{1}{p}  )} \| \u_h\| _{W^{l,p}}, \quad \forall~ \u_h \in \textbf{X}_h \, \text{or} \,\textbf{Y}_h , \\
		\| c_h \|_{W^{m,q}} \leq C h^{  l-m + n( \frac{1}{q} -\frac{1}{p}  )} \| c_h\| _{W^{l,p}}, \quad \forall~ c_h \in  Q_h.
	\end{align}
\end{subequations}

We set the initial iterates as
$
\u_h^0 = I_{1h}\u^0,
p_h^0 = I_{2h}p^0, 
\eta_h^0 = I_{3h}\eta^0,  
\sigmaa_h^0 = I_{4h}\sigmaa^0,
$
where \( I_{1h}, I_{2h}, I_{3h} \), and \( I_{4h} \) denote the interpolation operators associated with the finite element spaces 
$ \Y_h, M_h, Q_h$ and $\X_h$, respectively. 
Based on the interpolation properties, we then obtain the following approximation error estimates:
\begin{subequations}\label{biogu-33}
\begin{align}
	\| \u^0 - I_{1h}\u^0 \|_{L^2} + h \| \nabla (\u^0- I_{1h}\u^0)\|_{L^2} \leq C h^2 \| \u^0\|_{H^2},\\
	\| p^0 - I_{2h}p^0 \|_{L^2}\leq C h \| p^0\|_{H^1},\\
	\| \eta^0 - I_{3h}{\eta}^0\|_{L^2} +h\| \nabla ( \eta^0 - I_{3h}\eta^0)\|_{L^2}\leq C h^2 \| \eta^0\|_{H^2},\\
	\| \sigmaa^0 - I_{4h}\sigmaa^0 \|_{L^2}+h\| \nabla ( \sigmaa^0 - I_{4h}\sigmaa^0)\|_{L^2} \leq C h^2 \| \sigmaa^0\|_{H^2}.
\end{align}
\end{subequations}

Now, we will present the First-order and Second-order Gauge-Uzawa Finite element scheme for  the
chemo-repulsion-Navier-Stokes system.

\textbf{First-order Gauge--Uzawa algorithm for CRNS system:}

Start with $s^0_h=0$ and $( \eta^0_h,\sigmaa^0_h, \u^0_h)$  as the solutions of
\begin{align}
	(\eta^0_h,r_h)=(\eta^0,r_h), \quad (\sigma^0_h,\w_h)=(\sigma^0,\w_h), \quad (\u^0_h,\v_h)=(\u^0,\v_h), \\
	 \quad \forall (r_h,\w_h,\v_h) \in Q_h \times  \X_h \times \Y_h. \notag
\end{align}

\textbf{Step 1:} Find $(\eta_h^{n+1},\sigma^{n+1}_h,\hat{\u}^{n+1}_h) \in  Q_h \times  \X_h \times \Y_h$ as the solution of 
\begin{align}\label{biogu-1}
(D_{\tau}\eta^{n+1}_h ,r_h)-(\hat{\u}_h^{n+1}  \eta_h^{n}, \nabla r_h) +\mu_1 (\nabla \eta_h^{n+1},\nabla r_h)+(\sigmaa_h^{n+1} \eta_h^{n}, \nabla r_h)=0,
\end{align}
\begin{align}\label{biogu-2}
(D_{\tau} \sigmaa^{n+1}_h ,\w_h )-(\hat{\u}_h^{n+1} \cdot \sigmaa_h^{n}, \nabla \cdot \w_h) + \mu_2 ( \nabla \cdot \sigmaa_h^{n+1},\nabla \cdot \w_h)\\+ \mu_2 ( \nabla \times \sigmaa_h^{n+1},\nabla \times \w_h)+(\sigmaa_h^{n+1},\w_h)
-(\eta_h^{n} \nabla \eta_h^{n+1},\w_h)=0,\notag
\end{align}
\begin{align}\label{biogu-3}
	(\frac{\hat{\u}_h^{n+1}-\u^n_h}{\tau},\v_h)+b(\u_h^n,  \hat{\u}_h^{n+1},\v_h) +\mu_3(\nabla \hat{\u}_h^{n+1},\nabla \v_h) -\mu_3(s^n_h,\nabla \cdot \v_h)\\
	+ (\eta^{n}_h \nabla \eta_h^{n+1},\v_h)+(\sigmaa_h^{n} (\nabla \cdot \sigmaa_h^{n+1}),\v_h)=0 \notag
\end{align}

for all $(r_h,\w_h,\v_h) \in Q_h \times  \X_h \times \Y_h. $

\textbf{Step 2:}  Find $\rho^{n+1}_h \in M_h$ as the solution of 
\begin{align}\label{biogu-4}
	(\nabla \rho_h^{n+1},\nabla \psi_h)=(\nabla \cdot \hat{\u}_h^{n+1},\psi_h), \quad \forall \, \psi_h \in M_h.
\end{align}

\textbf{Step 3:} Update $s^{n+1}_h \in M_h$ according to 
\begin{align}\label{biogu-5}
	(s^{n+1}_h,q_h)=(s^n_h,q_h)-(\nabla \cdot \hat{\u}_h^{n+1},q_h), \quad \forall \, q_h \in M_h.
\end{align}

\textbf{Step 4:} Update $\u^{n+1}_h$$\in \textbf{Y}_h \oplus M_h $ according to   
\begin{align}\label{biogu-6}
\u_h^{n+1}=\hat{\u}_h^{n+1}+\nabla \rho_h^{n+1}
\end{align}

\textbf{Step 5:} Computed $p^{n+1}_h \in M_h$ via 
\begin{align}\label{biogu-7}
p_h^{n+1}=\mu_3 s_h^{n+1} -\frac{1}{\tau} \rho_h^{n+1}
\end{align}

\begin{remark}
	We note that $\mathbf{u}_h^{n+1}$ is a discontinuous function across interelement boundaries and that, in light of (\ref{biogu-4}) and (\ref{biogu-6}), $\mathbf{u}_h^{n+1}$ is discrete divergence-free in the sense that \cite{pyo2013}
	\begin{align}\label{biogu-divergencefree}
	(\mathbf{u}_h^{n+1}, \nabla \psi_h) = 0, \quad \forall \psi_h \in M_h.
	\end{align}
	
	The First-order Gauge--Uzawa algorithm employs a first-order back Euler difference formula for the temporal derivative and a semi-implicit treatment for the convection term, which effectively eliminates the need for any artificial boundary conditions on the pressure. 
	
	Once the first-order Gauge-Uzawa algorithm (\ref{biogu-1})--(\ref{biogu-7}) is solved, we can recover $c^{n+1}_h$ from
	\begin{align}
			\frac{c_h^{n+1}-c_h^n}{\tau}+\u_h^{n}\cdot \sigmaa_h^{n+1}-\mu_2 \nabla \cdot \sigmaa_h^{n+1} +c_h^{n+1}-\frac{ (\eta^n_h)^2}{2}=0,
	\end{align}
	which is equivalent to 
	\begin{align}
			\frac{c_h^{n+1}-c_h^n}{\tau}+\u_h^{n}\cdot \nabla c_h^{n+1}-\mu_2 \Delta c_h^{n+1}+c_h^{n+1}-\frac{(\eta^n_h)^2}{2}=0.
	\end{align}
\end{remark}

\textbf{Second-order Gauge--Uzawa algorithm for CRNS system:}


Compute $\eta^1_h,\sigmaa^1_h,\u^1_h,p^1_h$ via first-order scheme (\ref{biogu-1})--(\ref{biogu-7}) and set $\rho_h^1=\frac{-2\tau}{3}p_h^1$ and $s_h^1=0$.

\textbf{Step 1:} Find $(\eta^{n+1}_h,\sigmaa^{n+1}_h,\hat{\u}^{n+1}_h) \in  Q_h \times  \X_h \times \Y_h$ as the solutions of 
\begin{align}\label{biogubdf-1}
	(\frac{3\eta^{n+1}_h-4\eta^{n}_h+\eta_h^{n-1}}{2\tau},r_h)-(\hat{\u}_h^{n+1}  \bar{\eta}_h^{n}, \nabla r_h) +\mu_1 (\nabla \eta_h^{n+1},\nabla r_h)+(\sigmaa_h^{n+1} \bar{\eta}_h^{n}, \nabla r_h)=0,
\end{align}
\begin{align}\label{biogubdf-2}
	(\frac{3\sigmaa^{n+1}_h-4\sigmaa^{n}_h+\sigmaa_h^{n-1}}{2\tau},\w_h )-(\hat{\u}_h^{n+1} \cdot \bar{\sigmaa}_h^{n}, \nabla \cdot \w_h) + \mu_2 ( \nabla \cdot \sigmaa_h^{n+1},\nabla \cdot \w_h)\\+\mu_2 ( \nabla \times \sigmaa_h^{n+1},\nabla \times \w_h)+(\sigmaa_h^{n+1},\w_h)
	-(\bar{\eta}_h^{n} \nabla \eta^{n+1}_h,\w_h)=0,\notag
\end{align}
\begin{align}\label{biogubdf-3}
	(\frac{3\hat{\u}_h^{n+1}-4\u^n_h+\u_h^{n-1}}{2\tau},\v)+b(\bar{\u}_h^n,  \hat{\u}_h^{n+1},\v_h) +\mu_3(\nabla \hat{\u}_h^{n+1},\nabla \v_h) -(p^n_h,\nabla \cdot \v_h)\\
	+ (\bar{\eta}_h^{n} \nabla \eta^{n+1}_h,\v)+(\bar{\sigmaa}_h^{n} (\nabla \cdot \sigmaa_h^{n+1}),\v_h)=0 \notag,
\end{align}

for all $(r_h,\w_h,\v_h) \in Q_h \times  \X_h \times \Y_h. $

\textbf{Step 2:}  Find $\rho^{n+1}_h \in M_h$ as the solution of 
\begin{align}\label{biogubdf-4}
	(\nabla \rho_h^{n+1},\nabla \psi_h)=(\nabla \rho^{n}_h,\nabla \psi_h)+(\nabla \cdot \hat{\u}_h^{n+1},\psi_h), \quad \forall \, \psi_h \in M_h.
\end{align}

\textbf{Step 3:} Update $s^{n+1}_h \in M_h$ according to 
\begin{align}\label{biogubdf-5}
	(s^{n+1}_h,q_h)=(s_h^n,q_h)-(\nabla \cdot \hat{\u}_h^{n+1},q_h), \quad \forall \, q_h \in M_h.
\end{align}

\textbf{Step 4:} Update $\u^{n+1}_h\in \textbf{Y}_h \oplus M_h $ according to   
\begin{align}\label{biogubdf-6}
	\u^{n+1}_h=\hat{\u}_h^{n+1}+\nabla (\rho_h^{n+1}-\rho^{n}_h)
\end{align}

\textbf{Step 5:} Computed $p_h^{n+1} \in M_h$ via 
\begin{align}\label{biogubdf-7}
	p^{n+1}_h=\mu_3 s_h^{n+1} -\frac{3}{2\tau} \rho_h^{n+1}
\end{align}

\begin{remark}
The second-order Gauge–Uzawa algorithm utilizes a second-order backward difference formula for the temporal derivative and adopts a semi-implicit treatment for the convection term, which effectively removes the necessity of imposing any artificial boundary conditions on the pressure. We restrict our analysis to the stability, unique solvability, and convergence of the first-order Gauge–Uzawa algorithm, because the corresponding arguments for the second-order scheme can be derived in a similar manner.
\end{remark}

 \subsection{Unconditional stability}
 Now we will prove  the First-order  Gauge-Uzawa algorithm is unconditionally energy stable by the following theorem.
 \begin{theorem}
 	The discrete scheme (\ref{biogu-1})--(\ref{biogu-7}) is an unconditionally energy stable and satisfy the following energy decaying property:
 	\begin{align}\label{biogu-20}
 		\mathcal{E}^3(\eta^{n+1}_h, \sigmaa^{n+1}_h,\u^{n+1}_h,s^{n+1}_h) \leq \mathcal{E}^3(\eta^{n}_h, \sigmaa^{n}_h,\u^{n}_h.s^n_h)
 	\end{align}
 	for all $0\leq n\leq N-1$, where the energy functional is defined by 
 	\begin{align}\label{biogu-energy}
 		\mathcal{E}^3(\eta^{n}_h, \sigmaa^{n}_h,\u^{n}_h,s^n_h)=\| \eta^{n}_h\|^2_{L^2}+\| \sigmaa^{n}_h\|^2_{L^2}+\| \u^{n}_h\|^2_{L^2}+\|s^n_h\|^2_{L^2}.
 	\end{align}
 \end{theorem}
 
\begin{proof}
Choosing $ r_h=2\tau \eta_h^{n+1}$ in (\ref{biogu-1}), it follows that
\begin{align}\label{biogu-8}
	\| \eta^{n+1}_h\|^2_{L^2}-\| \eta^{n}_h\|^2_{L^2}+\| \eta^{n+1}_h-\eta^{n}_h\|^2_{L^2}- 2\tau (\hat{\u}^{n+1}_h \eta^n_h, \nabla \eta^{n+1}_h)\notag\\
	+ 2 \mu_1 \tau \| \nabla \eta^{n+1}_h\|^2_{L^2} + 2\tau (\sigmaa_h^{n+1}\eta^n_h,\nabla \eta^{n+1}_h)=0.
\end{align}

Setting $\w_h=2\tau\sigmaa^{n+1}_h$ in (\ref{biogu-2}), we deduce that 
\begin{align}\label{biogu-9}
		\| \sigmaa^{n+1}_h\|^2_{L^2}-\| \sigmaa^{n}_h\|^2_{L^2}+\| \sigmaa^{n+1}_h-\sigmaa^{n}_h\|^2_{L^2}-2\tau(\hat{\u}^{n+1}_h\cdot \sigmaa^n_h,\nabla \cdot \sigmaa^{n+1}_h)\notag\\
		+2 \mu_2 \tau \| \nabla \cdot \sigmaa^{n+1}_h\|^2_{L^2}+2 \mu_2 \tau \| \nabla \times \sigmaa^{n+1}_h\|^2_{L^2}
		 + 2\tau \| \sigmaa^{n+1}_h\|^2_{L^2}-2\tau(\eta^n_h\nabla \eta^{n+1}_h,\sigmaa^{n+1}_h)=0.
\end{align}

Letting $\v_h=2\tau\hat{\u}^{n+1}_h$ in (\ref{biogu-3}), it yields that 
\begin{align}
	2(\hat{\u}^{n+1}_h-\u^n_h,\hat{\u}^{n+1}_h) + 2 \tau b(\u_h^n,  \hat{\u}_h^{n+1},\hat{\u}^{n+1}_h) + 2 \mu_3 \tau \| \nabla \hat{\u}^{n+1}_h\|^2_{L^2}-2\mu_3\tau (s^n_h,\nabla \cdot \hat{\u}^{n+1}_h)\notag\\
	+ 2\tau ( \eta^n_h \nabla \eta^{n+1}_h,\hat{\u}^{n+1}_h)+2\tau(\sigmaa^n_h (\nabla \cdot \sigmaa^{n+1}_h),\hat{\u}^{n+1}_h)=0.
\end{align}

Thanks to (\ref{biogu-6}) and (\ref{biogu-divergencefree}), we derive that
\begin{align}
	(\hat{\u}_h^{n+1} - \u_h^n, \hat{\u}_h^{n+1}) 
	&= (\u_h^{n+1} - \nabla \rho_h^{n+1} - \u_h^n, \u_h^{n+1} - \nabla \rho_h^{n+1}) \notag\\
	&= (\u_h^{n+1} - \u_h^n, \u_h^{n+1}) + (\nabla \rho_h^{n+1}, \nabla \rho_h^{n+1}).
\end{align}

Therefore, there holds that 
\begin{align}\label{biogu-10}
		\| \u^{n+1}_h\|^2_{L^2}-\| \u^{n}_h\|^2_{L^2}+\| \u^{n+1}_h-\u^{n}_h\|^2_{L^2}+ 2 \| \nabla\rho^{n+1}_h\|^2_{L^2}+ 2 \tau b(\u_h^n,  \hat{\u}_h^{n+1},\hat{\u}^{n+1}_h)\notag\\
		 + 2 \mu_3 \tau \| \nabla \hat{\u}^{n+1}_h\|^2_{L^2}
		+ 2\tau ( \eta^n_h \nabla \eta^{n+1}_h,\hat{\u}^{n+1}_h)+2\tau(\sigmaa^n_h (\nabla \cdot \sigmaa^{n+1}_h),\hat{\u}^{n+1}_h)=2\mu_3\tau (s^n_h,\nabla \cdot \hat{\u}^{n+1}_h).
\end{align}
 
 Using (\ref{biogu-5}), we deduce that 
 \begin{align}\label{biogu-11}
 	2(s_h^n, \nabla \cdot\hat{\u}_h^{n+1}) &= 2(s_h^n - s_h^{n+1}, s_h^n)  \notag\\
 	&= \|s_h^n\|_{L^2}^2 - \|s_h^{n+1}\|_{L^2}^2 + \|s_h^n - s_h^{n+1}\|_{L^2}^2.
 \end{align}
 
  Combining (\ref{biogu-5}) with Div-grad inequality, it infers that
  \begin{align}
  	\|s_h^n - s_h^{n+1}\|_{L^2}^2 \leq \| \nabla \cdot \hat{\u}^{n+1}_h\|^2_{L^2} \leq \| \nabla \hat{\u}^{n+1}_h\|^2_{L^2}.
  	  \end{align}
  
 Taking sum of (\ref{biogu-8}), (\ref{biogu-9}) and (\ref{biogu-10}) and noticing that (\ref{biogu-trilinear1}), it yields that
 \begin{align}\label{biogu-12}
 	&	\| \eta^{n+1}_h\|^2_{L^2}-\| \eta^{n}_h\|^2_{L^2}+\| \eta^{n+1}_h-\eta^{n}_h\|^2_{L^2}+2 \mu_1 \tau \| \nabla \eta^{n+1}_h\|^2_{L^2} \\
 		&+	\| \sigmaa^{n+1}_h\|^2_{L^2}-\| \sigmaa^{n}_h\|^2_{L^2}+\| \sigmaa^{n+1}_h-\sigmaa^{n}_h\|^2_{L^2}+2 \mu_2 \tau \| \nabla \cdot  \sigmaa^{n+1}_h\|^2_{L^2}\notag \\
 		&+2 \mu_2 \tau \| \nabla \times  \sigmaa^{n+1}_h\|^2_{L^2} + 2\tau \| \sigmaa^{n+1}_h\|^2_{L^2}+\| \u^{n+1}_h\|^2_{L^2}-\| \u^{n}_h\|^2_{L^2} \notag \\
 		&
 		+\| \u^{n+1}_h-\u^{n}_h\|^2_{L^2}+ 2 \| \nabla\rho^{n+1}_h\|^2_{L^2}+ \mu_3 \tau \| \nabla \hat{\u}^{n+1}_h\|^2_{L^2}
 		+\|s_h^{n+1}\|_0^2  - \|s_h^n\|_0^2 \leq 0.\notag
 \end{align}
 from which (\ref{biogu-20}) follows immediately.
\end{proof}

 \subsection{Existence and uniqueness of the First-order Gauge-Uzawa algorithm}
 
 In the following, we prove the existence and uniqueness of the numerical solutions for the First-order Gauge-Uzawa algorithm (\ref{biogu-1})--(\ref{biogu-7}) by the following theorem.
 \begin{theorem}
 	Let $\eta,\sigmaa, \u$ and $p$ are the solutions of the CRNS system (\ref{biosav-rremodel}), then for  $0 \leq n \leq N-1$, the finite element discrete scheme (\ref{biogu-1})--(\ref{biogu-7}) admit the unique solution $(\eta^{n+1}_h,\sigmaa^{n+1}_h,\u^{n+1}_h,p^{n+1}_h) $ without any condition on the time step $\tau$ and mesh size $h$.
 \end{theorem}
 \begin{proof}
 	The unique solvability of the discrete scheme is established by show that its corresponding homogeneous linear system only has the trival solution, similar techniques has been applied in \cite{pan2004,libuyang2022,libuyang2024,lichenyangdensity2025}. If $(\eta^n_h,\sigmaa^n_h,\u^n_h,p^n_h)$ are given for $k=0,1,2,\dots,n$, then the algorithm (\ref{biogu-1})--(\ref{biogu-7}) has a unique solution if and only if the corresponding homogeneous linear equation
 	\begin{align}
 		\frac{1}{\tau} ( A,r_h)-(\textbf{C} \eta_h^{n}, \nabla r_h) +\mu_1 (\nabla A,\nabla r_h)+(\textbf{B} \eta_h^{n}, \nabla r_h)=0,\\
 			\frac{1}{\tau} ( \textbf{B},\w_h)-(\textbf{C}\cdot \sigmaa_h^{n}, \nabla \cdot \w_h) 
 			+ \mu_2 ( \nabla \cdot \textbf{B},\nabla \cdot  \w_h)\\
 			+\mu_2 ( \nabla \times \textbf{B},\nabla \times  \w_h)+(\textbf{B},\w_h)-(\eta_h^{n} \nabla A,\w_h)=0,\notag \\
 				\frac{1}{\tau} ( \textbf{C},\v_h)+ 2 \tau b(\u_h^n,  \textbf{C},\v_h)+ \mu_3(\nabla \textbf{C},\nabla \v_h)
 				+ (\eta^{n}_h \nabla A,\v_h)+(\sigmaa_h^{n} (\nabla \cdot \textbf{B}),\v_h)=0,
 	\end{align}
 only has  zero solutions $(A,\textbf{B},\textbf{C})=(0,\textbf{0},\textbf{0})$. Setting $(r_h,\w_h,\v_h)=(A,\textbf{B},\textbf{C})$, we can have
  	\begin{align}
 	\frac{1}{\tau} \|A\|_{L^2}^2-(\textbf{C} \eta_h^{n}, \nabla A) +\mu_1  \| \nabla A\|^2_{L^2}+(\textbf{B} \eta_h^{n}, \nabla A)=0,\\
 	(\frac{1}{\tau} +1)\|\textbf{B}\|_{L^2}^2-(\textbf{C}\cdot \sigmaa_h^{n}, \nabla \cdot \textbf{B}) + \mu_2 \| \nabla \times \textbf{B}\|^2_{L^2}+\mu_2 \| \nabla \cdot \textbf{B}\|^2_{L^2}-(\eta_h^{n} \nabla A,\textbf{B})=0,\\
 	\frac{1}{\tau} \|\textbf{C}\|_{L^2}^2 + \mu_3 \| \nabla \textbf{C}\|^2_{L^2}
 	+ (\eta^{n}_h \nabla A,\textbf{C})+(\sigmaa_h^{n} (\nabla \cdot \textbf{B}),\textbf{C})=0,
 \end{align} 
 
Summing up the above equations, we can obtain 
\begin{align}
	\frac{1}{\tau} (\|A\|_{L^2}^2+\|\textbf{C}\|_{L^2}^2 ) +(\frac{1}{\tau} +1)\|\textbf{B}\|_{L^2}^2\\
	+ \mu_1  \| \nabla A\|^2_{L^2}+	\mu_2 \| \nabla \cdot  \textbf{B}\|^2_{L^2}+\mu_2 \| \nabla \times \textbf{B}\|^2_{L^2}+\mu_3 \| \nabla \textbf{C}\|^2_{L^2}=0.\notag
\end{align}
which implies that $(A,\mathbf{B},\mathbf{C}) = (0,\mathbf{0},\mathbf{0})$. Consequently, the discrete problem admits a unique solution $(\eta^{n+1}_h, \sigma^{n+1}_h, \hat{\mathbf{u}}^{n+1}_h)$. Furthermore, the unique pair $(\mathbf{u}_h^{n+1}, p_h^{n+1})$ can be subsequently determined through \textbf{Steps 2--5}. 
\end{proof}

 \section{Convergence analysis}
  In this section, we present the main theoretical results of this paper.
 
 
At first, we formalize the definition of the Stokes projection $(P_1\u,P_2 p) : (\textbf{Y},M) \rightarrow (\textbf{Y}_h,M_h)$ by 
 \begin{align}
 	\mu_3 ( \nabla (P_1\u-\u),\nabla\v_h  )-(P_2 p-p,\nabla \cdot \v_h)=0, \quad \forall \, \v_h \in \textbf{Y}_n,\\
 	( \nabla \cdot (  P_1\u-\u) , q_h)=0,\quad \forall \, q_h\in M_h.
 \end{align}
 
The Ritz projection $P_3 : Q \rightarrow Q_h$ is defined by 
\begin{align}\label{biocnlf-projection-definition2}
	(   \nabla (\eta -P_3 \eta) , \nabla r_h  ) =0, \quad \forall r_h \in Q_h.
\end{align}

Fortin operator $P_4: \textbf{X}\rightarrow \textbf{X}_h$ is defined by \cite{fortin2000}
\begin{align}
	(   \nabla \times (\sigmaa - P_4 \sigmaa) , \nabla \times \w_h  ) +	(   \nabla \cdot (\sigmaa - P_4 \sigmaa) , \nabla \cdot \w_h  ) =0, \quad \forall \w_h \in \textbf{X}_h.
\end{align}
and there holds \cite{scott2008,heyinnian2015,ravindran2016,girault2012}
\begin{subequations}\label{biogu-projection-error}
\begin{align}
	\|\u-P_1\u \|_{L^2}+h\|p-P_2 p\|_{L^2} \leq& Ch^2 (\|\u\|_{H^2}+\|p\|_{H^1})\\
	\| \eta -P_3 \eta\|_{L^2} + h \| \nabla(\eta-P_3 \eta )\| \leq& C h^2 \| \eta\|_{H^2}, \\
	\| \sigmaa -P_4 \sigmaa \|_{L^2} + h \| \nabla(\sigmaa-P_4 \sigmaa )\| \leq& C h^2 \| \sigmaa\|_{H^2}.
\end{align}
\end{subequations}

The exact solution at time $t=t_{n+1}$ satisfy
 \begin{align}
 	D_{\tau}\eta^{n+1}+ \nabla \cdot (\u^{n+1} \eta^n)-\mu_1 \Delta \eta^{n+1}-\nabla \cdot ( \eta^n \sigmaa^{n+1})=R^{n+1}_\eta,\label{biogu-timediscrete1}\\
 	D_{\tau}\sigmaa^{n+1} + \nabla (\u^{n+1} \sigmaa^n) - \mu_2 \nabla(\nabla \cdot  \sigmaa^{n+1})+\mu_2\nabla \times \nabla \times \sigmaa^{n+1}+\sigmaa^{n+1}-\eta^n\nabla \eta^{n+1}=R^{n+1}_\sigmaa,\\
 	D_{\tau}\u^{n+1} + \u^n\cdot \nabla  \u^{n+1} -\mu_3 \Delta \u^{n+1} + \eta^n \nabla \eta^{n+1}+ \sigmaa^n(\nabla \cdot \sigmaa^{n+1})+\nabla p^{n+1}=R^{n+1}_\u\label{biogu-timediscrete2},\\
 	\nabla \cdot \u^{n+1}=0.
 \end{align}
 for $n=0,1,\dots,N-1$. 
 where
 \begin{align}
 	R^{n+1}_\eta =D_\tau \eta^{n+1}- 	\eta_t(t_{n+1})+\nabla \cdot (\u^{n+1} \eta^n)-\nabla \cdot (\u^{n+1} \eta^{n+1})\\
 	-\nabla \cdot ( \eta^n \sigmaa^{n+1})+\nabla \cdot ( \eta^{n+1} \sigmaa^{n+1})\notag,\\
 	R^{n+1}_{\sigmaa}=D_\tau 	\sigmaa^{n+1}-\sigmaa_t(t_{n+1})+ \nabla (\u^{n+1} \sigmaa^n)- \nabla (\u^{n+1} \sigmaa^{n+1})-\eta^n \nabla \eta^{n+1}+\eta^{n+1}\nabla \eta^{n+1},\\
 	R^{n+1}_{\u}=D_\tau \u^{n+1}-\u_t(t_{n+1})+\u^n\cdot \nabla \u^{n+1} -\u^{n+1}\cdot  \nabla \u^{n+1} +\eta^n \nabla \eta^{n+1}-\eta^{n+1} \nabla \eta^{n+1}\\
 	+\sigmaa^n (\nabla \cdot \sigmaa^{n+1})-\sigmaa^{n+1}(\nabla \cdot \sigmaa^{n+1}).\notag
 \end{align}
 
  After a straightforward manipulation, we can obtain the weak form of (\ref{biogu-timediscrete1})--(\ref{biogu-timediscrete2}) as 
 \begin{align}
 	(D_{\tau}\eta^{n+1}, r) -(\u^{n+1}  \eta^{n}, \nabla r) +\mu_1 (\nabla \eta^{n+1},\nabla r)+(\sigmaa^{n+1} \eta^{n}, \nabla r)=(R^{n+1}_\eta,r),\quad \forall \, r\in Q,\label{biogu-13}\\
 	( D_{\tau}\sigmaa^{n+1},\w )-(\u^{n+1} \cdot \sigmaa^{n}, \nabla \cdot \w) + \mu_2 ( \nabla  \cdot \sigmaa^{n+1},\nabla \cdot  \w)\\
 	+\mu_2 ( \nabla  \times \sigmaa^{n+1},\nabla \times  \w)+(\sigmaa^{n+1},\w)
 	-(\eta^{n} \nabla \eta^{n+1},\w)=(R^{n+1}_\sigmaa,\w),\quad \forall\, \w\in X,\notag\\
 	(D_{\tau}\u^{n+1},\v)+b(\u^n,  \u^{n+1},\v) +\mu_3(\nabla \u^{n+1},\nabla \v) 
 	+ (\eta^{n} \nabla \eta^{n+1},\v)\notag\\
 	+(\sigmaa^{n} (\nabla \cdot \sigmaa^{n+1}),\v)-(p^{n+1},\nabla \cdot \v)=(R^{n+1}_\u,\v),\quad \, \forall \, \v \in Y, \\
 	(\nabla \cdot \u^{n+1},q)=0,\quad \forall \,  q\in M.\label{biogu-14}
 \end{align}
 
In the next step, we establish an estimate for the truncation error $R^{n+1}_\eta, R^{n+1}_\sigmaa,R^{n+1}_\u$, which is stated in the following Lemma.
 \begin{lemma}
 	Under the regularity assumption (\ref{biogu-regularity}), there holds
\begin{align}\label{biogu-trunction-error}
\tau\sum_{n=1}^{N-1} \big(\| R^{n+1}_\eta\|^2_{L^2}+ \| R^{n+1}_\sigmaa\|^2_{L^2}+ \| R^{n+1}_\u\|^2_{L^2} \big)\leq C \tau^2.
\end{align}
 \end{lemma}
 
 \begin{proof}
 By using the Taylor's formula one has
\begin{align}
	D_{\tau}\u^{n+1}-\u_t(t_{n+1})=-\frac{1}{\tau}\int^{t_{n+1}}_{t_n} (t-t_n) \u_{tt} dt,
\end{align}
we have 
\begin{align}
	(R^{n+1}_\eta,r)=&-\big(   \frac{1}{\tau}\int^{t_{n+1}}_{t_n} (t-t_n) \eta_{tt} dt ,r \big)+\big(\u^{n+1} \int^{t_{n+1}}_{t_n} \eta_t dt, \nabla r \big) \\
	&-\big( \sigmaa^{n+1}   \int^{t_{n+1}}_{t_n} \eta_t dt, \nabla r\big),\notag\\
	(R^{n+1}_\sigmaa,\w)=&-\big(   \frac{1}{\tau}\int^{t_{n+1}}_{t_n} (t-t_n) \sigmaa_{tt} dt ,\w \big)
	+\big(\u^{n+1} \int^{t_{n+1}}_{t_n} \sigmaa_t dt, \nabla \w \big) \\
	&-\big( \int^{t_{n+1}}_{t_n} \eta_t dt \nabla \eta^{n+1},\w\big),\notag\\
	(R^{n+1}_\u,\v)=&-\big(   \frac{1}{\tau}\int^{t_{n+1}}_{t_n} (t-t_n) \u_{tt} dt ,\v\big)-\big( \int^{t_{n+1}}_{t_n} \u_t dt \nabla \u^{n+1},\v\big)\\
	&-\big( \int^{t_{n+1}}_{t_n} \eta_t dt \nabla \eta^{n+1},\v\big)-\big( \int^{t_{n+1}}_{t_n} \sigmaa_t dt \nabla \cdot \sigmaa^{n+1},\v\big)\notag.
\end{align}

According to the regularity assumptions (\ref{biogu-regularity}), we have
\begin{align}
	\| R^{n+1}_\u\|_{L^2} 
	\leq &\frac{C}{\tau} \int^{t_{n+1}}_{t_n} (t-t_n) \u_{tt} dt +C \|\nabla \u^{n+1}\|_{L^{\infty}}  \int^{t_{n+1}}_{t_n} \|\u_t\|_{L^2} dt\notag\\
	&+C\int^{t_{n+1}}_{t_n} \|\eta_t\|_{L^2} dt\| \nabla \eta^{n+1}\|_{L^\infty}+C\int^{t_{n+1}}_{t_n} \|\sigmaa_t\|_{L^2} dt\| \nabla \sigmaa^{n+1}\|_{L^\infty}\notag\\
	\leq& C\tau^{\frac{1}{2}} \big(\int^{t_{n+1}}_{t_n}   \| \u_{tt}\|^2_{L^{2}} +\|\u_t\|^2_{L^2}+ \|\eta_t\|^2_{L^2}+\|\sigmaa_t\|^2_{L^2}dt \big).\notag
\end{align}

Similarly
\begin{align}
	\| R^{n+1}_\eta\|_{L^2} 
	\leq &\frac{C}{\tau} \int^{t_{n+1}}_{t_n} (t-t_n) \eta_{tt} dt +C \| \u^{n+1}\|_{L^{\infty}}  \int^{t_{n+1}}_{t_n} \|\eta_t\|_{L^2} dt\notag\\
	&+C\int^{t_{n+1}}_{t_n} \|\eta_t\|_{L^2} dt\|  \sigmaa^{n+1}\|_{L^{\infty}},\notag\\
\leq& C\tau^{\frac{1}{2}} \big(\int^{t_{n+1}}_{t_n}   \| \eta_{tt}\|^2_{L^2} + \|\eta_t\|^2_{L^2}dt \big).\notag\\
	\| R^{n+1}_\sigmaa\|_{L^2} 
\leq &\frac{C}{\tau} \int^{t_{n+1}}_{t_n} (t-t_n) \sigmaa_{tt} dt +C \| \u^{n+1}\|_{L^{\infty}}  \int^{t_{n+1}}_{t_n} \|\sigmaa_t\|_{L^2} dt\notag\\
&+C\int^{t_{n+1}}_{t_n} \|\eta_t\|_{L^2} dt\|  \nabla \eta^{n+1}\|_{L^{\infty}},\notag\\
\leq& C\tau^{\frac{1}{2}} \big(\int^{t_{n+1}}_{t_n}   \| \sigmaa_{tt}\|^2_{L^2} +\|\sigmaa_t\|^2_{L^2}+ \|\eta_t\|^2_{L^2}dt \big).\notag
\end{align}
 from which (\ref{biogu-trunction-error}) follows immediately.
\end{proof}

 Denote 
 \begin{align*}
 	&\theta_\u^i= \u^i-P_1 \u^i,\quad e_\u^i=P_1\u^i-\u^i_h,\quad \hat{e}_\u^i=P_1\u^i-\hat{\u}^i_h, \\
 	&\theta_p^i= p^i-P_2 p^i,\quad e_p^i=P_2 p^i-p^i_h,\\
 	 	&\theta_\eta^i= \eta^i-P_3 \eta^i,\quad e_\eta^i=P_3\eta^i-\eta^i_h,\\
 	&\theta_\sigmaa^i= \sigmaa^i-P_4 \sigmaa^i,\quad e_\eta^i=P_4\sigmaa^i-\sigmaa^i_h.
 \end{align*}

\subsection{Reduced rate of convergence}
  \begin{theorem}\label{biogu-error-lemma}
 	Let $(\eta^i, \sigmaa^i, \mathbf{u}^i)$ denote the exact solutions of the continuous system \eqref{biosav-rremodel}, and let $(\eta_h^i, \sigmaa_h^i, \mathbf{u}_h^i)$ be the corresponding discrete solutions obtained from the numerical scheme \eqref{biogu-1}--\eqref{biogu-7}. Then, there exist positive constants $\tau_1$ and $h_1$ such that, for all $\tau \le \tau_1$ and $h \le h_1$, the following error estimates hold:
 	\begin{align}\label{biogu-19}
 		\|e^i_\eta\|^2_{L^2}+	\|e^i_\sigmaa\|^2_{L^2}+	\|e^i_\u\|^2_{L^2}\leq C_1(\tau + h^4),\quad \forall \, 1\leq i \leq N.
 	\end{align}
 	where $C_1$ is a positive constant independent of $i$, $h$, and $\tau$.
 \end{theorem}

 \subsection{The proof for $i=1$}
 Setting $n=0$ and taking the difference between  (\ref{biogu-1})--(\ref{biogu-7}) and (\ref{biogu-13})--(\ref{biogu-14}), we have the following error equation:
 \begin{align}
 	&	(D_{\tau} e_\eta^{1},r_h)+\mu_1(\nabla e^{1}_\eta,\nabla r_h)\label{biogu-0errorequation1}
 	\\
 	=&-(D_{\tau} \theta_\eta^{1},r_h) +(\u^{1}  \eta^{0}, \nabla r_h)-(\hat{\u}^{1}_h  \eta^{0}_h, \nabla r_h)\notag
 	\\
 	&-(\sigmaa^{1} \eta^{0}, \nabla r_h)
 	+(\sigmaa^{1}_h \eta^{0}_h, \nabla r_h)
 	+(R^{1}_\eta,r_h), \quad  \forall \, r_h \in Q_h,\notag\\
 	&	(D_{\tau} e_\sigmaa^{1},\w_h)+\mu_2(\nabla \cdot e^{1}_\sigmaa,\nabla \cdot \w_h)+\mu_2(\nabla \times e^{1}_\sigmaa,\nabla \times \w_h)+(e^{1}_\sigmaa,\w_h)\label{biogu-0errorequation2}
 	\\
 	=&-(D_{\tau} \theta_\sigmaa^{1},\w_h) +(\u^{1} \cdot \sigmaa^{0}, \nabla \cdot \w_h) -(\hat{\u}^{1}_h \cdot \sigmaa^{0}_h, \nabla \cdot \w_h) -(\theta^{1}_\sigmaa,\w_h)\notag
 	\\
 	&+(\eta^{0} \nabla \eta^{1},\w_h)-(\eta^{0}_h \nabla \eta^{1}_h,\w_h)+(R^{1}_\sigmaa,\w_h),\quad \forall \w_h \in \textbf{X}_h,\notag\\
 	&(\frac{\hat{e}^{1}_\u-e^0_\u}{\tau},\v_h)+\mu_3(\nabla \hat{e}^{1}_\u,\nabla \v_h)-(p^{1},\nabla \cdot \v_h)+\mu_3(s^0_h,\nabla \cdot \v_h)\label{biogu-0errorequation3}\\
 	=&-(D_{\tau} \theta_\u^{1},\v_h)-b(\u^0,  \u^{1},\v_h)+b(\u^0_h,  \hat{\u}^{1}_h,\v_h)\notag\\
 	&-(\eta^{0} \nabla \eta^{1},\v_h)+(\eta^{0}_h \nabla \eta^{1}_h,\v_h)-(\sigmaa^{0} (\nabla \cdot \sigmaa^{1}),\v_h)\notag\\
 	&+(\sigmaa^{0}_h (\nabla \cdot \sigmaa_h^{1}),\v_h)+(R^{1}_\u,\v_h),\quad \forall \, \v_h\in \textbf{Y}_h.\notag
 \end{align}

 Summing up (\ref{biogu-0errorequation1})--(\ref{biogu-0errorequation3}) and noticing that $s^0_h=0$, denote $e^0_\eta=0, e^0_\sigmaa=\textbf{0}, e^0_\u=\textbf{0}$, it follows that
 \begin{align}
 	(e&^1_\eta,r_h)+(e^1_\sigmaa,\w_h)+(\hat{e}^1_\u,\v_h)+\mu_1\tau (\nabla e^{1}_\eta,\nabla r_h)+\mu_2 \tau (\nabla \cdot e^{1}_\sigmaa,\nabla \cdot \w_h)\\
 	&
 	+\mu_2 \tau (\nabla \times e^{1}_\sigmaa,\nabla \times \w_h)+(e^{1}_\sigmaa,\w_h)+\mu_3\tau(\nabla e^{1}_\u,\nabla \w_h)\notag\\
 =&-\tau(D_{\tau} \theta_\eta^{1},r_h) -\tau(D_{\tau} \theta_\sigmaa^{1},\w_h) -\tau(D_{\tau} \theta_\u^{1},\v_h)
\notag\\
&+\tau(R^{1}_\eta,r_h)+\tau(R^{1}_\sigmaa,\w_h)+\tau(R^{1}_\u,\v_h) +\tau(p^1,\nabla\cdot \v)\notag\\
&+\tau(\u^{1}  \eta^{0}, \nabla r_h)-\tau(\hat{\u}^{1}_h  \eta^{0}_h, \nabla r_h)-\tau(\sigmaa^{1} \eta^{0}, \nabla r_h)
 +\tau(\sigmaa^{1}_h \eta^{0}_h, \nabla r_h)\notag
 \\
 	&+\tau(\u^{1} \cdot \sigmaa^{0}, \nabla \cdot \w_h) -\tau(\hat{\u}^{1}_h \cdot \sigmaa^{0}_h, \nabla \cdot \w_h) -\tau(\theta^{1}_\sigmaa,\w_h)\notag
 \\
 &+\tau(\eta^{0} \nabla \eta^{1},\w_h)-\tau(\eta^{0}_h \nabla \eta^{1}_h,\w_h)-\tau b(\u^0,  \u^{1},\v_h)+\tau b(\u^0_h,  \hat{\u}^{1}_h,\v_h)\notag\\
 &-\tau(\eta^{0} \nabla \eta^{1},\v_h)+\tau(\eta^{0}_h \nabla \eta^{1}_h,\v_h)-\tau(\sigmaa^{0} (\nabla \cdot \sigmaa^{1}),\v_h)+\tau(\sigmaa^{0}_h (\nabla \cdot \sigmaa_h^{1}),\v_h).\notag
 \end{align}

 Thanks to (\ref{biogu-6}) and (\ref{biogu-divergencefree}), we derive that 
 \begin{align}
 	(\hat{e}^{1}_\u,\hat{e}^{1}_\u)=(e^{1}_\u+\nabla \rho^{1}_h,e^{1}_\u+\nabla \rho^{1}_h)\\
 	=(e^{1}_\u,e^{1}_\u)+(\nabla \rho^{1}_h,\nabla \rho^{1}_h).\notag
 \end{align}
 
 Setting $(r_h,\w_h,\v_h)=(e^1_\eta,e^1_\sigmaa,\hat{e}^1_\u)$, it yields that
 \begin{align}\label{biogu-error-0equation}
 	&\|e^1_\eta\|^2_{L^2}+\|e^1_\sigmaa\|^2_{L^2}+\|e^1_\u\|^2_{L^2}+\mu_1\tau \| \nabla e^1_\eta\|^2_{L^2}\\
 	&+\mu_2\tau \| \nabla\cdot  e^1_\sigmaa\|^2_{L^2}+\mu_2\tau \| \nabla \times e^1_\sigmaa\|^2_{L^2}+\mu_3\tau \| \nabla e^1_\u\|^2_{L^2}+\|\nabla \rho^{1}_h\|^2_{L^2}\notag\\
 =&-\tau(D_{\tau} \theta_\eta^{1},e^1_\eta) -\tau(D_{\tau} \theta_\sigmaa^{1},e^1_\sigmaa) -\tau(D_{\tau} \theta_\u^{1},\hat{e}^1_\u)
 \notag\\
 &+\tau(R^{1}_\eta,e^1_\eta)+\tau(R^{1}_\sigmaa,e^1_\sigmaa)+\tau(R^{1}_\u,\hat{e}^1_\u) +\tau(p^1,\nabla\cdot \hat{e}_\u)\notag\\
 &+\tau(\u^{1}  \eta^{0}, \nabla e^1_\eta)-\tau(\hat{\u}^{1}_h  \eta^{0}_h, \nabla e^1_\eta)-\tau(\sigmaa^{1} \eta^{0}, \nabla e^1_\eta)
 +\tau(\sigmaa^{1}_h \eta^{0}_h, \nabla e^1_\eta)\notag
 \\
 &+\tau(\u^{1} \cdot \sigmaa^{0}, \nabla \cdot e^1_\sigmaa) -\tau(\hat{\u}^{1}_h \cdot \sigmaa^{0}_h, \nabla \cdot e^1_\sigmaa) -\tau(\theta^{1}_\sigmaa,e^1_\sigmaa)\notag
 \\
 &+\tau(\eta^{0} \nabla \eta^{1},e^1_\sigmaa)-\tau(\eta^{0}_h \nabla \eta^{1}_h,e^1_\sigmaa)-\tau b(\u^0,  \u^{1},\hat{e}^1_\u)+\tau b(\u^0_h,  \hat{\u}^{1}_h,\hat{e}^1_\u)\notag\\
 &-\tau(\eta^{0} \nabla \eta^{1},\hat{e}^1_\u)+\tau(\eta^{0}_h \nabla \eta^{1}_h,\hat{e}^1_\u)-\tau(\sigmaa^{0} (\nabla \cdot \sigmaa^{1}),\hat{e}^1_\u)+\tau(\sigmaa^{0}_h (\nabla \cdot \sigmaa_h^{1}),\hat{e}^1_\u)\notag\\
 =&\sum_{m=1}^{22} J_m.\notag
 \end{align}
 
 Application of  H\"{o}lder inequality, Young inequality, Lagrange's Mean Value Theorem and projection error (\ref{biogu-projection-error}), there exists $t1,t2,t2\in(t_0,t_1)$ such that
 \begin{align}
 J_1+J_2+J_3=&-\tau(D_{\tau} \theta_\eta^{1},e^1_\eta) -\tau(D_{\tau} \theta_\sigmaa^{1},e^1_\sigmaa) -\tau(D_{\tau} \theta_\u^{1},e^1_\u)\\
 \leq & C \tau ( \| D_\tau \theta^1_\eta\|^2_{L^2}+\| D_\tau \theta^1_\sigmaa\|^2_{L^2}+\| D_\tau \theta^1_\u\|^2_{L^2})\notag\\
& +\nu_1\tau\| \nabla e^1_\eta\|^2_{L^2}+C\tau\| e^1_\sigmaa\|^2_{L^2}+\nu_3\tau\| \nabla e^1_\u\|^2_{L^2}\notag\\
\leq &C \tau h^4\big(\| \eta_t(t_{j1})\|^2_{H^2}+\| \sigmaa_t(t_{j2})\|^2_{H^2}+\| \u_t(t_{j3})\|^2_{H^2}\big)\notag
\\&
+\nu_1\tau\| \nabla e^1_\eta\|^2_{L^2}+C\tau\| e^1_\sigmaa\|^2_{L^2}+\nu_3\tau\| \nabla e^1_\u\|^2_{L^2}.\notag
 \end{align}
 and 
 \begin{align}
 	J_4+J_5+J_6=&\tau(R^{1}_\eta,e^1_\eta)+\tau(R^{1}_\sigmaa,e^1_\sigmaa)+\tau(R^{1}_\u,e^1_\u) \\
 	\leq &C\tau \big( \| R^1_\eta\|^2_{L^2} +\| R^1_\sigmaa\|^2_{L^2} +\| R^1_\u\|^2_{L^2} \big)\notag\\
 	&+\nu_1\tau\| \nabla e^1_\eta\|^2_{L^2}+C\tau\| e^1_\sigmaa\|^2_{L^2}+\nu_3\tau\| \nabla e^1_\u\|^2_{L^2}.\notag
 \end{align}
 
 Thanks to (\ref{biogu-6}) and (\ref{biogu-divergencefree}), it infers that
 \begin{align}
 	J_{7}=\tau(\nabla p^{1},\hat{e}^{1}_\u)=&\tau(\nabla p^{1},e^{1}_\u+\nabla \rho^{1}_h)\\
 	=&\tau(\nabla p^{1},e^{1}_\u)+\tau(\nabla p^{1},\nabla \rho_h^{1})\notag\\
 	\leq& C\tau^2\|\nabla p^{1}\|_{L^2}^2+\nu_4\|\nabla\rho^{1}_h\|^2_{L^2}.\notag
 \end{align}

After a suitable reformulation, using (\ref{biogu-regularity}) and projection error (\ref{biogu-projection-error}) we can obtain
 \begin{align}
 	J_8+J_9=&\tau(\u^{1}  \eta^{0}, \nabla e^1_\eta)-\tau(\u^{1}_h  \eta^{0}_h, \nabla e^1_\eta)\\
 =&\tau\big(   (\eta^0-\eta^0_h) \u^{1},\nabla e^{1}_\eta\big)+\tau\big(\eta^0_h (\u^{1}-\hat{\u}^{1}_h),\nabla e^{1}_\eta\big)\notag\\
 =&\tau\big(   (\eta^0-\eta^0_h) \u^{1},\nabla e^{1}_\eta\big)+\tau\big(\eta^0_h \theta_\u^{1},\nabla e^{1}_\eta\big)+\tau\big(\eta^0_h \hat{e}_\u^{1},\nabla e^{1}_\eta\big)\notag\\
 \leq &C \tau \|\theta^0_\eta\|_{L^2}\|\u^{1}\|_{L^{\infty}} \| \nabla e^{1}_\eta\|_{L^2}+C\tau \|\eta^0_h\|_{L^{\infty}}\|\theta^{1}_\u\|_{L^2}\| \nabla e^{1}_\eta\|_{L^2}+\tau\big(\eta^0_h \hat{e}_\u^{1},\nabla e^{1}_\eta\big),\notag\\
 \leq &C \tau h^4 +\nu_1 \tau \| \nabla e^{1}_\eta\|^2_{L^2}+\tau\big(\eta^0_h \hat{e}_\u^{1},\nabla e^{1}_\eta\big).\notag
 \end{align}
 where we also use (\ref{inverse-inequality}) and 
 \begin{align}\label{biogu-eta-0infty}
	\| \eta^0_h \|_{L^{\infty}} \leq& C ( \| \eta^0\|_{L^{\infty}} + \| \eta^0-\eta^0_h\|_{L^{\infty}}   )\\
	\leq & C \| \eta^0\|_{H^2}+Ch^{-1}\| \theta^0_\eta\|_{L^2}\notag\\
	\leq &C.\notag
\end{align}

After a suitable reformulation, it follows that
 \begin{align}
 	J_{10}+J_{11}=&-\tau(\sigmaa^{1} \eta^{0}, \nabla e^1_\eta)
 	+\tau(\sigmaa^{1}_h \eta^{0}_h, \nabla e^1_\eta)\\
 	=&-\tau\big(   (\eta^0-\eta^0_h) \sigmaa^{1},\nabla e^{1}_\eta\big)-\tau\big(\eta^0_h (\sigmaa^{1}-\sigmaa^{1}_h),\nabla e^{1}_\eta\big)\notag\\
 	\leq &C \tau\|\theta^0_\eta\|_{L^2}\| \sigmaa^1\|_{L^{\infty}}\|\nabla e^1_\eta\|_{L^2}+C\tau \| \eta^0_h\|_{L^{\infty}}\|\theta^1_\sigmaa\|_{L^2}\| \nabla e^1_{\eta}\|_{L^2}-\tau\big( \eta^0_h e^1_\sigmaa, \nabla e^1_\eta\big)\notag\\
 	\leq &	C \tau h^4 +\nu_1 \tau \| \nabla e^{1}_\eta\|^2_{L^2}-\tau\big( \eta^0_h e^1_\sigmaa, \nabla e^1_\eta\big).\notag
 \end{align}

 After a simple rearrangement, in view of  (\ref{biogu-regularity}) and projection error (\ref{biogu-projection-error}), we can obtain
 \begin{align}
J_{12}+J_{13}=& \tau(\u^{1} \cdot \sigmaa^{0}, \nabla \cdot e^1_\sigmaa) -\tau(\hat{\u}^{1}_h \cdot \sigmaa^{0}_h, \nabla \cdot e^1_\sigmaa)\\
 =& \tau\big(   (\sigmaa^0-\sigmaa^0_h) \u^{1},\nabla \cdot e^{1}_\sigmaa\big)+\tau\big(\sigmaa^0_h (\u^{1}-\hat{\u}^{1}_h),\nabla \cdot e^{1}_\sigmaa\big)\notag\\
 \leq &C \tau \|\theta^0_\sigmaa\|_{L^2}\|\u^{1}\|_{L^{\infty}} \| \nabla \cdot e^{1}_\sigmaa\|_{L^2}+C\tau \|\sigmaa^0_h\|_{L^{\infty}}\|\theta^{1}_\u\|_{L^2}\| \nabla \cdot e^{1}_\sigmaa\|_{L^2}+\tau\big(\sigmaa^0_h \hat{e}^{1}_\u,\nabla \cdot e^{1}_\sigmaa\big),\notag\\
 \leq &C \tau h^4 +\nu_2 \tau \| \nabla \cdot e^{1}_\sigmaa\|^2_{L^2}+\tau\big(\sigmaa^0_h \hat{e}^{1}_\u,\nabla \cdot e^{1}_\sigmaa\big).\notag
 \end{align}
  where we also use (\ref{inverse-inequality}) and 
 \begin{align}\label{biogu-sigma-0infty}
 	\| \sigmaa^0_h \|_{L^{\infty}} \leq& C ( \| \sigmaa^0\|_{L^{\infty}} + \| \sigmaa^0-\sigmaa^0_h\|_{L^{\infty}}   )\\
 	\leq & C \| \sigmaa^0\|_{H^2}+Ch^{-1}\| \theta^0_\sigmaa\|_{L^2}\notag\\
 	\leq &C.\notag
 \end{align}
 
 It's easy to see that
 \begin{align}
 	J_{14}=-\tau(\theta^{1}_\sigmaa,e^1_{\sigmaa})\leq C \tau h^4+C \tau \| e^1_\sigmaa\|_{L^2}^2.
 \end{align}
 \begin{align}\label{biogu-015}
 J_{15}+J_{16}=&\tau(\eta^{0} \nabla \eta^{1},e^1_\sigmaa)-\tau(\eta^{n}_h \nabla \eta^{1}_h,e^1_\sigmaa)\\
 	=&\tau\big(   (\eta^0-\eta^0_h) \nabla \eta^{1}, e^{1}_\sigmaa\big)+\tau\big(\eta^0_h \nabla(\eta^{1}-\eta^{1}_h),e^{1}_\sigmaa\big)\notag\\
 	\leq &C \tau \|\theta^0_\eta\|_{L^2}\|\nabla \eta^{1}\|_{L^{\infty}} \| e^{1}_\sigmaa\|_{L^2}+\tau\big(\eta^0_h \nabla e_\eta^{1},e^{1}_\sigmaa\big)+\tau\big(\eta^0_h \nabla \theta_\eta^{1},e^{1}_\sigmaa\big)\notag\\
 	\leq &C \tau h^4 +C \tau \| e^{1}_\sigmaa\|^2_{L^2}+\tau\big(\eta^0_h \nabla e_\eta^{1},e^{1}_\sigmaa\big)+\tau\big(\eta^0_h \nabla \theta_\eta^{1},e^{1}_\sigmaa\big)\notag,
 \end{align}
 
 By projection error (\ref{biogu-projection-error}) and integrate by parts (IBP), we can deduce that
 \begin{align}\label{biogu-016}
 	\tau(\eta^0_h\nabla \theta^{1}_\eta,e^{1}_\sigmaa)=&-\tau(\nabla \eta^0_h \theta^{1}_\eta,e^{1}_\sigmaa)-\tau (\eta^0_h \theta^{1}_\eta, \nabla \cdot e^{1}_\sigmaa)\\
 	\leq &C\tau \| \nabla \eta^0_h\|_{L^\infty}\|\theta^{1}_\eta\|_{L^2} \|e^{1}_\sigmaa\|_{L^2}+C \tau \| \eta^0_h\|_{L^{\infty}} \| \theta^{1}_\eta\|_{L^2}\| \nabla \cdot e^{1}_\sigmaa\|_{L^2}\notag\\
 	\leq & C \tau h^4 +C \tau \| e^1_\sigmaa\|^2_{L^2}+ \nu_2 \tau \| \nabla \cdot e^{1}_\sigmaa\|^2_{L^2},\notag
 \end{align}
 where we also use (\ref{biogu-eta-0infty}), inverse inequalities (\ref{inverse-inequality}) and
 \begin{align}\label{biogu-eta-0w13}
	\| \eta^0_h \|_{{W^{1,\infty}} }\leq& C ( \| \eta^0\|_{{W^{1,\infty}}} + \| \eta^0-\eta^0_h\|_{{W^{1,\infty}} })\\
	\leq & C \| \eta^0\|_{W^{2,4}}+Ch^{-2}\| \theta^0_\eta\|_{L^2}\notag\\
	\leq &C.\notag
\end{align}
 
 Combining (\ref{biogu-015}) and (\ref{biogu-016}), it follows that
 \begin{align}
 	J_{15}+J_{16} \leq C \tau h^4 +C \tau \| e^1_\sigmaa\|^2_{L^2}+\nu_2 \tau \| \nabla \cdot e^{1}_\sigmaa\|^2_{L^2}+\tau\big(\eta^0_h \nabla e_\eta^{1},e^{1}_\sigmaa\big).
 \end{align}
 
 After a straightforward rearrangement, and using (\ref{biogu-trilinear1}), (\ref{biogu-regularity}), together with the projection error (\ref{biogu-projection-error}), we deduce that
 \begin{align}
 	J_{17}+J_{18}=&-\tau b(\u^0,  \u^{1},\hat{e}^1_\u)+\tau b(\u^0_h,  \hat{\u}^{1}_h,\hat{e}^1_\u)\\
 =& -\big( \tau b(\u^0-\u^0_h,\u^{1},\hat{e}^{1}_\u)+\tau b(\u^0_h,\u^{1}-\hat{\u}^{1}_h,\hat{e}^{1}_\u) \big)\notag\\
 =&-\big( \tau b(\u^0-\u^0_h,\u^{1},\hat{e}^{1}_\u)-\tau b(\u^0_h, \hat{e}^{1}_\u, \theta^{1}_\u)\big)\notag\\
 \leq &C \tau \| \theta^0_\u\|_{L^2} \| \nabla \u^{1}\|_{L^3} \| \hat{e}^{1}_\u\|_{L^6}+C\tau\| \u_h^0\|_{L^{\infty}} \|\nabla \hat{e}^{1}_\u\|_{L^2} \| \theta_\u^{1}\|_{L^2}\notag\\
 \leq &C \tau h^4 +\nu_3 \tau \| \nabla \hat{e}^{1}_\u\|^2_{L^2},\notag
 \end{align}
  where we also use (\ref{inverse-inequality}) and 
\begin{align}
	\| \u^0_h \|_{L^{\infty}} \leq& C ( \| \u^0\|_{L^{\infty}} + \| \u^0-\u^0_h\|_{L^{\infty}}   )\\
	\leq & C \| \u^0\|_{H^2}+Ch^{-1}\| \theta^0_\u\|_{L^2}\notag\\
	\leq &C.\notag
\end{align} 

After a straightforward rearrangement and by applying (\ref{biogu-regularity}), (\ref{biogu-eta-0w13}), and the projection error (\ref{biogu-projection-error}), we obtain
\begin{align}\label{biogu-017}
 J_{19}+J_{20}=&-\big(\tau(\eta^{0} \nabla \eta^{1},\hat{e}^{1}_\u)-\tau(\eta^{0}_h \nabla \eta^{1}_h,\hat{e}^{1}_\u)\big)\\
	 =&-\tau\big(   (\eta^0-\eta^0_h) \nabla \eta^{1}, \hat{e}^{1}_\u\big)-\tau\big(\eta^0_h \nabla(\eta^{1}-\eta^{1}_h),\hat{e}^{1}_\u)\big)\notag\\
	\leq &C \tau \|\theta^0_\eta\|_{L^2}\|\nabla \eta^{1}\|_{L^{3}} \|\hat{e}^{1}_\u\|_{L^6}
	-\tau\big(\eta^0_h \nabla \theta_\eta^{1},\hat{e}^{1}_\u\big) -\tau\big(\eta^0_h \nabla e_\eta^{1},\hat{e}^{1}_\u\big)\notag\\
	\leq &C \tau h^4 +\nu_3 \tau \| \nabla \hat{e}^{1}_\u \|^2_{L^2}-\tau\big(\eta^0_h \nabla \theta_\eta^{1},\hat{e}^{1}_\u\big)-\tau\big(\eta^0_h \nabla e_\eta^{1},\hat{e}^{1}_\u\big),\notag
\end{align}

By (\ref{biogu-eta-0w13}), (\ref{biogu-eta-0infty}), projection error (\ref{biogu-projection-error}) and integrate by parts (IBP), we can deduce that
\begin{align}\label{biogu-018}
	-\tau(\eta^0_h\nabla \theta^{1}_\eta,\hat{e}^{1}_\u)=&\tau(\nabla \eta^0_h \theta^{1}_\eta,\hat{e}^{1}_\u)+\tau (\eta^0_h \theta^{1}_\eta, \nabla \hat{e}^{1}_\u)\\
	\leq &C\tau \| \nabla \eta^0_h\|_{L^3}\|\theta^{1}_\eta\|_{L^2} \|\hat{e}^{1}_\u\|_{L^6}+C \tau \| \eta^0_h\|_{L^{\infty}} \| \theta^{1}_\eta\|_{L^2}\| \nabla  \hat{e}^{1}_\u\|_{L^2}\notag\\
	\leq & C \tau h^4 + \nu_3 \tau \| \nabla \hat{e}^{1}_\u\|^2_{L^2},\notag
\end{align}

Combining (\ref{biogu-017}) and (\ref{biogu-018}), it follows that
\begin{align}
	J_{19}+J_{20} \leq C \tau h^4 +\nu_3 \tau \| \nabla \hat{e}^{1}_\u\|^2_{L^2}-\tau\big(\eta^0_h \nabla e_\eta^{1},\hat{e}^{1}_\u\big).
\end{align}

 Furthermore
 \begin{align}
 	J_{21}+J_{22}=&-\tau(\sigmaa^{0} (\nabla \cdot \sigmaa^{1}),\hat{e}^{1}_\u)
 	+\tau(\sigmaa^{0}_h (\nabla \cdot \sigmaa_h^{1}),\hat{e}^{1}_\u)\\
 	=&-\tau \big( (\sigmaa^0-\sigmaa^0_h) \hat{e}^{1}_\u,\nabla \cdot \sigmaa^{1}  \big)  -\tau \big(\sigmaa^0_h \hat{e}^{1}_\u,\nabla \cdot (\sigmaa^{1}-\sigmaa^{1}_h) \big) \notag \\
 	\leq &C \tau\|\theta^0_\sigmaa\|_{L^2}\|\hat{e}^{1}_\u\|_{L^6} \|\nabla \cdot \sigmaa^{1}\|_{L^3}-\tau \big( \sigmaa^0_h \hat{e}^{1}_\u,\nabla \cdot e^{1}_\sigmaa \big)-\tau \big( \sigmaa^0_h \hat{e}^{1}_\u,\nabla \cdot \theta^{1}_\sigmaa \big)\notag\\
 	\leq & C \tau h^4+\nu_3 \tau \| \nabla \hat{e}^1_\u\|^2_{L^2}-\tau \big( \sigmaa^0_h \hat{e}^{1}_\u,\nabla \cdot e^{1}_\sigmaa \big)-\tau \big( \sigmaa^0_h \hat{e}^{1}_\u,\nabla \cdot \theta^{1}_\sigmaa \big).\notag
 \end{align}
 as for $-\tau \big( \sigmaa^0_h \hat{e}^{1}_\u,\nabla \cdot \theta^{1}_\sigmaa \big)$, we can deduce that 
 \begin{align}
 	-\tau \big( \sigmaa^0_h \hat{e}^{1}_\u,\nabla \cdot \theta^{1}_\sigmaa \big)
 	=&\tau \big(\nabla \sigmaa^0_h \hat{e}^{1}_\u, \theta^{1}_\sigmaa \big)
 	+\tau \big( \sigmaa^0_h \nabla \hat{e}^{1}_\u, \theta^{1}_\sigmaa \big)\\
 	\leq& C \tau \| \nabla \sigmaa^0_h \|_{L^3}\|\hat{e}^{1}_\u\|_{L^6}\|\theta^{1}_\sigmaa\|_{L^2}+C\tau\|  \sigmaa^0_h \|_{L^{\infty}}\| \nabla \hat{e}^{1}_\u\|_{L^2}\|\theta^{1}_\sigmaa\|_{L^2}\notag\\
 	\leq &C \tau h^4+\nu_3\tau\| \nabla \hat{e}^{1}_\u\|_{L^2}^2.\nonumber
 \end{align}
  where we also use IBP, (\ref{biogu-sigma-0infty}), inverse inequalities (\ref{inverse-inequality}) and for sufficiently small $h$, such that $Ch^{\frac{2}{3}}\leq C $, we have
 \begin{align}\label{biogu-signa-0w13}
 	\| \sigmaa^0_h \|_{{W^{1,3}} }\leq& C ( \| \sigmaa^0\|_{{W^{1,3}}} + \| \sigmaa^0-\sigmaa^0_h\|_{{W^{1,3}} })\\
 	\leq & C \| \sigmaa^0\|_{H^2}+Ch^{-\frac{4}{3}}\| \theta^0_\sigmaa\|_{L^2}\notag\\
 	\leq &C+Ch^{\frac{2}{3}}\notag\\
 	\leq & C.\notag
 \end{align}
 thus
 \begin{align}
 	J_{21}+J_{22}\leq C \tau h^4+\nu_3 \tau \| \nabla \hat{e}^{1}_\u\|^2_{L^2}-\tau \big( \sigmaa^0_h \hat{e}^{1}_\u,\nabla \cdot e^{1}_\sigmaa \big).
 \end{align}
 
 Furthermore, for sufficiently small $\nu_1,\nu_2,\nu_3,\tau$,  combining the above inequalities with (\ref{biogu-error-0equation}), by using (\ref{biogu-regularity}), (\ref{biogu-trunction-error}), there exists $C_1>0$ such that
 \begin{align}\label{biogu-i=1}
 	&\|e^1_\eta\|^2_{L^2}+\|e^1_\sigmaa\|^2_{L^2}+\|e^1_\u\|^2_{L^2}+\|\nabla \rho^{1}_h\|^2_{L^2}\\
 	&+\mu_1\tau \| \nabla e^1_\eta\|^2_{L^2}+\mu_2\tau \| \nabla \cdot e^1_\sigmaa\|^2_{L^2}+\mu_2\tau \| \nabla \times e^1_\sigmaa\|^2_{L^2}+\mu_3\tau \| \nabla e^1_\u\|^2_{L^2}\notag\\
 	\leq &C \tau h^4(\| \eta_t(t_{j1})\|^2_{H^2}+\| \sigmaa_t(t_{j2})\|^2_{H^2}+\| \u_t(t_{j3}))\|^2_{H^2}\notag\\
 	&+C\tau \big( \| R^1_\eta\|^2_{L^2} +\| R^1_\sigmaa\|^2_{L^2} +\| R^1_\u\|^2_{L^2} \big)+C\tau^2\|\nabla p^{1}\|_{L^2}^2+C\tau h^4 \notag\\
 	\leq &C\tau h^4+C\tau^2\notag\\
 	\leq &C_1(\tau^2+h^4).\notag
 \end{align}
 
Now, we have proved the optimal error estimate for $i=1$. The proof for $i=N$ will be  presented in the following.
 
 \subsection{The proof for $i=N$}
 \begin{proof}
 	Mathematical induction is employed  to (\ref{biogu-19}). Firstly, we assume that 
 	\begin{align}\label{biogu-induction}
	\|e^n_\eta\|_{L^2}+	\|e^n_\sigmaa\|_{L^2}+	\|e^n_\u\|_{L^2}\leq C_1h^2, \quad \forall \, 1\leq n \leq N-1.
\end{align}

For sufficiently small $h$ such that $Ch^{\frac{2}{3}} \leq C$, by (\ref{biogu-regularity}), inverse inequalities (\ref{inverse-inequality}) and triangle inequalities, it yields thats
\begin{align}
	\| \eta^n_h\|_{L^{\infty}} \leq& \| P_3 \eta^n\|_{L^{\infty}}+ \| e^n_\eta\|_{L^{\infty}}\label{biogu-eta-infty}\\
	\leq& \|\eta^n\|_{H^2}+C h^{-1} \|e^n_\eta\|_{L^2} \notag\\
	\leq& C,\notag\\
	\| \nabla \eta^n_h\|_{L^3} \leq& \| \eta^n\|_{W^{1,3}}+C \| e^n_\eta\|_{W^{1,3}}\label{biogu-eta-w13}\\
	\leq & \| \eta^n\|_{H^2}+C h^{-\frac{4}{3}}\| e^n_\eta\|_{L^2}\notag\\
	\leq& \| \eta^n\|_{H^2}+Ch^{\frac{2}{3}}\notag\\
	\leq & C,\notag\\
		\| \eta^n_h \|_{{W^{1,\infty}} }\leq& C ( \| \eta^n\|_{{W^{1,\infty}}} + \| \eta^n-\eta^n_h\|_{{W^{1,\infty}} })\label{biogu-eta-w1infty}\\
		\leq & C \| \eta^n\|_{W^{2,4}}+Ch^{-2}\| \theta^n_\eta\|_{L^2}\notag\\
		\leq &C.\notag
\end{align}
similarly
\begin{align}
	\| \u^n_h\|_{L^{\infty}} \leq& \| P_1 \u^n\|_{L^{\infty}}+ \| e^n_\u\|_{L^{\infty}}\label{biogu-u-infty}\\
	\leq& \|\u^n\|_{H^2}+C h^{-1} \|e^n_\u\|_{L^2} \notag\\
	\leq& C,\notag\\
	\| \sigmaa^n_h\|_{L^{\infty}} \leq& \| P_4 \sigmaa^n\|_{L^{\infty}}+ \| e^n_\sigmaa\|_{L^{\infty}}\label{biogu-sigma-infty}\\
	\leq& \|\sigmaa^n\|_{H^2}+C h^{-1} \|e^n_\sigmaa\|_{L^2} \notag\\
	\leq& C,\notag\\
		\| \nabla \sigmaa^n_h\|_{L^3} \leq& \| \sigmaa^n\|_{W^{1,3}}+C \| e^n_\sigmaa\|_{W^{1,3}}\label{biogu-sigma-w13}\\
	\leq & \| \sigmaa^n\|_{H^2}+C h^{-\frac{4}{3}}\| e^n_\sigmaa\|_{L^2}\notag\\
	\leq& \| \sigmaa^n\|_{H^2}+Ch^{\frac{2}{3}}\notag\\
	\leq & C.\notag
\end{align}

 Subtracting  (\ref{biogu-1})--(\ref{biogu-3}) from (\ref{biogu-13})--(\ref{biogu-14}), we can get the following error equation:
\begin{align}
&	(D_{\tau} e_\eta^{n+1},r_h)+\mu_1(\nabla e^{n+1}_\eta,\nabla r_h)\label{biogu-errorequation1}
	\\
	=&-(D_{\tau} \theta_\eta^{n+1},r_h) +(\u^{n+1}  \eta^{n}, \nabla r_h)-(\hat{\u}^{n+1}_h  \eta^{n}_h, \nabla r_h)\notag
	\\
	&-(\sigmaa^{n+1} \eta^{n}, \nabla r_h)
	+(\sigmaa^{n+1}_h \eta^{n}_h, \nabla r_h)
	+(R^{n+1}_\eta,r_h), \quad  \forall \, r_h \in Q_h,\notag\\
	&	(D_{\tau} e_\sigmaa^{n+1},\w_h)+\mu_2(\nabla \cdot e^{n+1}_\sigmaa,\nabla \cdot \w_h)+\mu_2(\nabla \times e^{n+1}_\sigmaa,\nabla \times \w_h)+(e^{n+1}_\sigmaa,\w_h)\label{biogu-errorequation2}
	\\
	=&-(D_{\tau} \theta_\sigmaa^{n+1},\w_h) +(\u^{n+1} \cdot \sigmaa^{n}, \nabla \cdot \w_h) -(\hat{\u}^{n+1}_h \cdot \sigmaa^{n}_h, \nabla \cdot \w_h) -(\theta^{n+1}_\sigmaa,\w_h)\notag
	\\
	&+(\eta^{n} \nabla \eta^{n+1},\w_h)-(\eta^{n}_h \nabla \eta^{n+1}_h,\w_h)+(R^{n+1}_\sigmaa,\w_h),\quad \forall \w_h \in X_h,\notag\\
	&(\frac{\hat{e}^{n+1}_\u-e^n_\u}{\tau},\v_h)+\mu_3(\nabla \hat{e}^{n+1}_\u,\nabla \w_h)-(p^{n+1},\nabla \cdot \v_h)+\mu_3(s^n_h,\nabla \cdot \v_h)\label{biogu-errorequation3}\\
	=&-(D_{\tau} \theta_\u^{n+1},\v_h)-b(\u^n,  \u^{n+1},\v_h)+b(\u^n_h,  \hat{\u}^{n+1}_h,\v_h)\notag\\
	&-(\eta^{n} \nabla \eta^{n+1},\v_h)+(\eta^{n}_h \nabla \eta^{n+1}_h,\v_h)-(\sigmaa^{n} (\nabla \cdot \sigmaa^{n+1}),\v_h)\notag\\
	&+(\sigmaa^{n}_h (\nabla \cdot \sigmaa_h^{n+1}),\v_h)+(R^{n+1}_\u,\v_h),\quad \forall \, \v_h\in \textbf{Y}_h.\notag
\end{align}
 
Using (\ref{biogu-5}), we deduce that 
\begin{align}
	2(s^n_h,\nabla \cdot \hat{e}^{n+1}_\u) &=-2 (s^n_h,\nabla \cdot \hat{\u}_h^{n+1})\\
	&=-2(s^n_h,s^n_h-s^{n+1}_h)\notag\\
	&=\|s^{n+1}_h\|^2_{L^2}-\|s^{n}_h\|^2_{L^2}-\|s^n_h-s^{n+1}_h\|^2_{L^2}.\notag
\end{align}

Combing (\ref{biogu-5}) with Div--grad inequality, it infers that 
\begin{align}
	\|s^n_h-s^{n+1}_h\|_{L^2}\leq \| \nabla \cdot \hat{e}^{n+1}_\u\|_{L^2}\leq \| \nabla \hat{e}^{n+1}_\u\|_{L^2}.
\end{align}

Thanks to (\ref{biogu-6}) and (\ref{biogu-divergencefree}), we derive that 
\begin{align}
	(\hat{e}^{n+1}_\u-e^n_\u,\hat{e}^{n+1}_\u)=(e^{n+1}_\u+\nabla \rho^{n+1}_h-e^n_\u,e^{n+1}_\u+\nabla \rho^{n+1}_h)\\
	=(e^{n+1}_\u-e^n_\u,e^{n+1}_\u)+(\nabla \rho^{n+1}_h,\nabla \rho^{n+1}_h).\notag
\end{align}

Thus, taking $(r_h,\w_h,\v_h)=2\tau(e^{n+1}_\eta,e^{n+1}_\sigmaa,\hat{e}^{n+1}_\u)$ in (\ref{biogu-errorequation1})--(\ref{biogu-errorequation3}), we have 
\begin{align}
	&\| e^{n+1}_\eta\|^2_{L^2}-\| e^n_\eta\|^2_{L^2}+\|  e^{n+1}_\eta- e^{n}_\eta\|^2_{L^2}+2\mu_1 \tau \| \nabla e^{n+1}_\eta\|^2_{L^2}\label{biogu-errorequation11}
	\\
	=&-2\tau(D_{\tau} \theta_\eta^{n+1},e^{n+1}_\eta) +2\tau(\u^{n+1}  \eta^{n}, \nabla e^{n+1}_\eta)-2\tau(\hat{\u}^{n+1}_h  \eta^{n}_h, \nabla e^{n+1}_\eta)\notag
	\\
	&-2\tau(\sigmaa^{n+1} \eta^{n}, \nabla e^{n+1}_\eta)
	+2\tau(\sigmaa^{n+1}_h \eta^{n}_h, \nabla e^{n+1}_\eta)
	+2\tau(R^{n+1}_\eta,e^{n+1}_\eta), \notag\\
	=&\sum_{m=1}^{6}I_m.\notag
	\\
	&\| e^{n+1}_\sigmaa\|^2_{L^2}-\| e^n_\sigmaa\|^2_{L^2}+\|  e^{n+1}_\sigmaa- e^{n}_\sigmaa\|^2_{L^2}\label{biogu-errorequation22}\\
	&+2\mu_2 \tau \| \nabla \cdot e^{n+1}_\sigmaa\|^2_{L^2}+\mu_2 \tau \| \nabla \times e^{n+1}_\sigmaa\|^2_{L^2}+2\tau \|  e^{n+1}_\sigmaa\|^2_{L^2}\notag\\
	=&-2\tau(D_{\tau} \theta_\sigmaa^{n+1},e^{n+1}_\sigmaa) +2\tau(\u^{n+1} \cdot \sigmaa^{n}, \nabla \cdot e^{n+1}_\sigmaa) -2\tau(\u^{n+1}_h \cdot \sigmaa^{n}_h, \nabla \cdot e^{n+1}_\sigmaa) \notag
	\\
	&-2\tau(\theta^{n+1}_\sigmaa,e^{n+1}_\sigmaa)+2\tau(\eta^{n} \nabla \eta^{n+1},e^{n+1}_\sigmaa)-2\tau(\eta^{n}_h \nabla \eta^{n+1}_h,e^{n+1}_\sigmaa)+2\tau(R^{n+1}_\sigmaa,e^{n+1}_\sigmaa),\notag\\
	=&\sum_{m=7}^{13}I_m.\notag\\
		&\| e^{n+1}_\u\|^2_{L^2}-\| e^n_\u\|^2_{L^2}+\|  e^{n+1}_\u- e^{n}_\u\|^2_{L^2}+\mu_3 \tau (\| s^{n+1}_h\|^2_{L^2}-\| s^n_h\|^2_{L^2})+\mu_3 \tau \| \nabla \hat{e}^{n+1}_\u\|^2_{L^2}+2\|  \nabla \rho^{n+1}_h\|^2_{L^2}\label{biogu-errorequation33}\\
			=&-2\tau(D_{\tau} \theta_\u^{n+1},\hat{e}^{n+1}_\u)-2\tau b(\u^n,  \u^{n+1},\hat{e}^{n+1}_\u)+2\tau b(\u^n_h,  \hat{\u}^{n+1}_h,\hat{e}^{n+1}_\u)\notag\\
		&-2\tau(\eta^{n} \nabla \eta^{n+1},\hat{e}^{n+1}_\u)+2\tau(\eta^{n}_h \nabla \eta^{n+1}_h,\hat{e}^{n+1}_\u)-2\tau(\sigmaa^{n} (\nabla \cdot \sigmaa^{n+1}),\hat{e}^{n+1}_\u)\notag\\
		&+2\tau(\sigmaa^{n}_h (\nabla \cdot \sigmaa_h^{n+1}),\hat{e}^{n+1}_\u)+2\tau(R^{n+1}_\u,\hat{e}^{n+1}_\u)-2\tau(p^{n+1},\nabla \cdot \hat{e}^{n+1}_\u)\notag,
		\\
		=&\sum_{m=14}^{22}I_m.\notag
\end{align}

Application of  H\"{o}lder inequality, Young inequality and projection error (\ref{biogu-projection-error}), we have
\begin{align}
	I_1=&-2\tau(D_{\tau} \theta_\eta^{n+1},e^{n+1}_\eta)\notag\\
	 \leq& C \tau \| D_{\tau} \theta_\eta^{n+1}\|^2_{L^2}+\epsilon_1 \tau \| \nabla e^{n+1}_\eta\|^2_{L^2}\notag\\
	\leq &C \tau h^4 \| D_{\tau} \eta^{n+1}\|^2_{L^2}+\epsilon_1 \tau \| \nabla e^{n+1}_\eta\|^2_{L^2}
\end{align}

After a simple rearrangement, using (\ref{biogu-regularity}), (\ref{biogu-eta-infty}) and projection error (\ref{biogu-projection-error}), we can derive the following inequalities:
\begin{align}
	I_2+I_3=&2\tau(\u^{n+1}  \eta^{n}, \nabla e^{n+1}_\eta)-2\tau(\hat{\u}^{n+1}_h  \eta^{n}_h, \nabla e^{n+1}_\eta)\\
	=&2 \tau\big(   (\eta^n-\eta^n_h) \u^{n+1},\nabla e^{n+1}_\eta\big)+2\tau\big(\eta^n_h (\u^{n+1}-\hat{\u}^{n+1}_h),\nabla e^{n+1}_\eta\big)\notag\\
	=&2 \tau\big(   (\eta^n-\eta^n_h) \u^{n+1},\nabla e^{n+1}_\eta\big)+2\tau\big(\eta^n_h \theta_\u^{n+1},\nabla e^{n+1}_\eta\big)+2\tau\big(\eta^n_h \hat{e}_\u^{n+1},\nabla e^{n+1}_\eta\big)\notag\\
	\leq &C \tau \|e^n_\eta+\theta^n_\eta\|_{L^2}\|\u^{n+1}\|_{L^{\infty}} \| \nabla e^{n+1}_\eta\|_{L^2}\notag
	\\
	&+C\tau \|\eta^n_h\|_{L^{\infty}}\|\theta^{n+1}_\u\|_{L^2}\| \nabla e^{n+1}_\eta\|_{L^2}+2\tau\big(\eta^n_h \hat{e}_\u^{n+1},\nabla e^{n+1}_\eta\big),\notag\\
	\leq &C \tau h^4+C \tau \| e^n_\eta\|^2_{L^2}+\epsilon_1 \tau \| \nabla e^{n+1}_\eta\|^2_{L^2}+2\tau\big(\eta^n_h \hat{e}_\u^{n+1},\nabla e^{n+1}_\eta\big).\notag\\
		I_4+I_5=&-2\tau(\sigmaa^{n+1} \eta^{n}, \nabla e^{n+1}_\eta)
	+2\tau(\sigmaa^{n+1}_h \eta^{n}_h, \nabla e^{n+1}_\eta)\\
	=&-2 \tau\big(   (\eta^n-\eta^n_h) \sigmaa^{n+1},\nabla e^{n+1}_\eta\big)-\big(\eta^n_h (\sigmaa^{n+1}-\sigmaa^{n+1}_h),\nabla e^{n+1}_\eta\big)\notag\\
\leq &	C \tau h^4 +C \tau( \| e^n_\eta\|^2_{L^2}+\| e^{n+1}_\sigmaa\|^2_{L^2})+\epsilon_1 \tau \| \nabla e^{n+1}_\eta\|^2_{L^2}.\notag
\end{align}

And
\begin{align}
	I_6=2\tau(R^{n+1}_\eta,e^{n+1}_\eta)
	\leq C \tau  \|R^{n+1}_\eta\|^2_{L^2}+\epsilon_1 \tau \| \nabla e^{n+1}_\eta\|^2_{L^2}.
\end{align}

For sufficiently small $\epsilon_1$, combing the above inequalities with (\ref{biogu-errorequation11}) yields 
\begin{align}\label{biogu-eta-errorr-esult}
	&\| e^{n+1}_\eta\|^2_{L^2}-\| e^n_\eta\|^2_{L^2}+\|  e^{n+1}_\eta- e^{n}_\eta\|^2_{L^2}+\mu_1 \tau \| \nabla e^{n+1}_\eta\|^2_{L^2}\\
	\leq &C \tau h^4 + C \tau( \| e^n_\eta\|^2_{L^2}+\| e^{n+1}_\sigmaa\|^2_{L^2})+C \tau h^4 \| D_{\tau} \eta^{n+1}\|^2_{L^2}+C\tau  \|R^{n+1}_\eta\|^2_{L^2}+2\tau\big(\eta^n_h \hat{e}_\u^{n+1},\nabla e^{n+1}_\eta\big).\notag
\end{align}

With the application of the H\"{o}lder inequality and Young inequalities, we can deal with $\sum_{m=7}^{13}I_m$ by 
\begin{align}
	I_7=&-2\tau(D_{\tau} \theta_\sigmaa^{n+1},e^{n+1}_\sigmaa)\\
	 \leq& C \tau \| D_{\tau} \theta_\sigmaa^{n+1}\|^2_{L^2}+C \tau \| e^{n+1}_\sigmaa\|^2_{L^2}\notag\\
	\leq &C \tau h^4 \| D_{\tau} \sigmaa^{n+1}\|^2_{L^2}+C \tau \| e^{n+1}_\sigmaa\|^2_{L^2}\notag.
\end{align}

After a simple rearrangement, in view of  (\ref{biogu-regularity}), (\ref{biogu-sigma-infty}) and projection error (\ref{biogu-projection-error}), we can obtain
\begin{align}
	I_8+I_9=&2\tau(\u^{n+1} \cdot \sigmaa^{n}, \nabla \cdot e^{n+1}_\sigmaa) -2\tau(\hat{\u}^{n+1}_h \cdot \sigmaa^{n}_h, \nabla \cdot e^{n+1}_\sigmaa)\\
	=&2 \tau\big(   (\sigmaa^n-\sigmaa^n_h) \u^{n+1},\nabla \cdot e^{n+1}_\sigmaa\big)+2\tau\big(\sigmaa^n_h (\u^{n+1}-\hat{\u}^{n+1}_h),\nabla \cdot e^{n+1}_\sigmaa\big)\notag\\
	\leq &C \tau \|e^n_\sigmaa+\theta^n_\sigmaa\|_{L^2}\|\u^{n+1}\|_{L^{\infty}} \| \nabla \cdot e^{n+1}_\sigmaa\|_{L^2}\notag\\
	&+C\tau \|\sigmaa^n_h\|_{L^{\infty}}\|\theta^{n+1}_\u\|_{L^2}\| \nabla \cdot e^{n+1}_\sigmaa\|_{L^2}+2\tau\big(\sigmaa^n_h \hat{e}^{n+1}_\u,\nabla \cdot e^{n+1}_\sigmaa\big),\notag\\
	\leq &C \tau h^4 +C \tau\| e^n_\sigmaa\|^2_{L^2}+\epsilon_2 \tau \| \nabla\cdot  e^{n+1}_\sigmaa\|^2_{L^2}+2\tau\big(\sigmaa^n_h \hat{e}^{n+1}_\u,\nabla \cdot e^{n+1}_\sigmaa\big).\notag
\end{align}
\begin{align}
		I_{10}=-2\tau(\theta^{n+1}_\sigmaa,e^{n+1}_\sigmaa)\leq C\tau h^4+ C\tau\| e^{n+1}_\sigmaa\|^2_{L^2}.
\end{align}
\begin{align}\label{biogu-15}
I_{11}+I_{12}=&2 \tau(\eta^{n} \nabla \eta^{n+1},e^{n+1}_\sigmaa)-2\tau(\eta^{n}_h \nabla \eta^{n+1}_h,e^{n+1}_\sigmaa)\\
	=&2 \tau\big(   (\eta^n-\eta^n_h) \nabla \eta^{n+1}, e^{n+1}_\sigmaa\big)+2\tau\big(\eta^n_h \nabla(\eta^{n+1}-\eta^{n+1}_h),e^{n+1}_\sigmaa\big)\notag\\
\leq &C \tau \|e^n_\eta+\theta^n_\eta\|_{L^2}\|\nabla \eta^{n+1}\|_{L^{\infty}} \| e^{n+1}_\sigmaa\|_{L^2}\notag\\
&+C\tau \|\eta^n_h\|_{L^{\infty}}\| \nabla e^{n+1}_\eta\|_{L^2}\|  e^{n+1}_\sigmaa\|_{L^2}+2\tau\big(\eta^n_h \nabla \theta_\eta^{n+1},e^{n+1}_\sigmaa\big)\notag\\
\leq &C \tau h^4 +C \tau( \| e^n_\eta\|^2_{L^2}+\| e^{n+1}_\sigmaa\|^2_{L^2})\\
&+\epsilon_1\tau \| \nabla e^{n+1}_\eta\|^2_{L^2}+2\tau\big(\eta^n_h \nabla \theta_\eta^{n+1},e^{n+1}_\sigmaa\big),\notag
\end{align}

By (\ref{biogu-eta-w1infty}), projection error (\ref{biogu-projection-error}) and integrate by parts (IBP), we can deduce that
\begin{align}\label{biogu-16}
	2\tau(\eta^n_h\nabla \theta^{n+1}_\eta,e^{n+1}_\sigmaa)=&-2\tau(\nabla \eta^n_h \theta^{n+1}_\eta,e^{n+1}_\sigmaa)-2\tau (\eta^n_h \theta^{n+1}_\eta, \nabla \cdot e^{n+1}_\sigmaa)\\
	\leq &C\tau \| \nabla \eta^n_h\|_{L^\infty}\|\theta^{n+1}_\eta\|_{L^2} \|e^{n+1}_\sigmaa\|_{L^2}+C \tau \| \eta^n_h\|_{L^{\infty}} \| \theta^{n+1}_\eta\|_{L^2}\| \nabla \cdot e^{n+1}_\sigmaa\|_{L^2}\notag\\
	\leq & C \tau h^4 + \epsilon_2 \tau \| \nabla \cdot e^{n+1}_\sigmaa\|^2_{L^2}+C\tau \|e^{n+1}_\sigmaa\|_{L^2}^2 ,\notag
\end{align}

Combining (\ref{biogu-15}) and (\ref{biogu-16}), it follows that
\begin{align}
I_{11}+I_{12} \leq C \tau h^4 +C \tau( \| e^n_\eta\|^2_{L^2}+\| e^{n+1}_\sigmaa\|^2_{L^2})+\epsilon_2 \tau \| \nabla \cdot e^{n+1}_\sigmaa\|^2_{L^2}.
\end{align}
\begin{align}
		I_{13}=2\tau(R^{n+1}_\sigmaa,e^{n+1}_\sigmaa)
	\leq C \tau  \|R^{n+1}_\sigmaa\|^2_{L^2}+C \tau \| e^{n+1}_\sigmaa\|^2_{L^2}.
\end{align}

For sufficiently small $\epsilon_2$, combing the above inequalities with (\ref{biogu-errorequation22}) yields
\begin{align}\label{biogu-sigma-errorr-esult}
	&\| e^{n+1}_\sigmaa\|^2_{L^2}-\| e^n_\sigmaa\|^2_{L^2}+\|  e^{n+1}_\sigmaa- e^{n}_\sigmaa\|^2_{L^2}\\
	&+\mu_2 \tau \| \nabla \cdot e^{n+1}_\sigmaa\|^2_{L^2}+\mu_2 \tau \| \nabla \times e^{n+1}_\sigmaa\|^2_{L^2}+2\tau \|  e^{n+1}_\sigmaa\|^2_{L^2}\notag\\
	\leq &C \tau h^4\notag+ C \tau(  \| e^n_\sigmaa\|^2_{L^2}+\| e^n_\eta\|^2_{L^2}+\| e^{n+1}_\eta\|^2_{L^2}+\| e^{n+1}_\sigmaa\|^2_{L^2})\notag\\
	&+C \tau h^4 \| D_{\tau} \sigmaa^{n+1}\|^2_{L^2}+C\tau \|R^{n+1}_\sigmaa\|^2_{L^2}+2\tau\big(\sigmaa^n_h \hat{e}^{n+1}_\u,\nabla \cdot e^{n+1}_\sigmaa\big).\notag
\end{align}

As for $\sum_{m=14}^{22}I_m$, through the H\"{o}lder inequality and Young inequalities, we conclude that
\begin{align}
		I_{14}=&-2\tau(D_{\tau} \theta_\u^{n+1},\hat{e}^{n+1}_\u)\\
		 \leq &C \tau \| D_{\tau} \theta_\u^{n+1}\|^2_{L^2}+\epsilon_3 \tau \| \nabla \hat{e}^{n+1}_\u\|^2_{L^2}\notag\\
		 \leq &C \tau  h^4\| D_{\tau} \u^{n+1}\|^2_{L^2}+\epsilon_3 \tau \| \nabla \hat{e}^{n+1}_\u\|^2_{L^2}.\notag
\end{align}

After a simple rearrangement, by (\ref{biogu-trilinear1}),  (\ref{biogu-regularity}), (\ref{biogu-u-infty}), projection error (\ref{biogu-projection-error}), we deduce that
\begin{align}
	I_{15}+I_{16}=&-2\tau b(\u^n,  \u^{n+1},\hat{e}^{n+1}_\u)+2\tau b(\u^n_h,  \hat{\u}^{n+1}_h,\hat{e}^{n+1}_\u)\\
	=& - 2\tau b(\u^n-\u^n_h,\u^{n+1},\hat{e}^{n+1}_\u)-2\tau b(\u^n_h,\u^{n+1}-\hat{\u}^{n+1}_h,\hat{e}^{n+1}_\u)\notag\\
	=&- 2\tau b(\u^n-\u^n_h,\u^{n+1},\hat{e}^{n+1}_\u)+2\tau b(\u^n_h, \hat{e}^{n+1}_\u, \theta^{n+1}_\u)\notag\\
	\leq &C \tau \| e^n_\u+\theta^n_\u\|_{L^2} \| \nabla \u^{n+1}\|_{L^3} \| \hat{e}^{n+1}_\u\|_{L^6}+C\tau\| \u^n_h\|_{L^{\infty}} \|\nabla \hat{e}^{n+1}_\u\|_{L^2} \| \theta^{n+1}_\u\|_{L^2}\notag\\
	\leq &C \tau h^4 +\epsilon_3 \tau \| \nabla \hat{e}^{n+1}_\u\|^2_{L^2}+C \tau \| e^n_\u\|^2_{L^2}.\notag
\end{align}

After a simple rearrangement, in terms of  (\ref{biogu-regularity}), (\ref{biogu-sigma-infty}), and projection error (\ref{biogu-projection-error}), we have
\begin{align}\label{biogu-17}
I_{17}+I_{18}=&-\big(2 \tau(\eta^{n} \nabla \eta^{n+1},\hat{e}^{n+1}_\u)-2\tau(\eta^{n}_h \nabla \eta^{n+1}_h,\hat{e}^{n+1}_\u)\big)\\
	\leq &2 \tau|\big(   (\eta^n-\eta^n_h) \nabla \eta^{n+1}, \hat{e}^{n+1}_\u\big)|-2\tau\big(\eta^n_h \nabla(\eta^{n+1}-\eta^{n+1}_h),\hat{e}^{n+1}_\u\big)\notag\\
	\leq &C \tau \|e^n_\eta+\theta^n_\eta\|_{L^2}\|\nabla \eta^{n+1}\|_{L^{3}} \|\hat{e}^{n+1}_\u\|_{L^6}\notag\\
	&+2\tau|\big(\eta^n_h \nabla \theta_\eta^{n+1},\hat{e}^{n+1}_\u\big)| -2\tau\big(\eta^n_h \nabla e_\eta^{n+1},\hat{e}^{n+1}_\u\big)\notag\\
	\leq &C \tau h^4 +C \tau \| e^n_\eta\|^2_{L^2}+\epsilon_3 \tau \| \nabla \hat{e}^{n+1}_\u \|^2_{L^2}\notag\\
	&+2\tau|\big(\eta^n_h \nabla \theta_\eta^{n+1},\hat{e}^{n+1}_\u\big)|-2\tau\big(\eta^n_h \nabla e_\eta^{n+1},\hat{e}^{n+1}_\u\big),\notag
\end{align}

By (\ref{biogu-eta-w13}), projection error (\ref{biogu-projection-error}) and integrate by parts (IBP), we can deduce that
\begin{align}\label{biogu-18}
	2\tau|(\eta^n_h\nabla \theta^{n+1}_\eta,\hat{e}^{n+1}_\u)|=&|-2\tau(\nabla \eta^n_h \theta^{n+1}_\eta,\hat{e}^{n+1}_\u)-2\tau (\eta^n_h \theta^{n+1}_\eta, \nabla \hat{e}^{n+1}_\u)|\\
	\leq &C\tau \| \nabla \eta^n_h\|_{L^3}\|\theta^{n+1}_\eta\|_{L^2} \|\hat{e}^{n+1}_\u\|_{L^6}+C \tau \| \eta^n_h\|_{L^{\infty}} \| \theta^{n+1}_\eta\|_{L^2}\| \nabla  \hat{e}^{n+1}_\u\|_{L^2}\notag\\
	\leq & C \tau h^4 + \epsilon_3\tau \| \nabla \hat{e}^{n+1}_\u\|^2_{L^2},\notag
\end{align}

Combining (\ref{biogu-17}) and (\ref{biogu-18}), it follows that
\begin{align}
	I_{17}+I_{18} \leq C \tau h^4 +C \tau \| e^n_\eta\|^2_{L^2}+\epsilon_3 \tau \| \nabla \hat{e}^{n+1}_\u\|^2_{L^2}-2\tau\big(\eta^n_h \nabla e_\eta^{n+1},\hat{e}^{n+1}_\u\big).
\end{align}

Furthermore
\begin{align}\label{biogu-21}
	I_{19}+I_{20}=&-2\tau(\sigmaa^{n} (\nabla \cdot \sigmaa^{n+1}),\hat{e}^{n+1}_\u)
	+2\tau(\sigmaa^{n}_h (\nabla \cdot \sigmaa_h^{n+1}),\hat{e}^{n+1}_\u)\\
	=&-2\tau \big( (\sigmaa^n-\sigmaa^n_h) \hat{e}^{n+1}_\u,\nabla \cdot \sigmaa^{n+1}  \big)  -2\tau \big(\sigmaa^n_h \hat{e}^{n+1}_\u,\nabla \cdot (\sigmaa^{n+1}-\sigmaa^{n+1}_h) \big) \notag \\
	\leq &C \tau\|e^n_\sigmaa+\theta^n_\sigmaa\|_{L^2}\|\hat{e}^{n+1}_\u\|_{L^6} \|\nabla \cdot \sigmaa^{n+1}\|_{L^3}\notag\\
	&-2\tau \big( \sigmaa^n_h \hat{e}^{n+1}_\u,\nabla \cdot e^{n+1}_\sigmaa \big)-2\tau \big( \sigmaa^n_h \hat{e}^{n+1}_\u,\nabla \cdot \theta^{n+1}_\sigmaa \big)\notag\\
	\leq & C \tau h^4+C \tau \|e^n_\sigmaa\|^2_{L^2}+ \epsilon_3\tau \| \nabla \hat{e}^{n+1}_\u\|^2_{L^2}\notag\\
	&-2\tau \big( \sigmaa^n_h \hat{e}^{n+1}_\u,\nabla \cdot e^{n+1}_\sigmaa \big)-2\tau \big( \sigmaa^n_h \hat{e}^{n+1}_\u,\nabla \cdot \theta^{n+1}_\sigmaa \big).\notag
\end{align}
where by using IBP, (\ref{biogu-sigma-infty}) and (\ref{biogu-sigma-w13}), it follows that
\begin{align}\label{biogu-22}
	-2\tau \big( \sigmaa^n_h \hat{e}^{n+1}_\u,\nabla \cdot \theta^{n+1}_\sigmaa \big)
	=&2\tau \big(\nabla \sigmaa^n_h \hat{e}^{n+1}_\u, \theta^{n+1}_\sigmaa \big)
	+2\tau \big( \sigmaa^n_h \nabla \hat{e}^{n+1}_\u, \theta^{n+1}_\sigmaa \big)\\
	\leq& C \tau \| \nabla \sigmaa^n_h \|_{L^3}\|\hat{e}^{n+1}_\u\|_{L^6}\|\theta^{n+1}_\sigmaa\|_{L^2}+C\tau\|  \sigmaa^n_h \|_{L^{\infty}}\| \nabla \hat{e}^{n+1}_\u\|_{L^2}\|\theta^{n+1}_\sigmaa\|_{L^2}\notag\\
	\leq &C \tau h^4+\epsilon_3\tau\| \nabla \hat{e}^{n+1}_\u\|_{L^2}^2.\nonumber
\end{align}
thus, combining (\ref{biogu-21}) and (\ref{biogu-22}), is follows that
\begin{align}
	I_{19}+I_{20}\leq C \tau h^4+C \tau \|e^n_\sigmaa\|^2_{L^2}+\epsilon_3 \tau \| \nabla \hat{e}^{n+1}_\u\|^2_{L^2}-2\tau \big( \sigmaa^n_h \hat{e}^{n+1}_\u,\nabla \cdot e^{n+1}_\sigmaa \big).
\end{align}
\begin{align}
	I_{21}=&2\tau(R^{n+1}_\u,\hat{e}^{n+1}_\u)
	\leq C \tau  \|R^{n+1}_\u\|^2_{L^2}+\epsilon_3 \tau \| \nabla \hat{e}^{n+1}_\u\|^2_{L^2}.
\end{align}

Thanks to (\ref{biogu-regularity}), (\ref{biogu-6}) and (\ref{biogu-divergencefree}), it infers that
\begin{align}
	I_{22}=2\tau(\nabla p^{n+1},\hat{e}^{n+1}_\u)=&2\tau(\nabla p^{n+1},e^{n+1}_\u+\nabla \rho^{n+1}_h)\\
	=&2\tau((\nabla p^{n+1},e^{n+1}_\u)+2\tau(\nabla p^{n+1},\nabla \rho^{n+1})\notag\\
	\leq& C\tau^2 \|\nabla p^{n+1}\|_{L^2}^2+\epsilon_4\|\nabla\rho^{n+1}\|^2_{L^2}.\notag
\end{align}

For sufficiently small $\epsilon_3,\epsilon_4$, combining the above inequalities with (\ref{biogu-errorequation33}), it follows that
\begin{align}\label{biogu-u-errorr-esult}
	&\| e^{n+1}_\u\|^2_{L^2}-\| e^n_\u\|^2_{L^2}+\|  e^{n+1}_\u- e^{n}_\u\|^2_{L^2}+\mu_3 \tau \| \nabla e^{n+1}_\u\|^2_{L^2}+ \|  \nabla \rho^{n+1}_h\|^2_{L^2}\\
	\leq & C \tau h^4 \| D_{\tau} \u^{n+1}\|^2_{L^2}+C \tau h^4 +C \tau \big(\| e^n_\u\|^2_{L^2}+ \| e^n_\eta\|^2_{L^2}+\|e^n_\sigmaa\|^2_{L^2}\big)
	+\epsilon_1\tau \| \nabla e^{n+1}_\eta\|^2_{L^2}\notag\\
	-&2\tau\big(\eta^n_h \nabla e_\eta^{n+1},\hat{e}^{n+1}_\u\big)-2\tau \big( \sigmaa^n_h \hat{e}^{n+1}_\u,\nabla \cdot e^{n+1}_\sigmaa \big)+C \tau  \|R^{n+1}_\u\|^2_{L^2}+C\tau^2 \|\nabla p^{n+1}\|_{L^2}^2.\notag
\end{align}

Furthermore, for sufficiently small $\epsilon_1,\epsilon_2,\epsilon_3$, summing up (\ref{biogu-eta-errorr-esult}), (\ref{biogu-sigma-errorr-esult}) and (\ref{biogu-u-errorr-esult}), it yields that
\begin{align}\label{biogu-error-result}
	&\| e^{n+1}_\eta\|^2_{L^2}-\| e^n_\eta\|^2_{L^2}+\mu_1 \tau \| \nabla e^{n+1}_\eta\|^2_{L^2}\\
	&+\| e^{n+1}_\sigmaa\|^2_{L^2}-\| e^n_\sigmaa\|^2_{L^2}+\mu_2 \tau \| \nabla \cdot e^{n+1}_\sigmaa\|^2_{L^2}+\mu_2 \tau \| \nabla \times e^{n+1}_\sigmaa\|^2_{L^2}+2\tau \|  e^{n+1}_\sigmaa\|^2_{L^2}\notag\\
	&+\| e^{n+1}_\u\|^2_{L^2}-\| e^n_\u\|^2_{L^2}+\mu_3 \tau \| \nabla e^{n+1}_\u\|^2_{L^2}+ \|  \nabla \rho^{n+1}_h\|^2_{L^2}\notag\\
	&+\big(\|e^{n+1}_\u-e^{n+1}_\u\|^2_{L^2}+\|e^{n+1}_\eta-e^{n+1}_\eta\|^2_{L^2}+\|e^{n+1}_\sigmaa-e^{n+1}_\sigmaa\|^2_{L^2}\big)\notag\\
	\leq& C \tau (\tau^2+h^4) +C\tau^2 \|\nabla p^{n+1}\|_{L^2}^2 +C\tau ( \|R^{n+1}_\eta\|^2_{L^2}+\|R^{n+1}_\sigmaa\|^2_{L^2}+\|R^{n+1}_\sigmaa|^2_{L^2})\notag\\
	&+C \tau h^4 \big(\| D_{\tau} \eta^{n+1}\|^2_{L^2}+\| D_{\tau} \sigmaa^{n+1}\|^2_{L^2}+\| D_{\tau} \u^{n+1}\|^2_{L^2}\big)\notag\\
	&+C \tau \big(\| e^n_\u\|^2_{L^2}+ \| e^n_\eta\|^2_{L^2}+\|e^n_\sigmaa\|^2_{L^2}\big)+C \tau \big(\| e^{n+1}_\eta\|^2_{L^2}+\| e^{n+1}_\sigmaa\|^2_{L^2}\big).\notag
\end{align}

From (\ref{biogu-regularity}),  (\ref{biogu-trunction-error}),
totaling (\ref{biogu-error-result}) across $n=0$ to $n=N-1$, we
obtain
\begin{align}
	\| e&^{N}_\eta\|^2_{L^2}+\| e^{N}_\sigmaa\|^2_{L^2}+\| e^{N}_\u\|^2_{L^2}+\mu_1\tau\sum_{n=0}^{N-1} \| \nabla e^{n+1}_\eta\|^2_{L^2}+\mu_2\tau\sum_{n=0}^{N-1} \| \nabla \cdot e^{n+1}_\sigmaa\|^2_{L^2}\\
	&+\mu_2\tau\sum_{n=0}^{N-1} \| \nabla \times e^{n+1}_\sigmaa\|^2_{L^2}+\mu_3\tau\sum_{n=0}^{N-1} \| \nabla e^{n+1}_\u\|^2_{L^2}+2\tau \sum_{n=0}^{N-1} \| e^{n+1}_\sigmaa\|^2_{L^2}+ \sum_{n=0}^{N-1}\|  \nabla \rho^{n+1}_h\|^2_{L^2}\notag\\
	&+ \sum_{n=0}^{N-1} \big(\|e^{n+1}_\u-e^{n+1}_\u\|^2_{L^2}+\|e^{n+1}_\eta-e^{n+1}_\eta\|^2_{L^2}+\|e^{n+1}_\sigmaa-e^{n+1}_\sigmaa\|^2_{L^2}\big)\notag\\
	\leq & C \tau \sum_{n=0}^{N-1} \big(\| e^n_\u\|^2_{L^2}+ \| e^n_\eta\|^2_{L^2}+\|e^n_\sigmaa\|^2_{L^2}\big)+C \tau \sum_{n=0}^{N-1} \big(\| e^{n+1}_\eta\|^2_{L^2}+\| e^{n+1}_\sigmaa\|^2_{L^2}\big)+C(\tau+h^4).\notag
\end{align}

Using the Gronwall lemma, there exists $C_1>0$ such that
\begin{align}\label{biogu-i=N}
	\| e&^{n+1}_\eta\|^2_{L^2}+\| e^{n+1}_\sigmaa\|^2_{L^2}+\| e^{n+1}_\u\|^2_{L^2}+\sum_{n=0}^{N-1}\|  \nabla \rho^{n+1}_h\|^2_{L^2}\\
	&+\tau\sum_{n=0}^{N-1}\big(\mu_1 \| \nabla e^{n+1}_\eta\|^2_{L^2}+\mu_2 \| \nabla \cdot e^{n+1}_\sigmaa\|^2_{L^2} +\mu_2 \| \nabla \times e^{n+1}_\sigmaa\|^2_{L^2} +\mu_3 \| \nabla e^{n+1}_\u\|^2_{L^2}  \big)\notag\\
	&+ \sum_{n=0}^{N-1} \big(\|e^{n+1}_\u-e^{n+1}_\u\|^2_{L^2}+\|e^{n+1}_\eta-e^{n+1}_\eta\|^2_{L^2}+\|e^{n+1}_\sigmaa-e^{n+1}_\sigmaa\|^2_{L^2}\big)\notag\\
	\leq &C_1(\tau+h^4).\notag
\end{align}

Thus, the proof of Theorem \ref{biogu-error-lemma} is done.
\end{proof}

\begin{remark}
	It is observed that the optimal convergence rate is attained at the initial time step, whereas the rate at the final time step becomes suboptimal. 
	The reduced convergence rate stated in Theorem~\ref{biogu-error-lemma} is solely attributed to the presence of the term 
	$C\tau^{2}\|\nabla p^{n+1}\|_{L^{2}}^{2}$. 
	To recover optimal convergence, this term would need to be removed. 
	To address this issue rigorously, we will employ a duality argument in the subsequent analysis, following the methodology developed in 
	\cite{pyo2005, pyo2006, wangzhaowei2025}.
	Nevertheless, this suboptimal estimate is indispensable for controlling the convection term in the subsequent analysis, 
	which in turn enables us to derive the optimal overall convergence order. 
\end{remark}

\subsection{Optimal error estimates}
In this section, we will use the preliminary estimates from the previous subsection to
provide an optimal error estimates.

We consider the stationary Stokes equations which will be used in a 
duality argument \cite{galdi2011}:
\begin{subequations}\label{biogu-stokes}
\begin{align}
	-\Delta \mathbf{a} + \nabla d &= \mathbf{g} \quad \text{in } \Omega,\\
	\nabla \cdot  \mathbf{a} &= 0 \quad \text{in } \Omega,\\
	\mathbf{a} &= 0 \quad \text{on } \partial\Omega.
\end{align}
\end{subequations}

The unique solution $(\mathbf{a},d) \in H_0^1(\Omega)\times L_0^2(\Omega)$ 
of the stationary Stokes equations satisfies
\begin{align}\label{biogu-stokes-regularity}
	\|\mathbf{a}\|_{H^2} + \|d\|_{H^1} \le C \|\mathbf{g}\|_{L^2}.
\end{align}

This above assumption is valid under the condition that $\partial \Omega$ is of class $C^2$ or $\Omega$ is a convex polygonal domain \cite{constantin1988,dauge1989},
and the following equation holds \cite{wangzhaowei2025}:
\begin{align}\label{biogu-dualinequality}
	\|\mathbf{g}\|_* \leq \sup_{\mathbf{v} \in \V} \frac{(\mathbf{g}, \mathbf{v})}{\|\mathbf{v}\|_{H^1}} = \sup_{\mathbf{v} \in \V} \frac{(-\Delta \mathbf{a}, \mathbf{v})}{\|\mathbf{v}\|_{H^1}} \leq \|\nabla \mathbf{a}\|_{L^2}.
\end{align}

Now let $(\mathbf{a}_h, d_h) \in \Y_h \times M_h$ indicate the finite element solution 
of (\ref{biogu-stokes}), namely,
\begin{subequations}\label{biogu-stokes-projection}
	\begin{align}
		(\nabla \mathbf{a}_h , \nabla \mathbf{v}_h ) 
		- ( d_h , \operatorname{div} \mathbf{v}_h )
		&= ( \mathbf{g}, \mathbf{v}_h )
		\quad \forall\, \mathbf{v}_h \in \Y_h,\\
		( q_h , \operatorname{div} \mathbf{a}_h ) &= 0
		\quad \forall\, q_h \in M_h .
	\end{align}
\end{subequations}

\begin{lemma}\label{biogu-stokes-lemma} 
	Let $(\mathbf{a}, d) \in H_0^1(\Omega) \times L_0^2(\Omega)$ be the solutions of (\ref{biogu-stokes}) 
	and $(\mathbf{a}_h, d_h) = \mathcal{S}_h(\mathbf{a}, d) \in Y_h \times M_h$ 
	be the Stokes projections defined by (\ref{biogu-stokes-projection}), respectively. Then \cite{scott2008,galdi2011}
	\begin{equation}
		\|\mathbf{a} - \mathbf{a}_h\|_{L^2} 
		+ h \|\mathbf{a} - \mathbf{a}_h\|_{H^1} 
		+ h \|d - d_h\|_{L^2}
		\le C h^{2}\bigl( \|\mathbf{a}\|_{H^2} + \|d\|_{H^1} \bigr).
	\end{equation}
\end{lemma}

If (\ref{biogu-stokes-regularity}) also holds, then the right-hand side is bounded by $Ch^{2}\|\mathbf{g}\|_{L^2}$ and
\begin{align}
	\|\mathbf{g}\|_{*} \le C \|\nabla \mathbf{a}\|_{L^2} \le Ch \|\mathbf{g}\|_{L^2} + C \|\nabla \mathbf{a}_{h}\|_{L^2},\\
	|||\mathbf{a} - \mathbf{a}_{h}|||:= 
	\|\mathbf{a} - \mathbf{a}_{h}\|_{L^{\infty}}
	+ \|\nabla(\mathbf{a} - \mathbf{a}_{h})\|_{L^{3}}
	\le C \|\mathbf{g}\|_{L^2}.
\end{align}

Next, we will give the optimal error estimates by the following theorem.
\begin{theorem}\label{biogu-theorem-optimal}
	 	Let $(\eta^i, \sigmaa^i, \mathbf{u}^i)$ denote the exact solutions of the continuous system \eqref{biosav-rremodel}, and let $(\eta_h^i, \sigmaa_h^i, \mathbf{u}_h^i)$ be the corresponding discrete solutions obtained from the numerical scheme \eqref{biogu-1}--\eqref{biogu-7}.  The following error estimates hold:
	\begin{align}\label{bogu-optimal}
		\|e^{n+1}_\u\|_*+\|e^{n+1}_\eta\|^2_{L^2}+\|e^{n+1}_\sigmaa\|^2_{L^2}+\| \nabla \a_h^{n+1}\|^2_{L	^2}\leq C(\tau^2+h^4).
	\end{align}
\end{theorem}

 Let $(\a^{n+1},d^{n+1})$ and $(\a_h^{n+1},d_h^{n+1})$ be the solutions of Stokes equations (\ref{biogu-stokes}) and (\ref{biogu-stokes-projection}) with $\g=e^{n+1}_\u$. Then Lemma \ref{biogu-stokes-lemma} yields a crucial inequality
 \begin{align}\label{biogu-stokes-reguliarity2}
	\|\a^{n+1}-\a^{n+1}_h\|_{L^2} + h\|\nabla (\a^{n+1}-\a^{n+1}_h )\|_{L^2} +h\|\nabla (d^{n+1}-d^{n+1}_h )\|_{L^2} \le C h^2 \|e^{n+1}_\u\|_{L^2}.
 \end{align}
 
Notice the $\a^{n+1}_h$ is discrete divergence free, then $(\nabla \rho_h^{n+1},\a^{n+1}_h)=0$, so
\begin{align}
	(\hat{e}^{n+1}_\u-e^{n}_\u,\a_h^{n+1})&=(e^{n+1}_\u-e^n_\u+\nabla \rho_h^{n+1},\a_h^{n+1})\\
	&=(e^{n+1}_\u-e^n_\u,\a_h^{n+1})\notag\\
	&=(-\Delta \a_h^{n+1}-\Delta \a_h^{n}+\nabla d_h^{n+1}+\nabla d_h^n,\a_h^{n+1})\notag\\
	&=\frac{1}{2}\big(\|\nabla \a_h^{n+1}\|^2_{L^2} - \|\nabla \a_h^{n}\|^2_{L^2}+ \|\nabla \a_h^{n+1}-\nabla \a_h^n\|^2_{L^2}\big).\notag
\end{align}

From (\ref{biogu-stokes-projection}), one has
\begin{align}
	(\nabla \hat{e}_{\u}^{n+1}, \nabla \a_h^{n+1}) 
= (\hat{e}_{\u}^{n+1}, e_{\u}^{n+1} - \nabla d_h^{n+1}) 
= \|e_{\u}^{n+1}\|_{L^2}^2 - (\hat{e}_{\u}^{n+1}, \nabla d_h^{n+1}).
\end{align}

Taking $\v_h=2\tau \a_h^{n+1} \in \Y_h$ in (\ref{biogu-errorequation3}), we have
\begin{align}\label{biogu-aquation}
		&\|\nabla \a_h^{n+1}\|^2_{L^2} - \|\nabla \a_h^{n}\|^2_{L^2}+ \|\nabla \a_h^{n+1}-\nabla \a_h^n\|^2_{L^2}+2\mu_3 \tau\|e^{n+1}_\u\|^2_{L^2}\\
	=&2\tau(\hat{e}_{\u}^{n+1}, \nabla d_h^{n+1})+2\tau(R^{n+1}_\u,\a_h^{n+1}) -2\tau(D_{\tau} \theta_\u^{n+1},\a_h^{n+1})+2\tau ( p^{n+1},\nabla \cdot \a^{n+1}_h)\notag\\
	&-2\tau b(\u^n,  \u^{n+1},\a_h^{n+1})+2\tau b(\u^n_h,  \hat{\u}^{n+1}_h,\a_h^{n+1})-2\tau(\eta^{n} \nabla \eta^{n+1},\a_h^{n+1})\notag\\
	&+2\tau(\eta^{n}_h \nabla \eta^{n+1}_h,\a_h^{n+1})-2\tau(\sigmaa^{n} (\nabla \cdot \sigmaa^{n+1}),\a_h^{n+1})+2\tau(\sigmaa^{n}_h (\nabla \cdot \sigmaa_h^{n+1}),\a_h^{n+1}).\notag
\end{align}

By using (\ref{biogu-6}), (\ref{biogu-stokes-regularity}) and (\ref{biogu-stokes-reguliarity2}), we can deduce
\begin{align}
	2\tau(\hat{e}_{\u}^{n+1}, \nabla d_h^{n+1}) =&2\tau (\nabla \rho_h^{n+1},\nabla d_h^{n+1})\\
	=&2\tau(\nabla \rho^{n+1}_h,\nabla(d_h^{n+1}-d^{n+1})+\nabla d^{n+1})\notag\\
	\leq &C \tau \| \nabla \rho^{n+1}_h\|^2_{L^2}+\epsilon_5 \tau \| e^{n+1}_\u\|^2_{L^2}+C\tau h^2 \| e^{n+1}_\u\|^2_{L^2}.\notag
\end{align}

With the application of the H\"{o}lder inequality and Young inequalities, we can get
\begin{align}
	2\tau(R^{n+1}_\u,\a_h^{n+1})
	\leq &C \tau  \|R^{n+1}_\u\|^2_{L^2}+C \tau \| \nabla \a_h^{n+1}\|^2_{L^2}.\\
	-2\tau(D_{\tau} \theta_\u^{n+1},\a_h^{n+1})
	\leq &C \tau \| D_{\tau} \theta_\u^{n+1}\|^2_{L^2}+C \tau \| \nabla \a_h^{n+1}\|^2_{L^2}\\
	\leq &C \tau  h^4\| D_{\tau} \u^{n+1}\|^2_{L^2}+C \tau \| \nabla \a_h^{n+1}\|^2_{L^2}.\notag
\end{align}

Next we notice that $\a^{n+1}_h$
is discrete divergence free and $\a^{n+1}$ is divergence free, and by using  Lemma \ref{biogu-stokes-lemma}, the projection error (\ref{biogu-projection-error}) and (\ref{biogu-stokes-regularity}), we can derive
\begin{align}
	2\tau(p^{n+1},\a^{n+1}_h)=&2\tau(p^{n+1}-P_{2} p^{n+1},\nabla \cdot (\a^{n+1}_h-\a^{n+1}))\\
	\leq &C\tau h \| p^{n+1}-P_{2} p^{n+1}\|_{L^2} \|\a^{n+1}\|_{H^2}\notag\\
	\leq & C \tau h^4 +\epsilon_5\tau \|e^{n+1}_\u\|^2_{L^2}.
\end{align}

After a simple rearrangement, by (\ref{biogu-regularity}), (\ref{biogu-u-infty}), projection error (\ref{biogu-projection-error}), we deduce that
\begin{align}
&-2\tau b(\u^n,  \u^{n+1},\a_h^{n+1})+2\tau b(\u^n_h,  \hat{\u}^{n+1}_h,\a_h^{n+1})\\
	=& - 2\tau b(\u^n-\u^n_h,\u^{n+1},\a_h^{n+1})-2\tau b(\u^n_h,\u^{n+1}-\hat{\u}^{n+1}_h,\a_h^{n+1})\notag\\
	=&- 2\tau b(\u^n-\u^n_h,\u^{n+1},\a_h^{n+1})+2\tau b(\u^n_h, \hat{e}^{n+1}_\u, \a_h^{n+1})\notag\\
	=&- 2\tau b(\theta^n_\u,\u^{n+1},\a_h^{n+1})- 2\tau b(e^n_\u-e^{n+1}_\u,\u^{n+1},\a_h^{n+1})\notag\\
	&- 2\tau b(e^{n+1}_\u,\u^{n+1},\a_h^{n+1})+2\tau b(\u^n_h, \hat{e}^{n+1}_\u, \a_h^{n+1})\notag\\
	\leq &C \tau h^4 +C \tau \| \nabla \a_h^{n+1}\|^2_{L^2}+C\tau \| \nabla \rho^{n+1}_h\|^2_{L^2}\notag\\
	&+\epsilon_5\tau \| e^{n+1}_\u\|^2_{L^2}+C\tau\|e^{n+1}_\u-e^n_\u\|^2_{L^2}.\notag
\end{align}

After a simple rearrangement, in terms of  (\ref{biogu-regularity}), (\ref{biogu-sigma-infty}), and projection error (\ref{biogu-projection-error}), we have
\begin{align}\label{biogu-0017}
	&-\big(2 \tau(\eta^{n} \nabla \eta^{n+1},\a_h^{n+1})-2\tau(\eta^{n}_h \nabla \eta^{n+1}_h,\a_h^{n+1})\big)\\
	\leq &2 \tau|\big(   (\eta^n-\eta^n_h) \nabla \eta^{n+1}, \a_h^{n+1}\big)|-2\tau\big(\eta^n_h \nabla(\eta^{n+1}-\eta^{n+1}_h),\a_h^{n+1}\big)\notag\\
	\leq &C \tau \|e^n_\eta+\theta^n_\eta\|_{L^2}\|\nabla \eta^{n+1}\|_{L^{3}} \|\a_h^{n+1}\|_{L^6}\notag\\
	&+2\tau|\big(\eta^n_h \nabla \theta_\eta^{n+1},\a_h^{n+1}\big)| -C\tau \|\eta^n_h\|_{L^\infty} \| \nabla e_\eta^{n+1}\|_{L^2}\|\a_h^{n+1}\|_{L^2}\notag\\
	\leq &C \tau h^4 +C \tau \| e^n_\eta\|^2_{L^2}+C \tau \|  \nabla \a_h^{n+1} \|^2_{L^2}+\epsilon_6 \tau \|\nabla e^{n+1}_\eta\|_{L^2}^2\notag\\
	&+2\tau|\big(\eta^n_h \nabla \theta_\eta^{n+1},\a_h^{n+1}\big)|,\notag
\end{align}

By (\ref{biogu-eta-w13}), projection error (\ref{biogu-projection-error}) and integrate by parts (IBP), we can deduce that
\begin{align}\label{biogu-0018}
	2\tau|(\eta^n_h\nabla \theta^{n+1}_\eta,\a_h^{n+1})|=&|-2\tau(\nabla \eta^n_h \theta^{n+1}_\eta,\a_h^{n+1})-2\tau (\eta^n_h \theta^{n+1}_\eta, \nabla \a_h^{n+1})|\\
	\leq &C\tau \| \nabla \eta^n_h\|_{L^3}\|\theta^{n+1}_\eta\|_{L^2} \|\a_h^{n+1}\|_{L^6}+C \tau \| \eta^n_h\|_{L^{\infty}} \| \theta^{n+1}_\eta\|_{L^2}\| \nabla  \a_h^{n+1}\|_{L^2}\notag\\
	\leq & C \tau h^4 + C\tau \| \nabla \a_h^{n+1}\|^2_{L^2},\notag
\end{align}

Combining (\ref{biogu-0017}) and (\ref{biogu-0018}), it follows that
\begin{align}
	&-\big(2 \tau(\eta^{n} \nabla \eta^{n+1},\a_h^{n+1})-2\tau(\eta^{n}_h \nabla \eta^{n+1}_h,\a_h^{n+1})\big)\\
\leq &C \tau h^4 +C \tau \| e^n_\eta\|^2_{L^2}+C \tau \|  \nabla \a_h^{n+1} \|^2_{L^2}+\epsilon_6 \tau \|\nabla e^{n+1}_\eta\|_{L^2}^2.\notag
\end{align}

Furthermore
\begin{align}\label{biogu-0021}
&-2\tau(\sigmaa^{n} (\nabla \cdot \sigmaa^{n+1}),\a_h^{n+1})
	+2\tau(\sigmaa^{n}_h (\nabla \cdot \sigmaa_h^{n+1}),\a_h^{n+1})\\
	=&-2\tau \big( (\sigmaa^n-\sigmaa^n_h) \a_h^{n+1},\nabla \cdot \sigmaa^{n+1}  \big)  -2\tau \big(\sigmaa^n_h \a_h^{n+1},\nabla \cdot (\sigmaa^{n+1}-\sigmaa^{n+1}_h) \big) \notag \\
	\leq &C \tau\|e^n_\sigmaa+\theta^n_\sigmaa\|_{L^2}\|\a_h^{n+1}\|_{L^6} \|\nabla \cdot \sigmaa^{n+1}\|_{L^3}\notag\\
	&-C\tau \|\sigmaa^n_h\|_{L^3} \|\a_h^{n+1}\|_{L^6}\|\nabla \cdot e^{n+1}_\sigmaa\|_{L^2} -2\tau \big( \sigmaa^n_h \a_h^{n+1},\nabla \cdot \theta^{n+1}_\sigmaa \big)\notag\\
	\leq & C \tau h^4+C \tau \|e^n_\sigmaa\|^2_{L^2}+ C\tau \| \nabla \a_h^{n+1}\|^2_{L^2}+\epsilon_7 \tau \|\nabla \cdot e^{n+1}_\sigmaa\|_{L^2}^2 \notag\\
	&-2\tau \big( \sigmaa^n_h \a_h^{n+1},\nabla \cdot \theta^{n+1}_\sigmaa \big).\notag
\end{align}
where by using IBP, (\ref{biogu-sigma-infty}) and (\ref{biogu-sigma-w13}), it follows that
\begin{align}\label{biogu-0022}
	-2\tau \big( \sigmaa^n_h \a_h^{n+1},\nabla \cdot \theta^{n+1}_\sigmaa \big)
	=&2\tau \big(\nabla \sigmaa^n_h \a_h^{n+1}, \theta^{n+1}_\sigmaa \big)
	+2\tau \big( \sigmaa^n_h \nabla \a_h^{n+1}, \theta^{n+1}_\sigmaa \big)\\
	\leq& C \tau \| \nabla \sigmaa^n_h \|_{L^3}\|\a_h^{n+1}\|_{L^6}\|\theta^{n+1}_\sigmaa\|_{L^2}+C\tau\|  \sigmaa^n_h \|_{L^{\infty}}\| \nabla \a_h^{n+1}\|_{L^2}\|\theta^{n+1}_\sigmaa\|_{L^2}\notag\\
	\leq &C \tau h^4+C\tau\| \nabla \a_h^{n+1}\|_{L^2}^2.\nonumber
\end{align}
thus, combining (\ref{biogu-0021}) and (\ref{biogu-0022}), is follows that
\begin{align}
&-2\tau(\sigmaa^{n} (\nabla \cdot \sigmaa^{n+1}),\a_h^{n+1})
+2\tau(\sigmaa^{n}_h (\nabla \cdot \sigmaa_h^{n+1}),\a_h^{n+1})\\
=&C \tau h^4+C \tau \|e^n_\sigmaa\|^2_{L^2}+ C\tau \| \nabla \a_h^{n+1}\|^2_{L^2}+\epsilon_7 \tau \|\nabla \cdot e^{n+1}_\sigmaa\|_{L^2}^2.\notag
\end{align}
Substituting the above estimates into (\ref{biogu-aquation}), we can derive that
\begin{align}\label{biogu-31}
	&\|\nabla \a_h^{n+1}\|^2_{L^2} - \|\nabla \a_h^{n}\|^2_{L^2}+ \|\nabla \a_h^{n+1}-\a_h^n\|^2_{L^2}+\mu_3 \tau\|e^{n+1}_\u\|^2_{L^2}\\
	\leq & C \tau \| \nabla \a_h^{n+1}\|_{L^2}^2+ C\tau h^4+C\tau \| \nabla \rho^{n+1}_h\|^2_{L^2}+C\tau\| R^{n+1}_\u\|^2_{L^2}\notag\\
	&+C\tau h^4 \| D_{\tau}\u^{n+1}\|^2_{L^2}+C\tau\|e^{n+1}_\u-e^n_\u\|^2_{L^2}+C \tau ( \|e^n_\eta\|^2_{L^2}+\|e^n_\sigmaa\|^2_{L^2})\notag\\
	&+\epsilon_5\tau \|e^{n+1}_\u\|^2_{L^2}+C\tau h^2\|e^{n+1}_\u\|^2_{L^2}+ \epsilon_6\tau \|\nabla e^{n+1}_\eta\|^2_{L^2}+\epsilon_7 \tau \|\nabla \cdot e^{n+1}_\sigmaa\|_{L^2}^2.\notag
\end{align}

By using (\ref{biogu-6}), H\"{o}lder inequality and Young inequality, we have
\begin{align}\label{biogu-32}
	&2\tau\big(\eta^n_h \hat{e}_\u^{n+1},\nabla e^{n+1}_\eta\big)+2\tau\big(\sigmaa^n_h \hat{e}^{n+1}_\u,\nabla \cdot e^{n+1}_\sigmaa\big)\\
	\leq & C\tau \| \nabla \rho^{n+1}_h\|^2_{L^2}+   \epsilon_5\tau \|e^{n+1}_\u\|^2_{L^2}+\epsilon_6\tau \|\nabla e^{n+1}_\eta\|^2_{L^2}+\epsilon_7 \tau \|\nabla \cdot e^{n+1}_\sigmaa\|_{L^2}^2.   \notag
\end{align}

Setting $(r_h,\w_h)=2\tau(e^{n+1}_\eta,e^{n+1}_\sigmaa)$ in (\ref{biogu-errorequation11}), (\ref{biogu-errorequation22}) and combining (\ref{biogu-eta-errorr-esult}), (\ref{biogu-sigma-errorr-esult}), (\ref{biogu-31}) and (\ref{biogu-32}), for sufficiently small $h, \epsilon_5,\epsilon_6,\epsilon_7$, we can have
\begin{align}
	&\|\nabla \a_h^{n+1}\|^2_{L^2} - \|\nabla \a_h^{n}\|^2_{L^2}+ \|\nabla \a_h^{n+1}-\nabla\a_h^n\|^2_{L^2}+\mu_3 \tau\|e^{n+1}_\u\|^2_{L^2}\\
	&+\| e^{n+1}_\eta\|^2_{L^2}-\| e^n_\eta\|^2_{L^2}+\|  e^{n+1}_\eta- e^{n}_\eta\|^2_{L^2}+\mu_1 \tau \| \nabla e^{n+1}_\eta\|^2_{L^2}\notag\\
	&+\| e^{n+1}_\sigmaa\|^2_{L^2}-\| e^n_\sigmaa\|^2_{L^2}+\|  e^{n+1}_\sigmaa- e^{n}_\sigmaa\|^2_{L^2}\notag\\
	&+\mu_2 \tau \| \nabla \cdot e^{n+1}_\sigmaa\|^2_{L^2}+\mu_2 \tau \| \nabla \times e^{n+1}_\sigmaa\|^2_{L^2}+2\tau \|  e^{n+1}_\sigmaa\|^2_{L^2}\notag\\
	\leq &	
  C\tau h^4+C\tau \| \nabla \rho^{n+1}_h\|^2_{L^2}+C\tau\|e^{n+1}_\u-e^n_\u\|^2_{L^2}\notag\\
 &+C \tau( \| \nabla \a_h^{n+1}\|_{L^2}^2+ \| e^n_\sigmaa\|^2_{L^2}+\| e^n_\eta\|^2_{L^2}+\| e^{n+1}_\eta\|^2_{L^2}+\| e^{n+1}_\sigmaa\|^2_{L^2})\notag\\
 &+C \tau h^4( \| D_{\tau} \sigmaa^{n+1}\|^2_{L^2}+D_{\tau}\u^{n+1}\|^2_{L^2}+D_{\tau}\eta^{n+1}\|^2_{L^2})\notag\\
 &+C\tau( \|R^{n+1}_\sigmaa\|^2_{L^2}+\|R^{n+1}_\eta\|^2_{L^2}+\|R^{n+1}_\u\|^2_{L^2})\notag
\end{align}

Summing up from $n=1$ to $N-1$ and by using (\ref{biogu-trunction-error}), (\ref{biogu-i=N}), we can deduce
\begin{align}
	&\| \nabla \a_h^{n+1}\|^2_{L	^2}+\|e^{n+1}_\eta\|^2_{L^2}+\|e^{n+1}_\sigmaa\|^2_{L^2}+\mu_{1}\tau\sum_{n=0}^{N-1} \| \nabla e^{n+1}_\eta\|^2_{L^2}+\mu_{3}\tau\sum_{n=0}^{N-1} \|  e^{n+1}_\u\|^2_{L^2}\\
	&+\mu_{2}\tau\sum_{n=0}^{N-1} \| \nabla \cdot e^{n+1}_\sigmaa\|^2_{L^2}+\mu_{2}\tau\sum_{n=0}^{N-1} \| \nabla \times e^{n+1}_\sigmaa\|^2_{L^2}+2\tau \sum_{n=0}^{N-1} \|  e^{n+1}_\sigmaa\|^2_{L^2}\notag\\
	\leq & C\tau \sum_{n=0}^{N-1}( \| \nabla \a_h^{n+1}\|_{L^2}^2+\| e^{n+1}_\eta\|^2_{L^2}+\| e^{n+1}_\sigmaa\|^2_{L^2})+C(\tau^2+h^4)+C\tau (\tau+h^4).\notag
\end{align}

By using Gronwall's inequality, we have 
\begin{align}
\| \nabla \a_h^{n+1}\|^2_{L	^2}+\|e^{n+1}_\eta\|^2_{L^2}+\|e^{n+1}_\sigmaa\|^2_{L^2}\leq C(\tau^2+h^4).
\end{align}

According to (\ref{biogu-dualinequality}), one has
\begin{align}
	\| e^{n+1}_\u\|_* \leq \| \nabla \a_h^{n+1}\|_{L^2}.
\end{align}
so we end with (\ref{bogu-optimal}). The proof of Theorem \ref{biogu-theorem-optimal} is completed.
\section{Numerical experiences}

\subsection{Convergence test}\label{biogu-convergence-test}
In this subsection, we first verify the accuracy of the proposed numerical scheme. The computational domain is taken as
\[
\Omega = \{(x,y) \in \mathbb{R}^2 : x \in [0,1], \, y \in [0,1]\},
\]
and the model parameters are set as $\mu_1 = \mu_2 = \mu_3 = 1$. The final simulation time is chosen as $T = 1$, and the exact solutions are prescribed as
\begin{align}
	\begin{cases} 
		\eta(x,y,t) = \cos(2 \pi x) \cos(\pi y) \sin(t), \\ 
			c(x,y,t) = \cos(\pi x) \cos(2 \pi y) \sin(t), \\ 
		\sigmaa(x,y,t) = \nabla c = \begin{pmatrix}  -\pi \sin(\pi x) \cos(2 \pi y) \sin(t) \\ -2 \pi \cos(\pi x) \sin(2 \pi y) \sin(t) \end{pmatrix}, \\ 
		\mathbf{u}(x,y,t) = \begin{pmatrix} \sin(\pi x) \cos(\pi y) \sin(t)\\ -\cos(\pi x) \sin(\pi y) \sin(t) \end{pmatrix}, \\ 
		p(x,y,t) = \cos(\pi x) \cos(\pi y) \sin(t).
	\end{cases}
\end{align}

We denote the $L^2$-norm errors at the final time $t_N$ as
\begin{align*}
	\| r - r_h \|_{L^2} &= \| r(t_N) - r_h^N \|_{L^2}, \\
	\| \mathbf{v} - \mathbf{v}_h \|_{L^2} &= \| \mathbf{v}(t_N) - \mathbf{v}_h^N \|_{L^2}.
\end{align*}

To verify the temporal convergence, the spatial step size is fixed at $h = 1/200$, and the temporal step sizes are varied as $\tau = 1/4, 1/8, \dots, 1/64$. The corresponding temporal $L^2$-norm errors of velocity $\u$, pressure $p$, cell density $\eta$, and chemical concentration $c$ are presented in Fig.~\ref{biogu-table-h200} and Table~\ref{biogu-fig-h200}, indicating that the temporal convergence rates are $\mathcal{O}(\tau)$ for all variables. 

Furthermore, to examine the spatial convergence, the temporal step size is fixed at $\tau = 0.001$, and the spatial step sizes are chosen as $h = 1/4, 1/8, \dots, 1/64$. The $L^2$-norm errors of the velocity $\u$, pressure $p$, cell density $\eta$, and chemical concentration $c$ are displayed in Fig.~\ref{biogu-table-tau1000} and Table~\ref{biogu-fig-tau1000}. These results demonstrate that the spatial convergence rates are $\mathcal{O}(h^2)$ for $\u$, $\eta$, and $c$, and $\mathcal{O}(h)$ for $p$, which are in excellent agreement with the theoretical predictions.

\begin{table}[htbp] 
	\centering  
	\caption{$L^2$ errors and convergence rates of $(\mathbf{u},p,\eta,c)$ wtih $h=200$}   
	\rowcolors{1}{gray!20}{white} 
	\begin{tabular}{ccccccccc}          
		$\tau$ & $\| \u-\u_h\|_{L^2}$ & rate  & $\| p-p_h\|_{L^2}$   & rate  & $\| \eta-\eta_h\|_{L^2}$  & rate  & $\| c-c_h\|_{L^2}$ & rate 
		\\    1/4   & 0.009004 &       & 0.858101 &       & 0.019537 &       & 0.016358 & 
		\\    1/8   & 0.004565 & 0.98  & 0.444491 & 0.95  & 0.009271 & 1.08  & 0.00836 & 0.97  \\    1/16  & 0.00233 & 0.97  & 0.223035 & 0.99  & 0.004508 & 1.04  & 0.004236 & 0.98  \\    1/32  & 0.001106 & 1.08  & 0.109997 & 1.02  & 0.002225 & 1.02  & 0.002138 & 0.99  \\    1/64  & 5.26E-04 & 1.07  & 5.46E-02 & 1.01  & 1.11E-03 & 1.00  & 0.00108 & 0.98  \\    \end{tabular}
	\label{biogu-table-h200}
\end{table}%

\begin{table}[htbp]  
	\centering  
	\caption{$L^2$ errors and convergence rates of $(\mathbf{u},p,\eta,c)$ wtih $\tau=1000$}   
	\rowcolors{1}{gray!20}{white} 
	\begin{tabular}{ccccccccc}          
		h& $\| \u-\u_h\|_{L^2}$ & rate  & $\| p-p_h\|_{L^2}$   & rate  & $\| \eta-\eta_h\|_{L^2}$  & rate  & $\| c-c_h\|_{L^2}$ & rate 
		\\    1/4   & 0.081541 &       & 0.176033 &       & 0.130127 &       & 0.129293 &  
		\\    1/8   & 0.020719 & 1.98  & 0.104231 & 0.76  & 0.039501 & 1.72  & 0.038345 & 1.75  
		\\    1/16  & 0.005172 & 2.00  & 0.054895 & 0.93  & 0.01048 & 1.91  & 0.010105 & 1.92  
		\\    1/32  & 0.001289 & 2.00  & 0.029247 & 0.91  & 0.002678 & 1.97  & 0.002575 & 1.97  \\    1/64  & 3.20E-04 & 2.01  & 1.64E-02 & 0.83  & 6.89E-04 & 1.96  & 0.00066 & 1.96  \\    \end{tabular}  
	\label{biogu-table-tau1000}
\end{table}%

\begin{figure}
	\centering
	\includegraphics[width=0.7\linewidth]{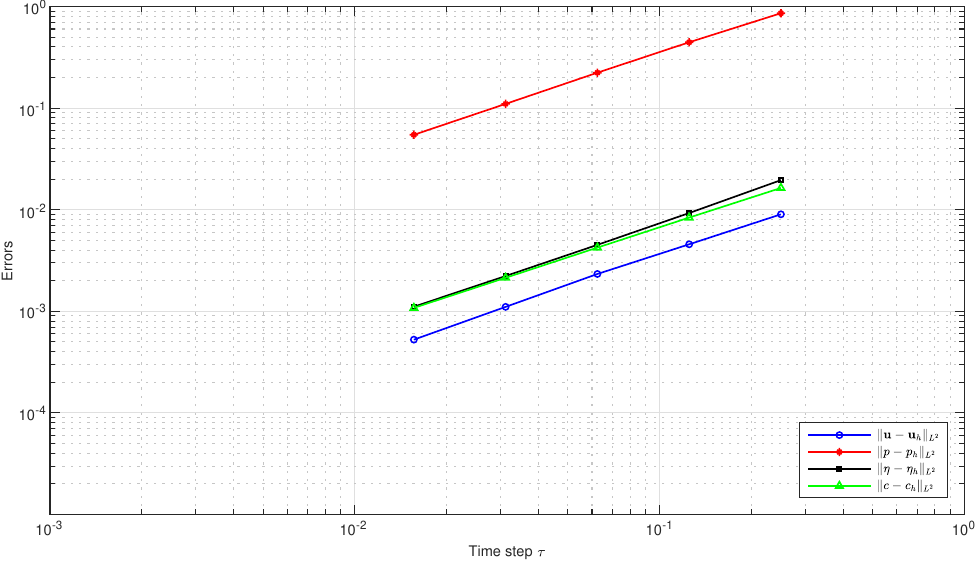}
	\caption{Convergence history of $(\u,p,\eta,c)$ for different $\tau$.}
	\label{biogu-fig-h200}
\end{figure}

\begin{figure}
	\centering
	\includegraphics[width=0.7\linewidth]{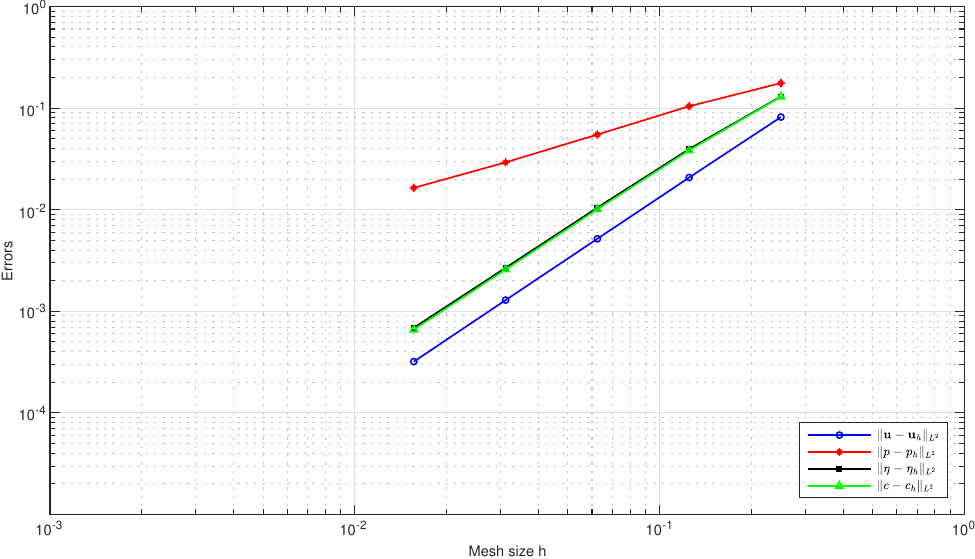}
	\caption{Convergence history of $(\u,p,\eta,c)$ for different $h$.}
	\label{biogu-fig-tau1000}
\end{figure}

\subsection{Stability test}
To numerically validate the unconditional energy stability of our proposed scheme, we systematically monitor the temporal evolution of the discrete free energy functional defined in (\ref{biogu-energy}). The computational setup used in the convergence test of Subsection \ref{biogu-convergence-test} is adopted here, except that the total simulation time is set to $T = 2.5$. Figure \ref{biogu-energy-figure} presents the computed energy trajectories under various initial conditions, including representative cases such as:
\begin{align}
\begin{cases} 
	\eta_0(x,y) = \cos(2\pi x) + \sin(2\pi y) + 3, \\ 
	c_0(x,y) = \cos(2\pi x) + \sin(2\pi y) - 2\pi y + 9, \\ 
	\sigma_0 (x,y)= \nabla c_0 = 2\pi\begin{pmatrix} -\sin(2\pi x) \\ \cos(2\pi y) - 1 \end{pmatrix}, \\ 
	\mathbf{u}_0(x,y) = \begin{pmatrix} \sin(2\pi y)(-\cos(2\pi x + \pi) - 1) \\ \sin(2\pi x)(\cos(2\pi y + \pi) + 1) \end{pmatrix}, \\ 
	p_0(x,y) = \cos(2\pi x) + \sin(2\pi y).
\end{cases}
\end{align}

Remarkably, for all tested parameter configurations, the discrete energy demonstrates a strictly monotonic decay throughout the entire temporal evolution, No spurious oscillations or artificial energy growth are observed at any stage of the simulation. This numerically observed behavior provides compelling evidence that our scheme successfully preserves the essential energy dissipation structure at the discrete level. The unconditional energy stability, which we rigorously established through theoretical analysis, is thus conclusively verified by the numerical experiment.

\begin{figure}
	\centering
	\includegraphics[width=1\linewidth]{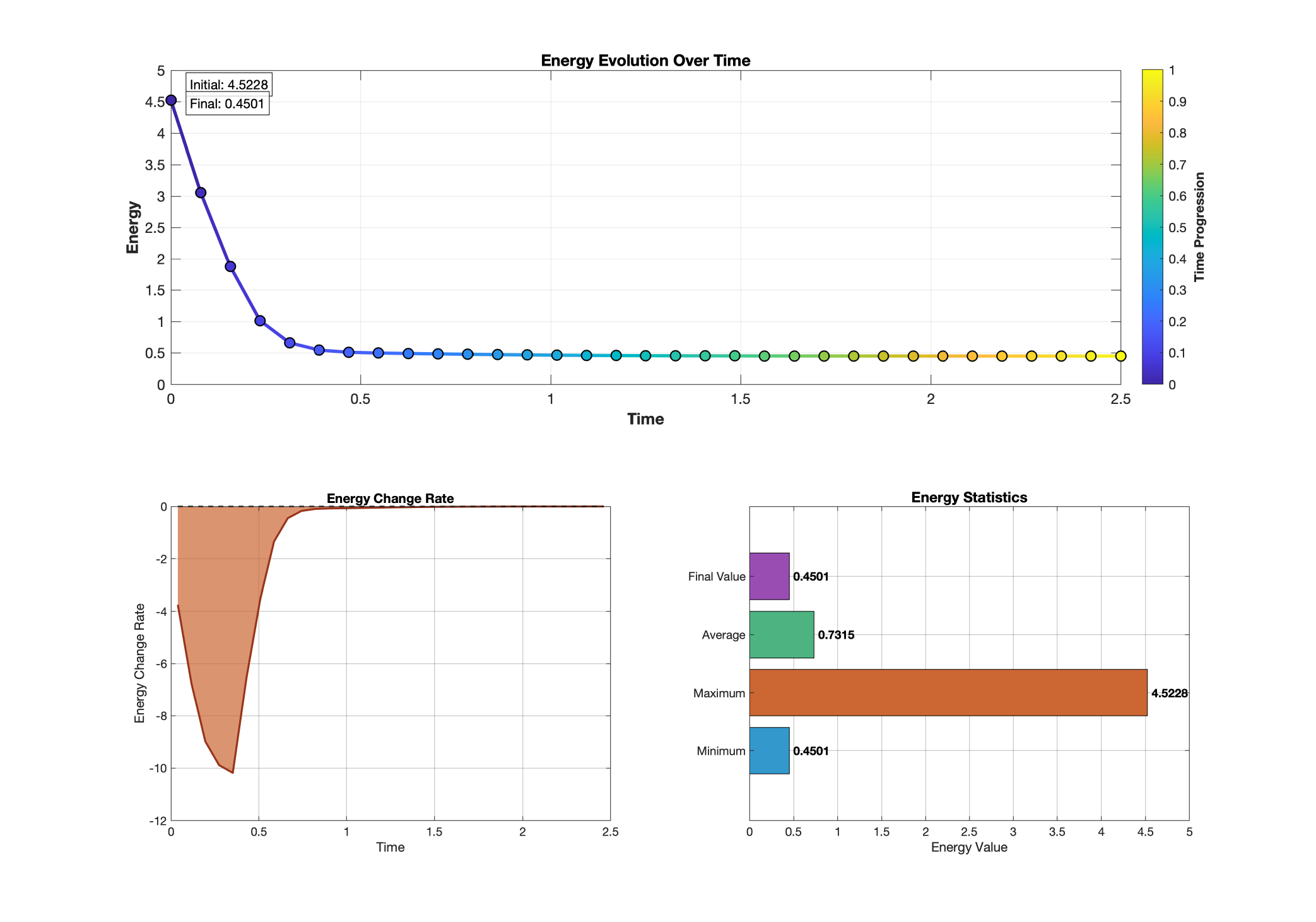}
	\caption{Time evolution of the free energy functional $\mathcal{E}^3 $}
	\label{biogu-energy-figure}
\end{figure}

\subsection{The chemo-repulsion behavior}
We employ previously developed schemes to conduct a series of simulations for chemo–repulsion dynamics.
In this example, we consider the computational domain $\Omega = [0,1] \times [0,1]$. The temporal step size is set to $\tau= 0.0001$, and the spatial mesh size is chosen as $h = 1/100$. The model parameters are taken as $\mu_1 =\mu_2=\mu_3=1$. The initial data are prescribed as follows  \cite{guillen2}:
\begin{align}
	\begin{cases}
		\eta_0(x,y) = -10\, x y (2 - x)(2 - y)\exp\!\left(-10(y-1)^2 - 10(x-1)^2\right) + 10.0001,\\
		c_0(x,y) = 200\, x y (2 - x)(2 - y)\exp\!\left(-30(y-1)^2 - 30(x-1)^2\right) + 0.0001,\\
		\u_0(x,y) = \textbf{0}, \qquad p_0(x,y)=0.
	\end{cases}
\end{align}

The numerical simulations corresponding to this setup are illustrated in Figure \ref{biogu-benckmark2-c} and \ref{biogu-benckmark2-eta}, which displays the temporal evolution of the cell density and the associated chemical concentration at different time step $k$. The results clearly demonstrate that the microorganisms tend to migrate toward regions with lower chemical concentration, highlighting the characteristic chemo-repulsion mechanism. This phenomenon provides valuable insight into the influence of repulsive chemotactic forces on the overall spatial distribution of the cell population.

\begin{figure}[htbp]
	\centering
	\begin{minipage}{0.30\textwidth}
		\centering
		\includegraphics[width=\textwidth]{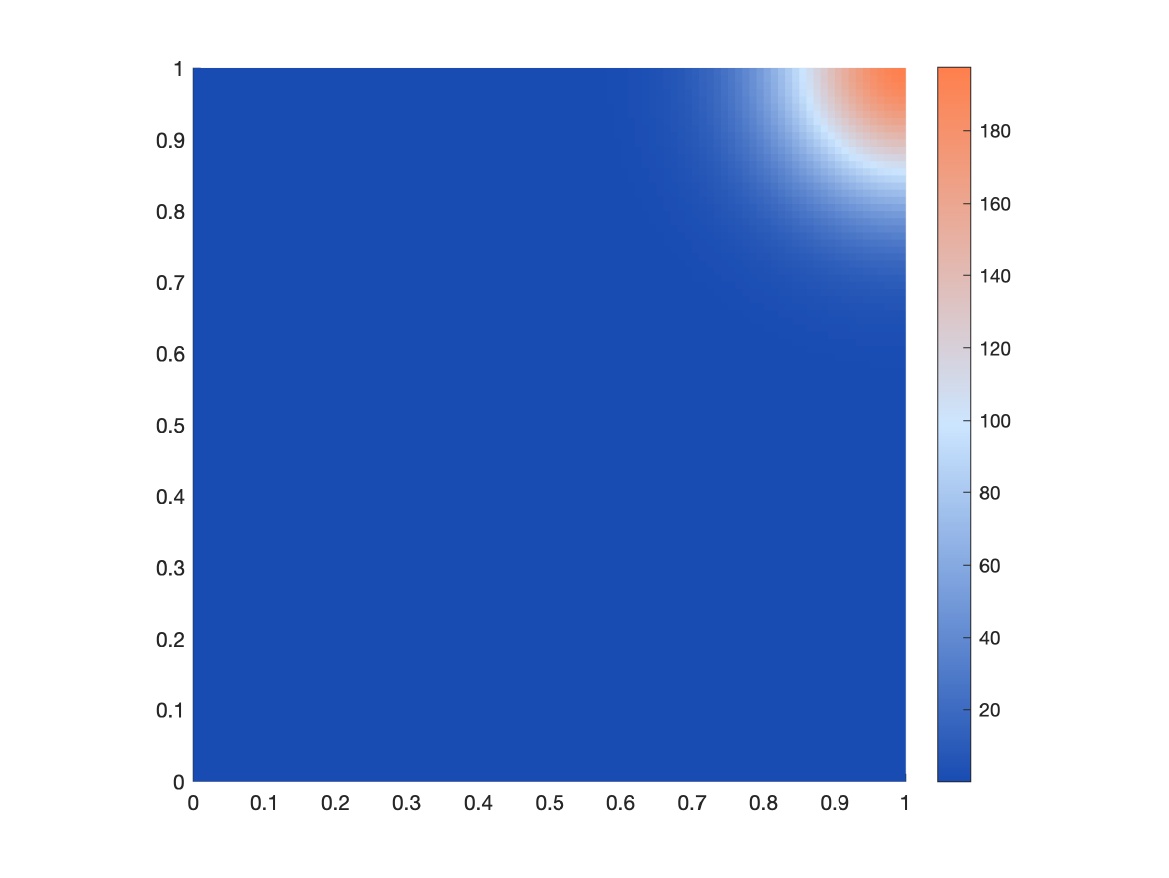}
	\end{minipage}
	\begin{minipage}{0.30\textwidth}
		\centering
		\includegraphics[width=\textwidth]{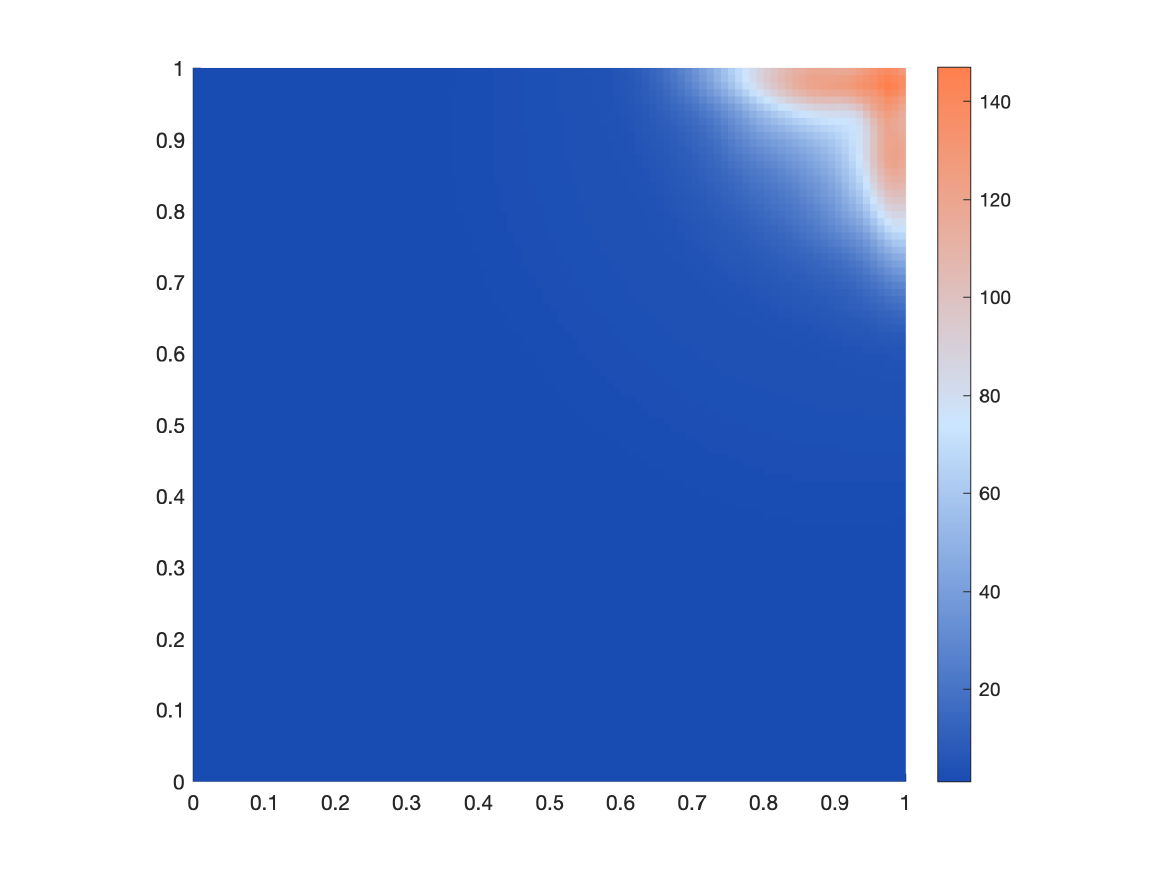}
	\end{minipage}
	\begin{minipage}{0.30\textwidth}
		\centering
		\includegraphics[width=\textwidth]{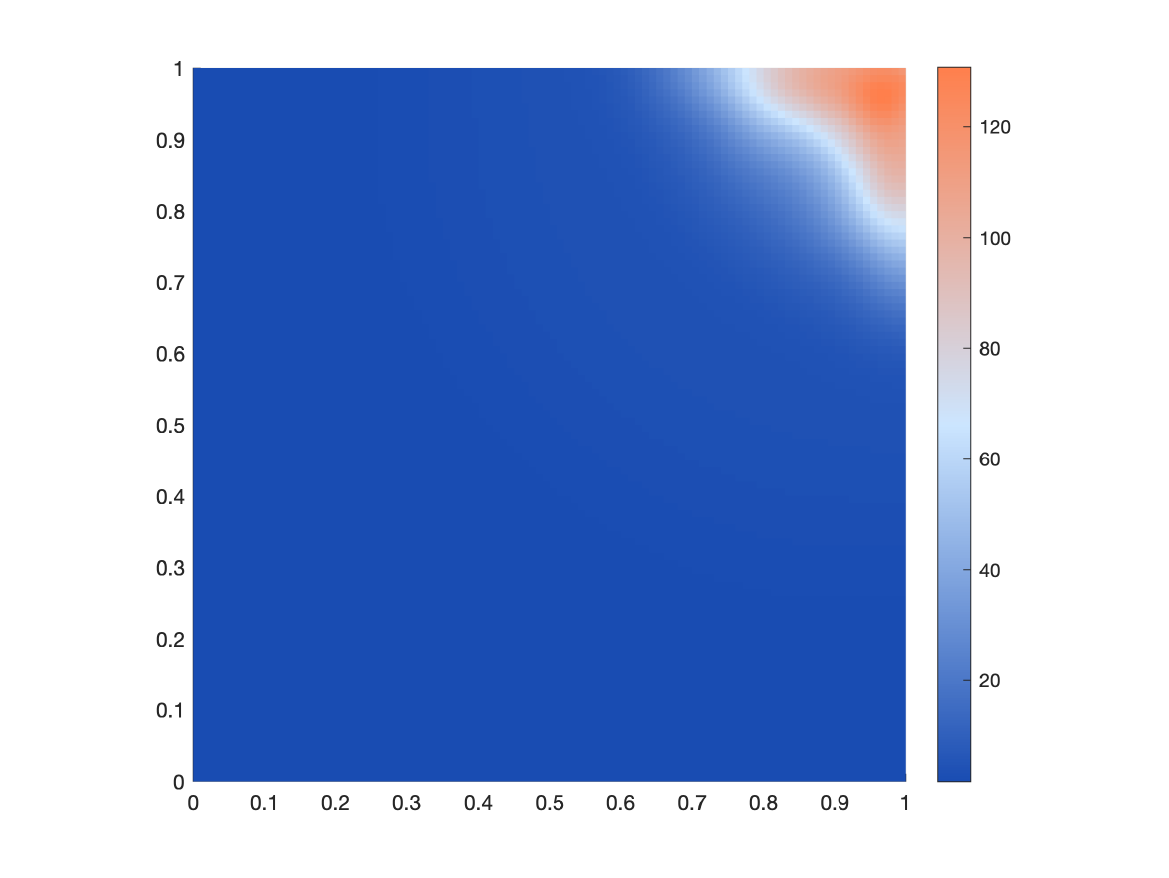}
	\end{minipage}
	
	\vspace{0.8em}
	\begin{minipage}{0.30\textwidth}
		\centering
		\includegraphics[width=\textwidth]{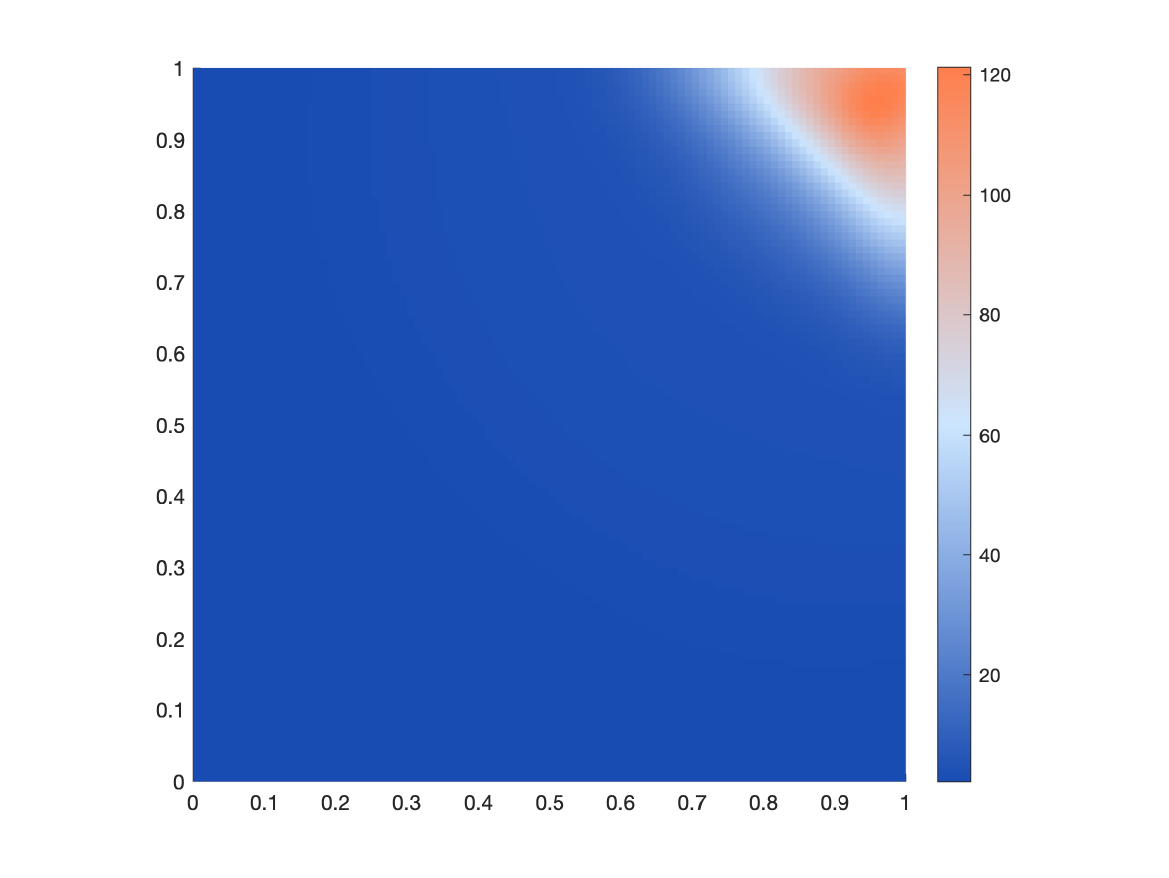}
	\end{minipage}
	\begin{minipage}{0.30\textwidth}
		\centering
		\includegraphics[width=\textwidth]{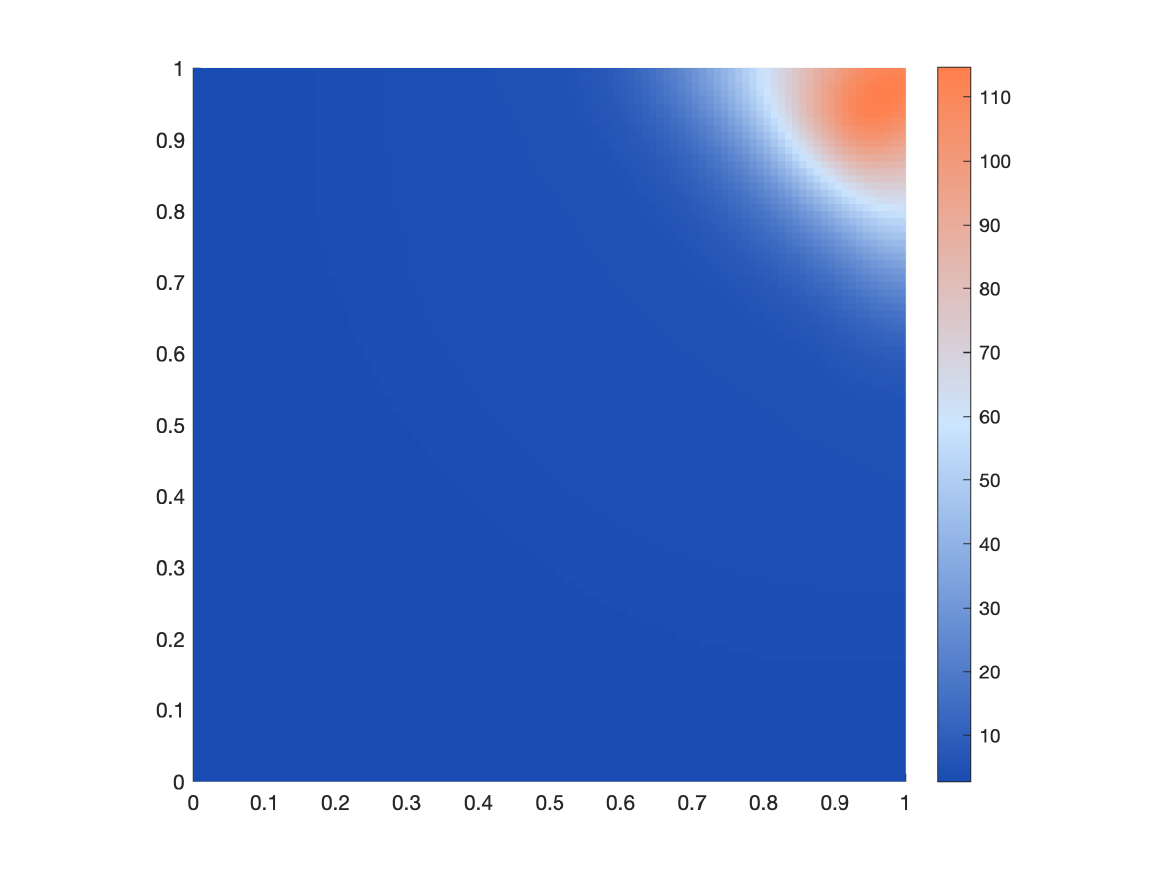}
	\end{minipage}
	\begin{minipage}{0.30\textwidth}
		\centering
		\includegraphics[width=\textwidth]{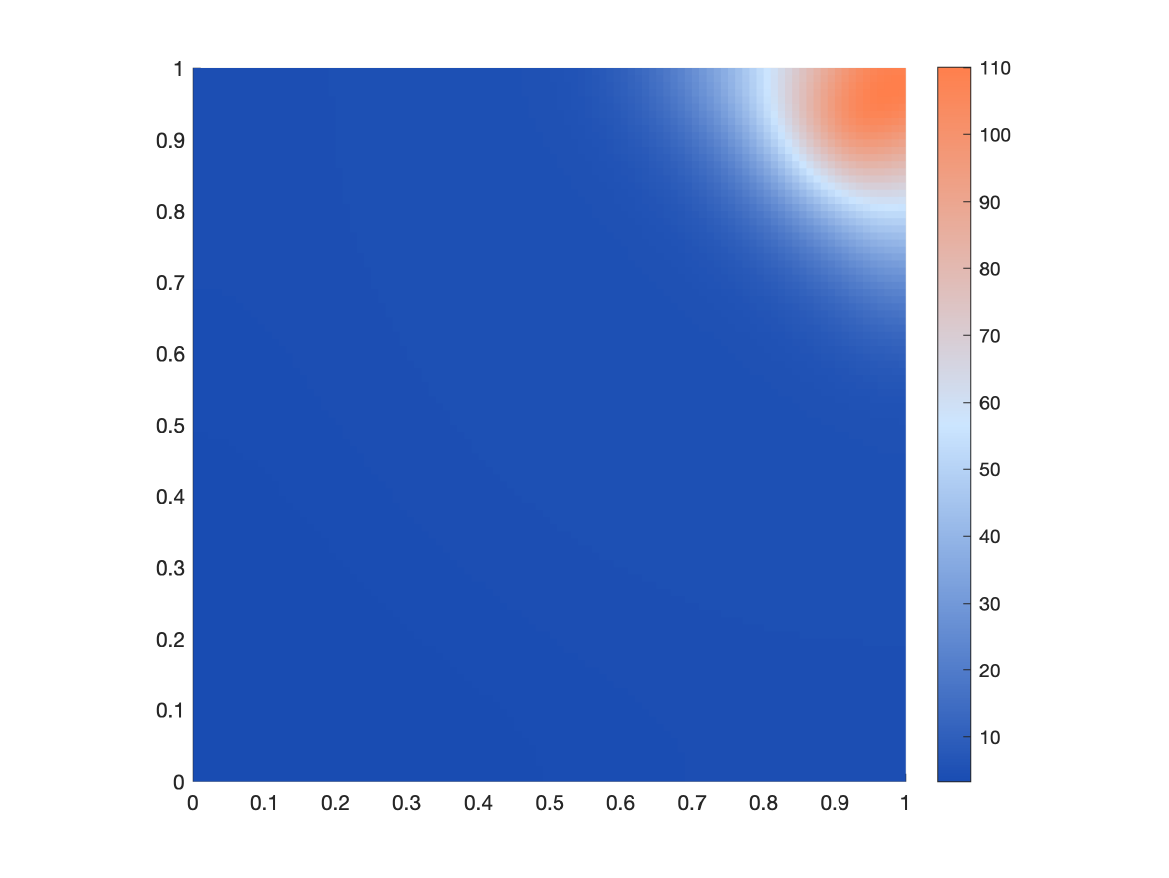}
	\end{minipage}
	
	\caption{Chemical concentration at $k=1,10,15,20,25,30$.}
	\label{biogu-benckmark2-c}
\end{figure}

\begin{figure}[htbp]
	\centering
	\begin{minipage}{0.30\textwidth}
		\centering
		\includegraphics[width=\textwidth]{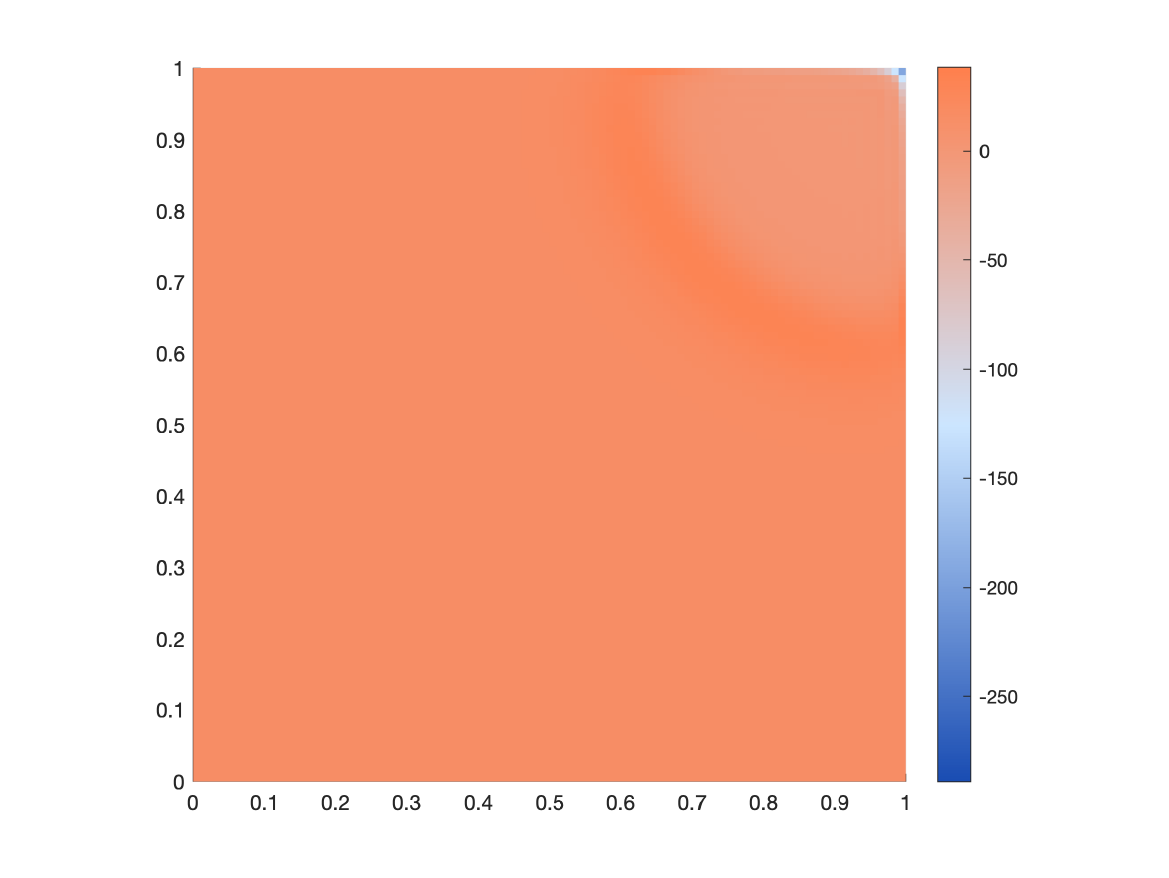}
	\end{minipage}
	\begin{minipage}{0.30\textwidth}
		\centering
		\includegraphics[width=\textwidth]{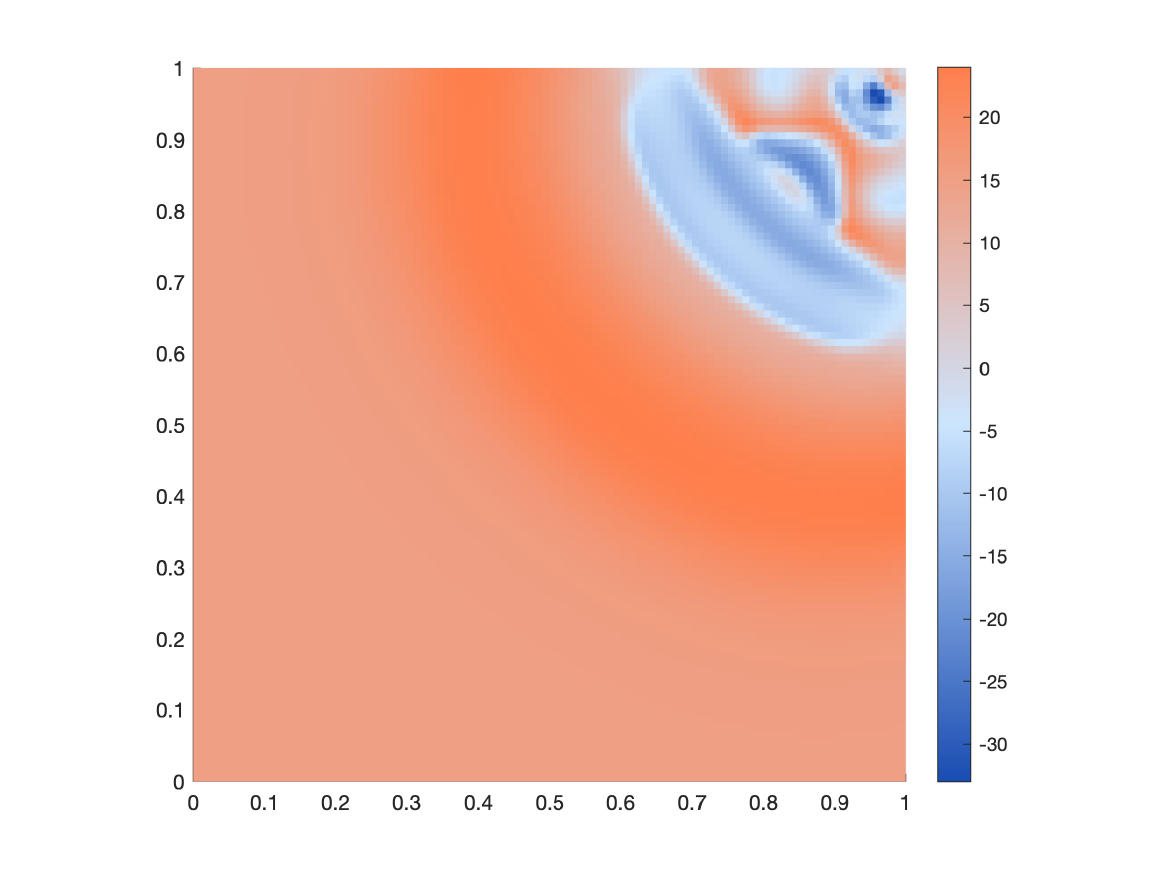}
	\end{minipage}
	\begin{minipage}{0.30\textwidth}
		\centering
		\includegraphics[width=\textwidth]{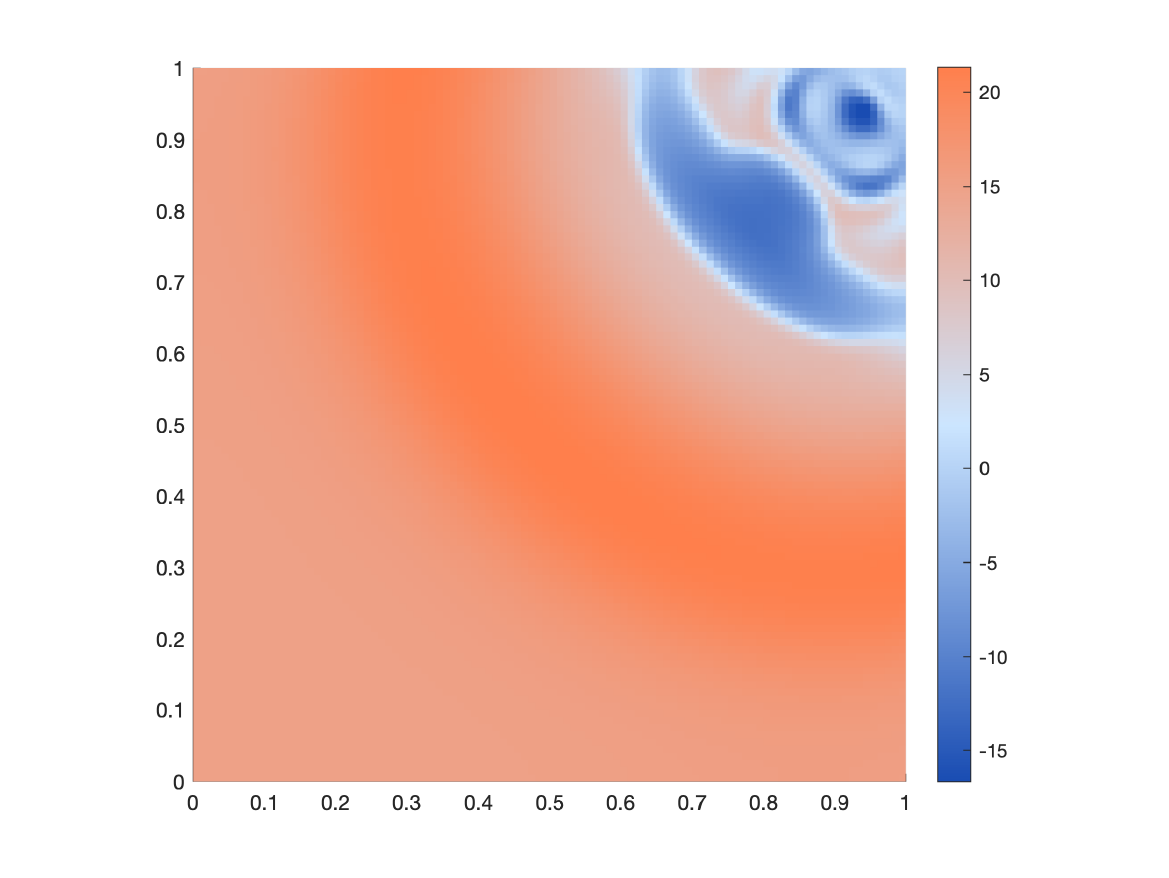}
	\end{minipage}
	
	\vspace{0.8em}
	\begin{minipage}{0.30\textwidth}
		\centering
		\includegraphics[width=\textwidth]{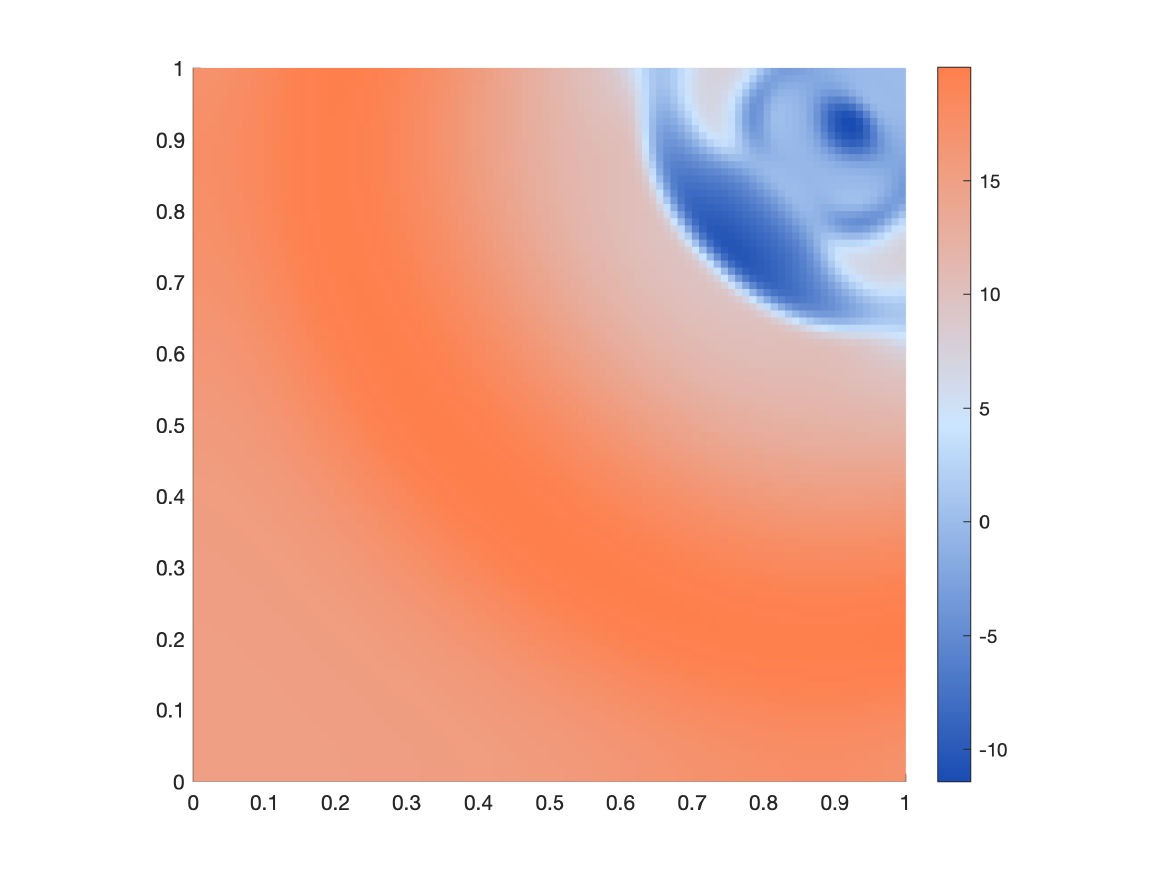}
	\end{minipage}
	\begin{minipage}{0.30\textwidth}
		\centering
		\includegraphics[width=\textwidth]{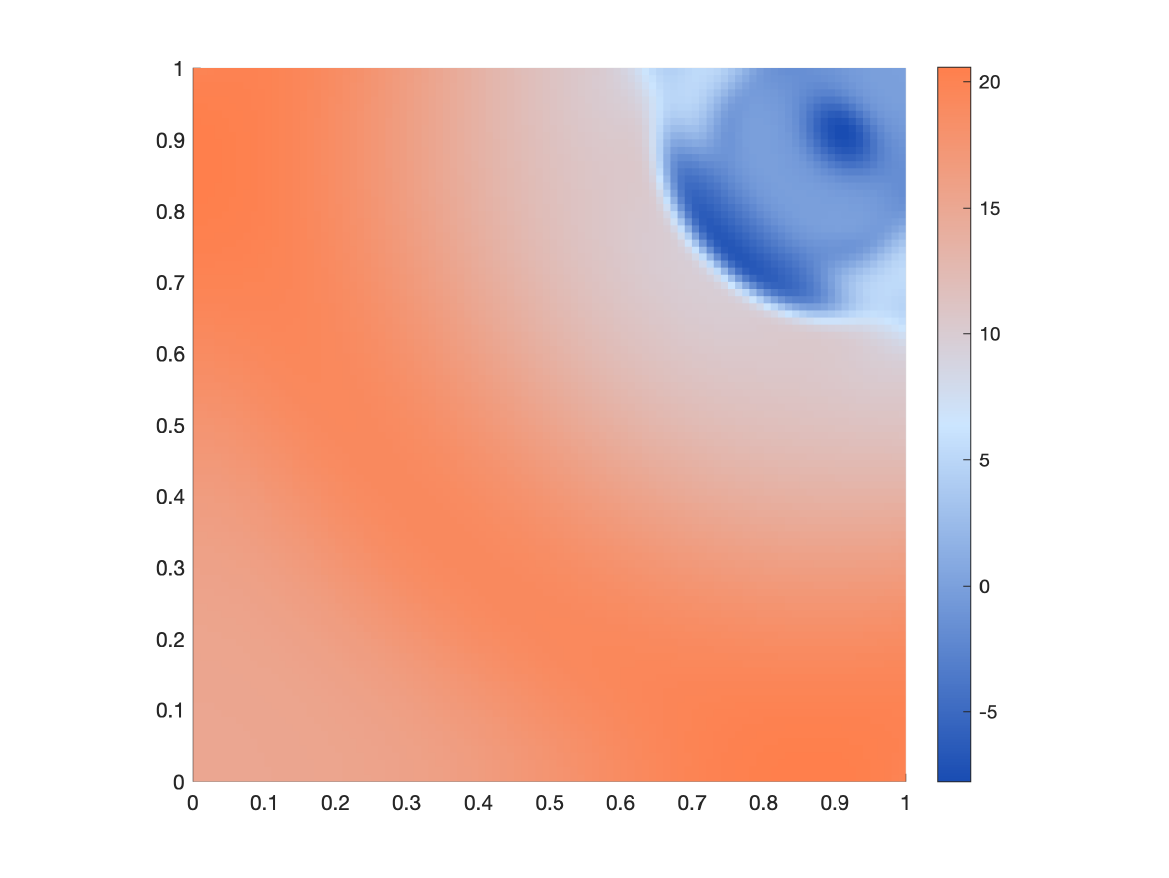}
	\end{minipage}
	\begin{minipage}{0.30\textwidth}
		\centering
		\includegraphics[width=\textwidth]{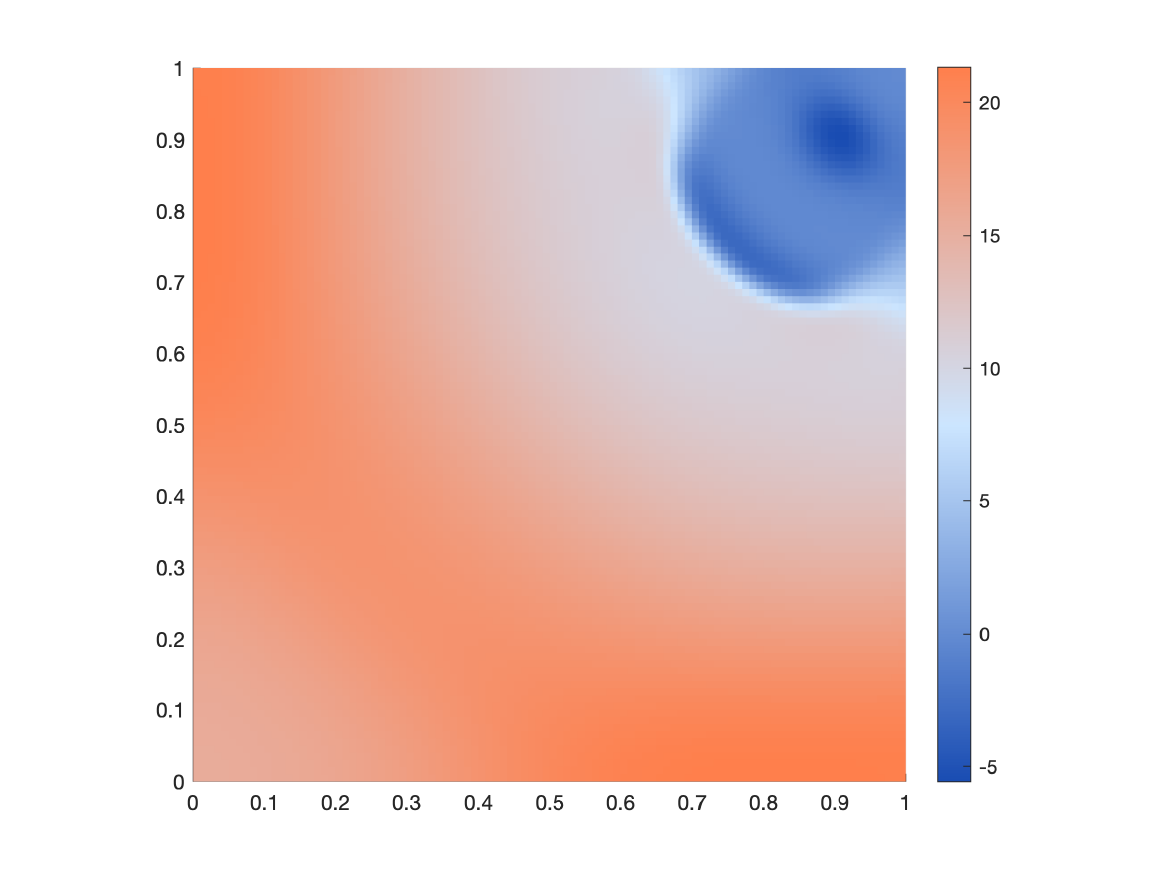}
	\end{minipage}
	
	\caption{Cell density at $k=1,10,15,20,25,30$.}
	\label{biogu-benckmark2-eta}
\end{figure}

\begin{figure}[htbp] 
	\centering 
	\begin{minipage}{0.45\textwidth} 
		\centering  
		\includegraphics[width=\textwidth]{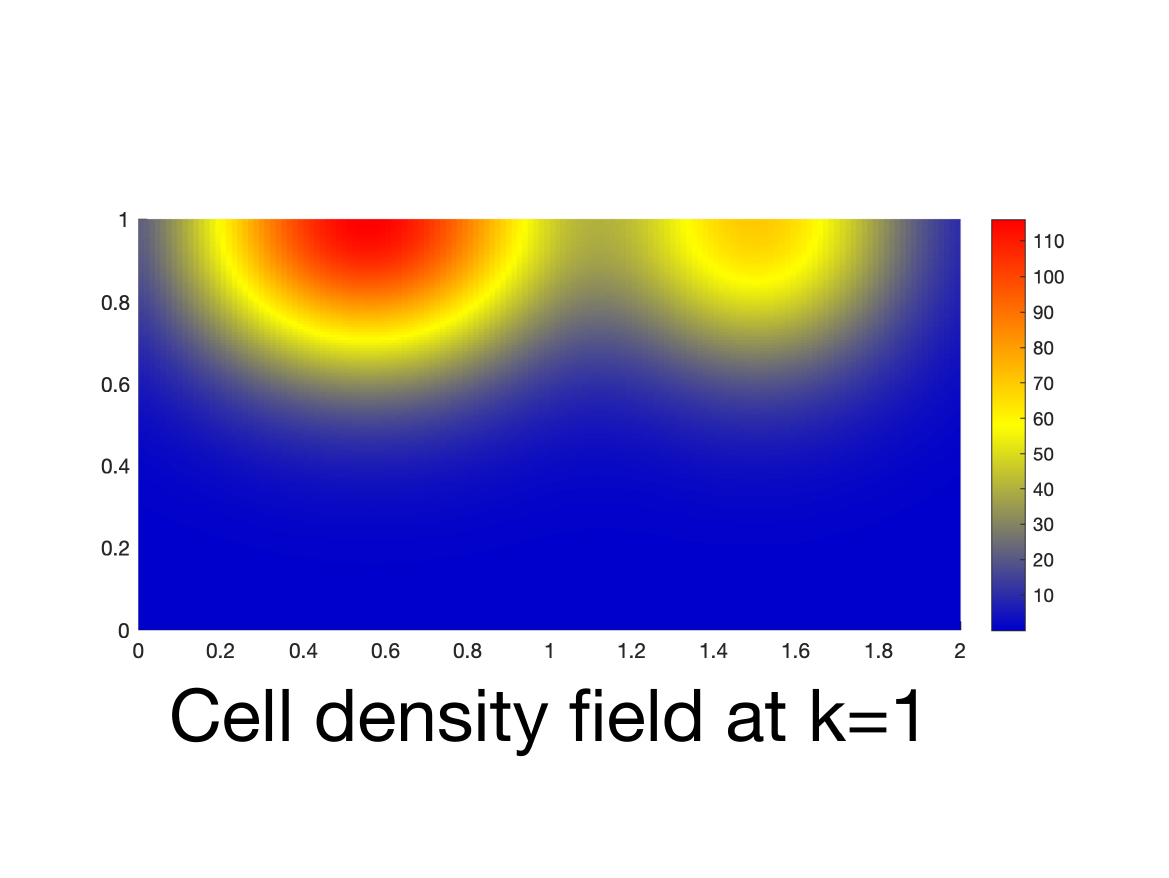}  
	\end{minipage}  
	\begin{minipage}{0.45\textwidth} 
		\centering  
		\includegraphics[width=\textwidth]{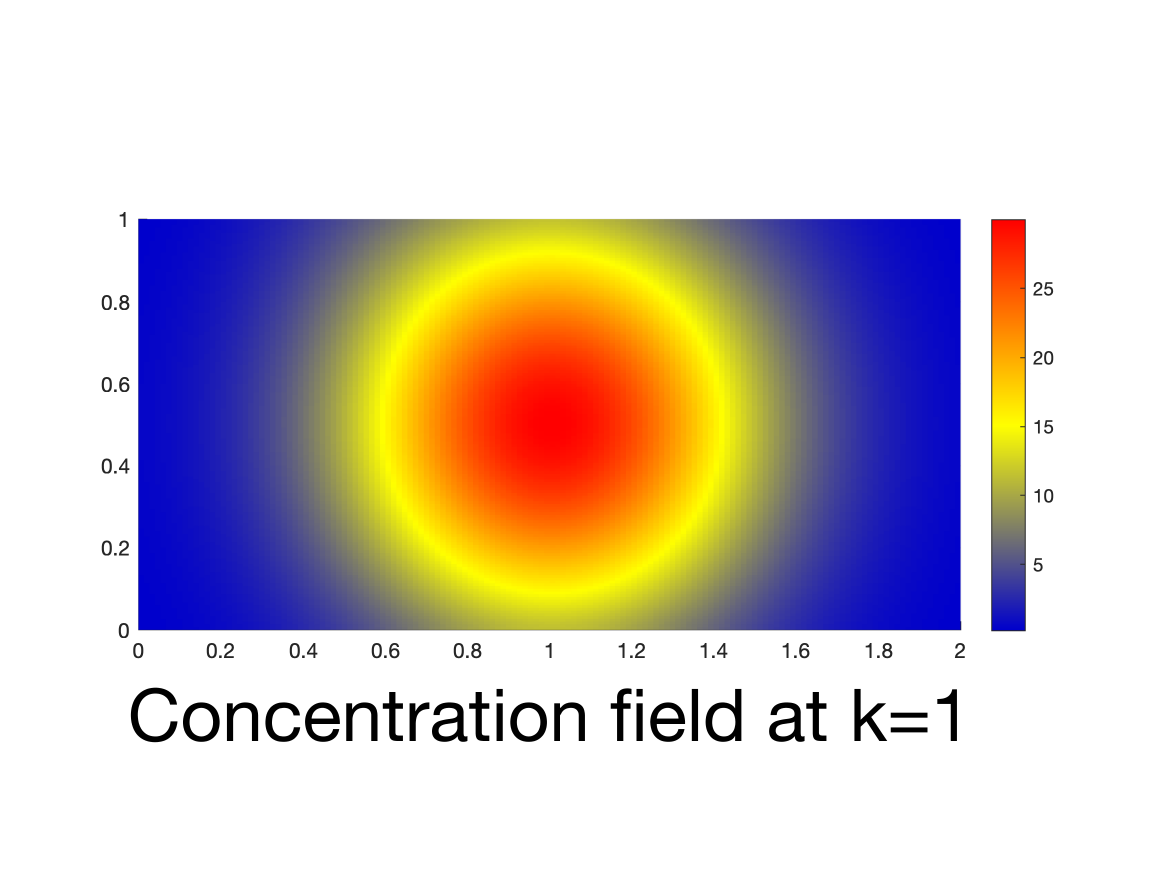} 
	\end{minipage}  
	\qquad
	\begin{minipage}{0.45\textwidth}  
		\centering  
		\includegraphics[width=\textwidth]{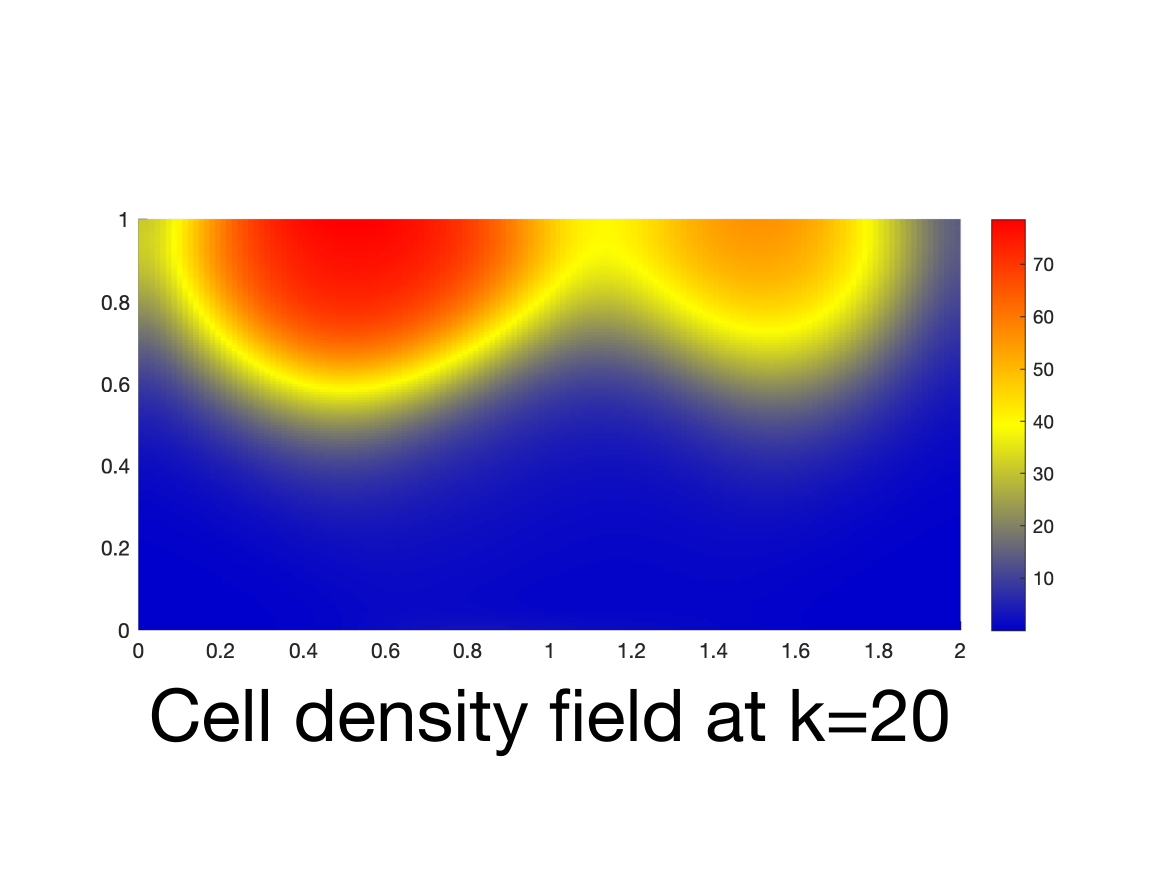}  
	\end{minipage}  	\begin{minipage}{0.45\textwidth} 
		\centering  
		\includegraphics[width=\textwidth]{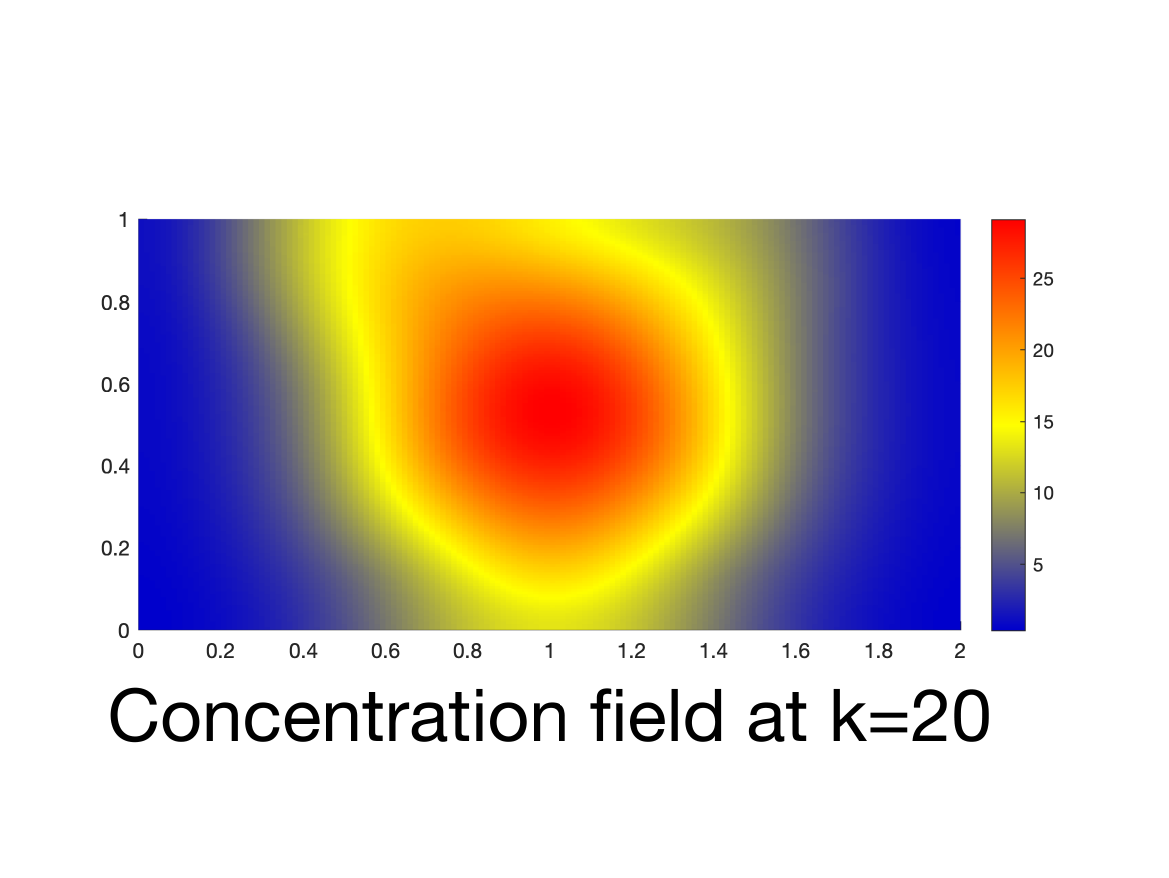}  
	\end{minipage}  
	\qquad
	\begin{minipage}{0.45\textwidth} 
		\centering  
		\includegraphics[width=\textwidth]{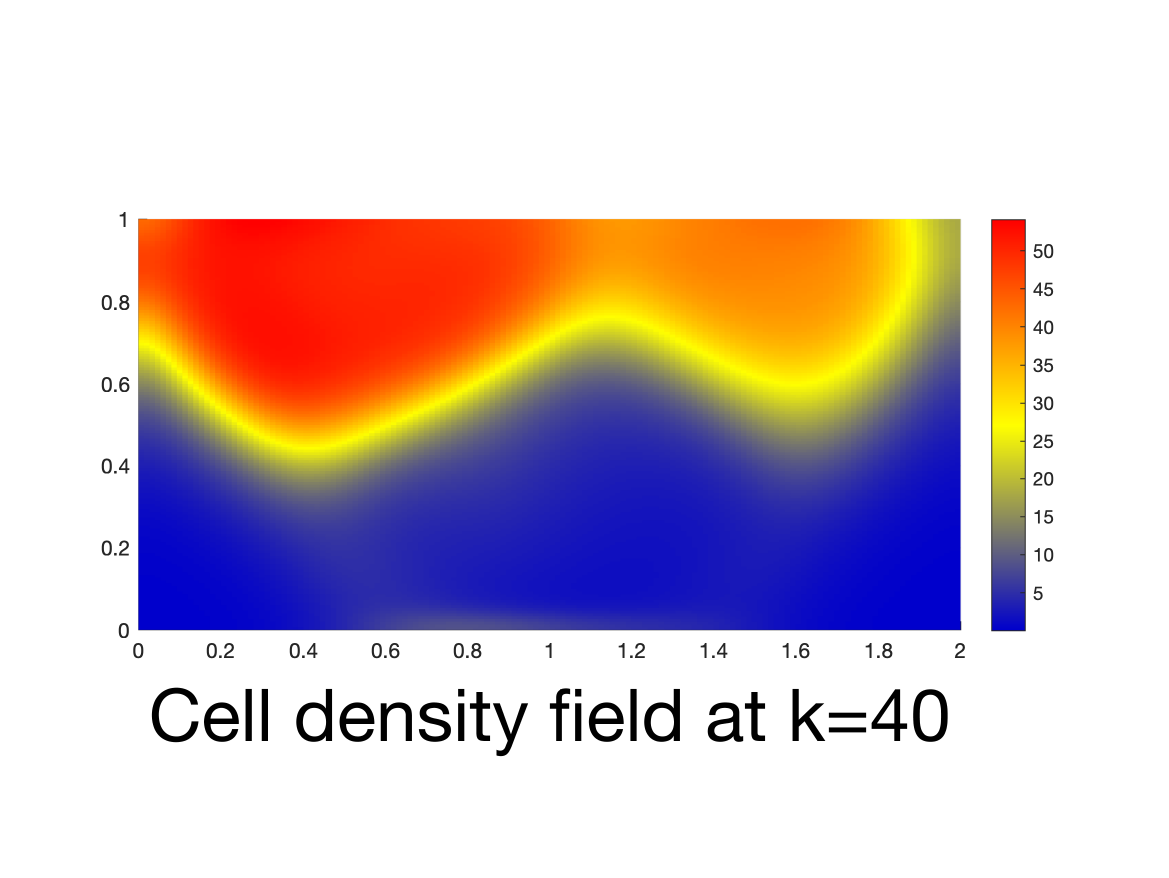} 
	\end{minipage}  
	\begin{minipage}{0.45\textwidth}  
		\centering  
		\includegraphics[width=\textwidth]{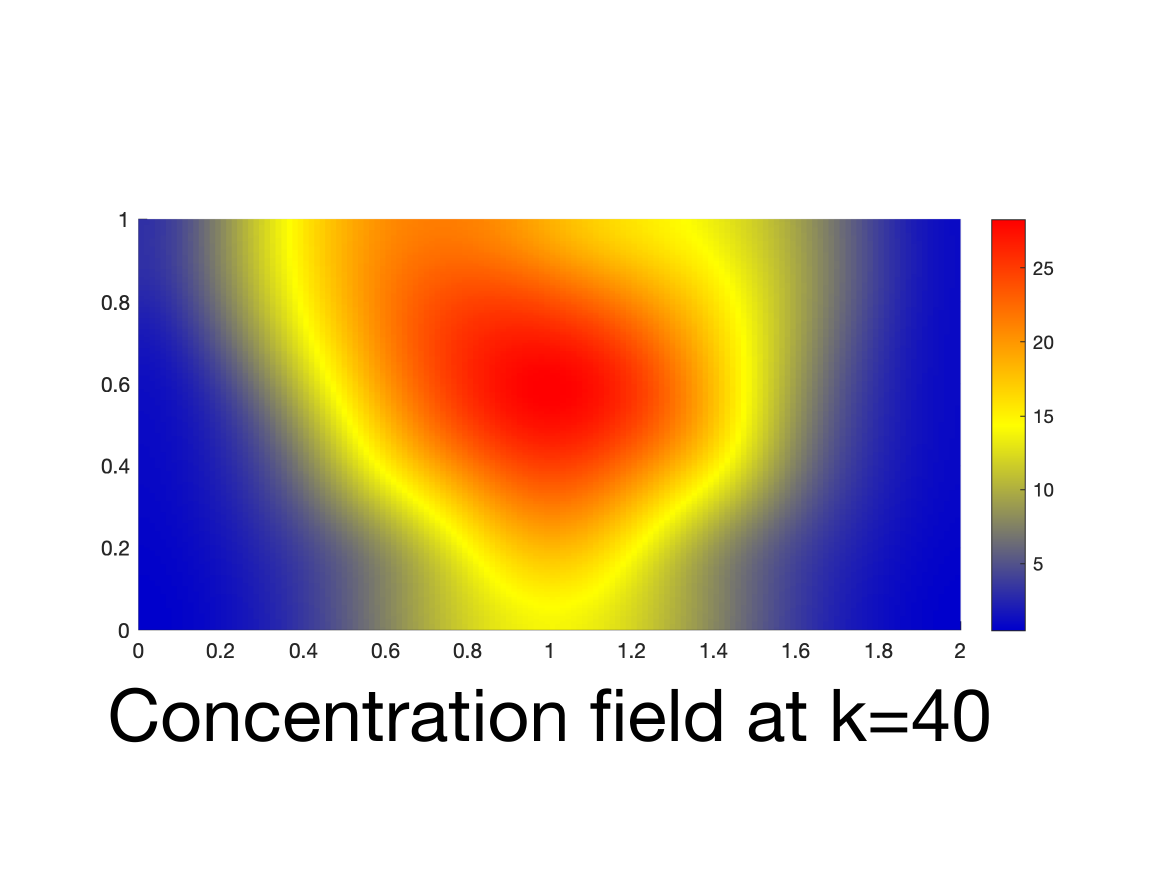}  
	\end{minipage}  
	\qquad
	\begin{minipage}{0.45\textwidth} 
		\centering  
		\includegraphics[width=\textwidth]{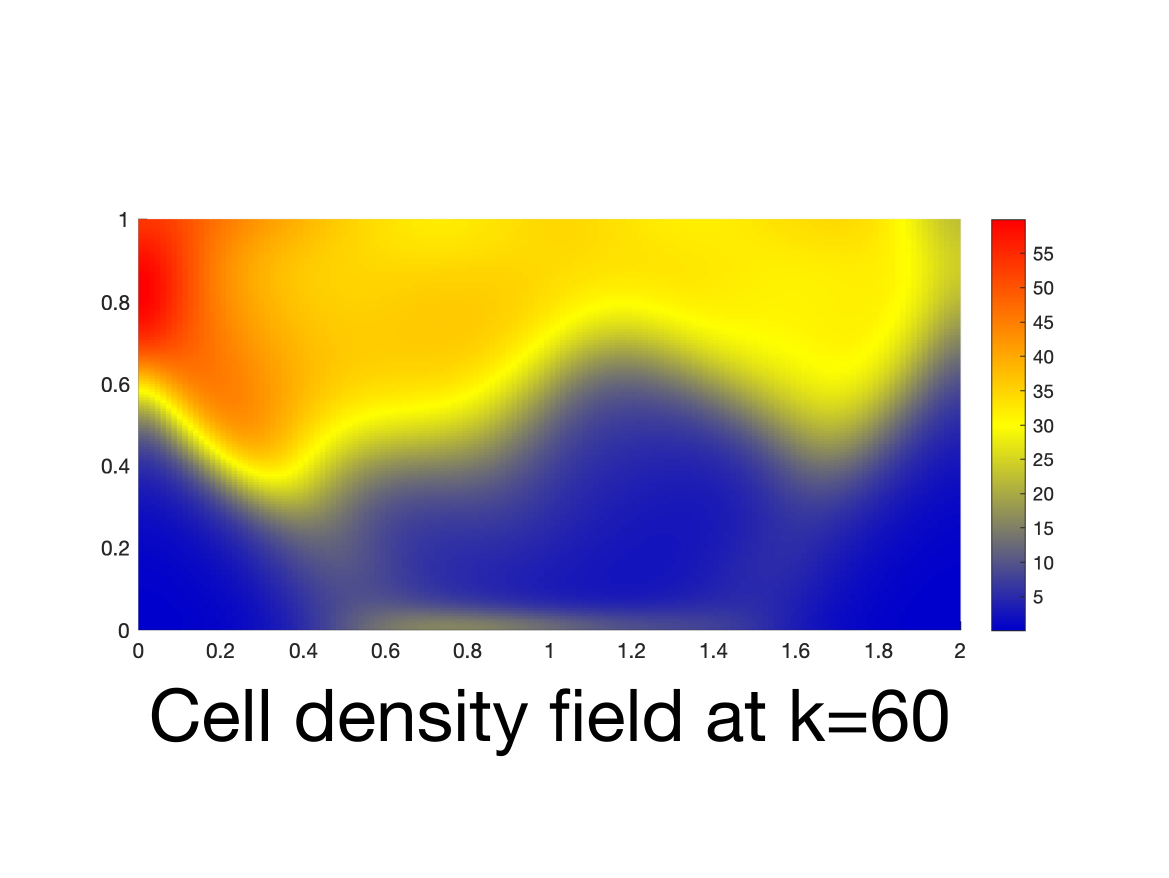} 
	\end{minipage}  
	\begin{minipage}{0.45\textwidth}  
		\centering  
		\includegraphics[width=\textwidth]{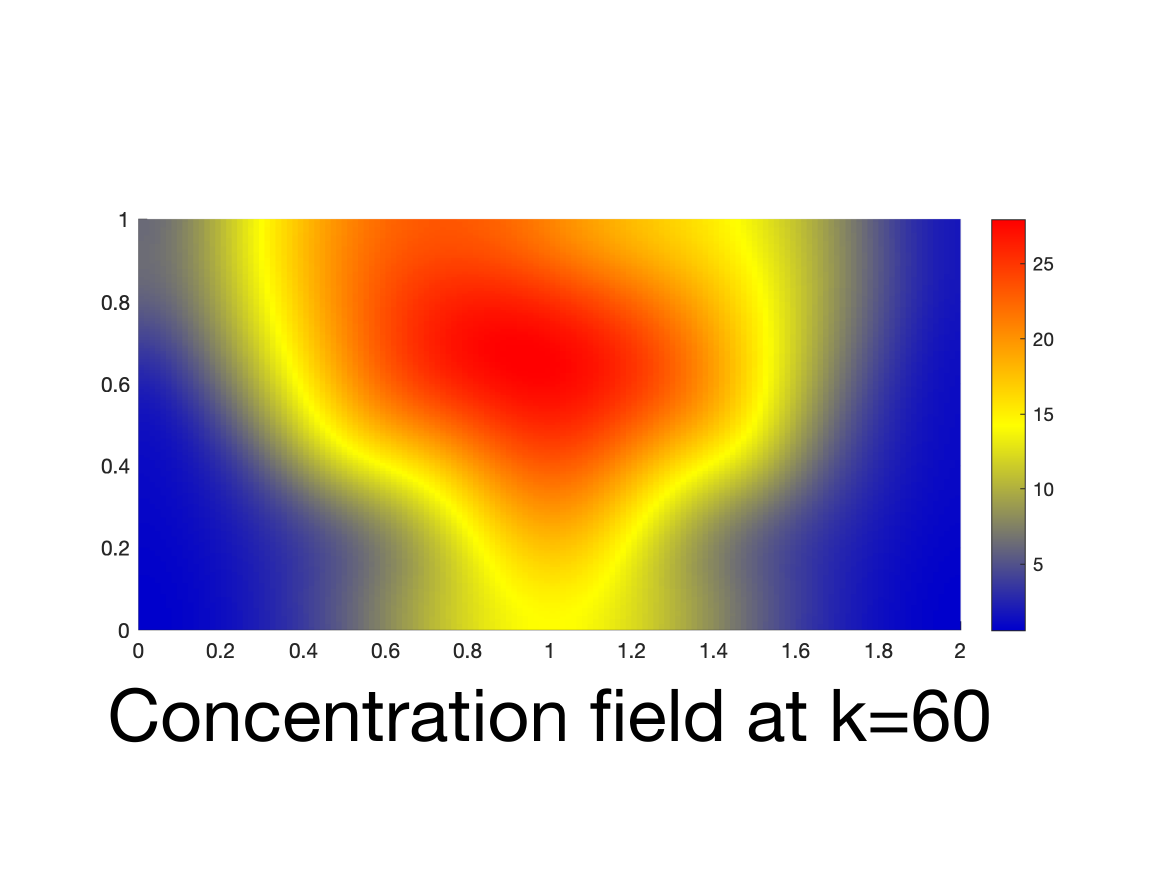}  
	\end{minipage}  
	\vspace{1.0em}
	\caption{Cell density vs. chemical concentration at $k=1,20,40,60$.}  
	\label{biogu-chemo-celldensity}  
\end{figure}

\begin{figure}[htbp] 
	\centering 
	\begin{minipage}{0.45\textwidth} 
		\centering  
		\includegraphics[width=\textwidth]{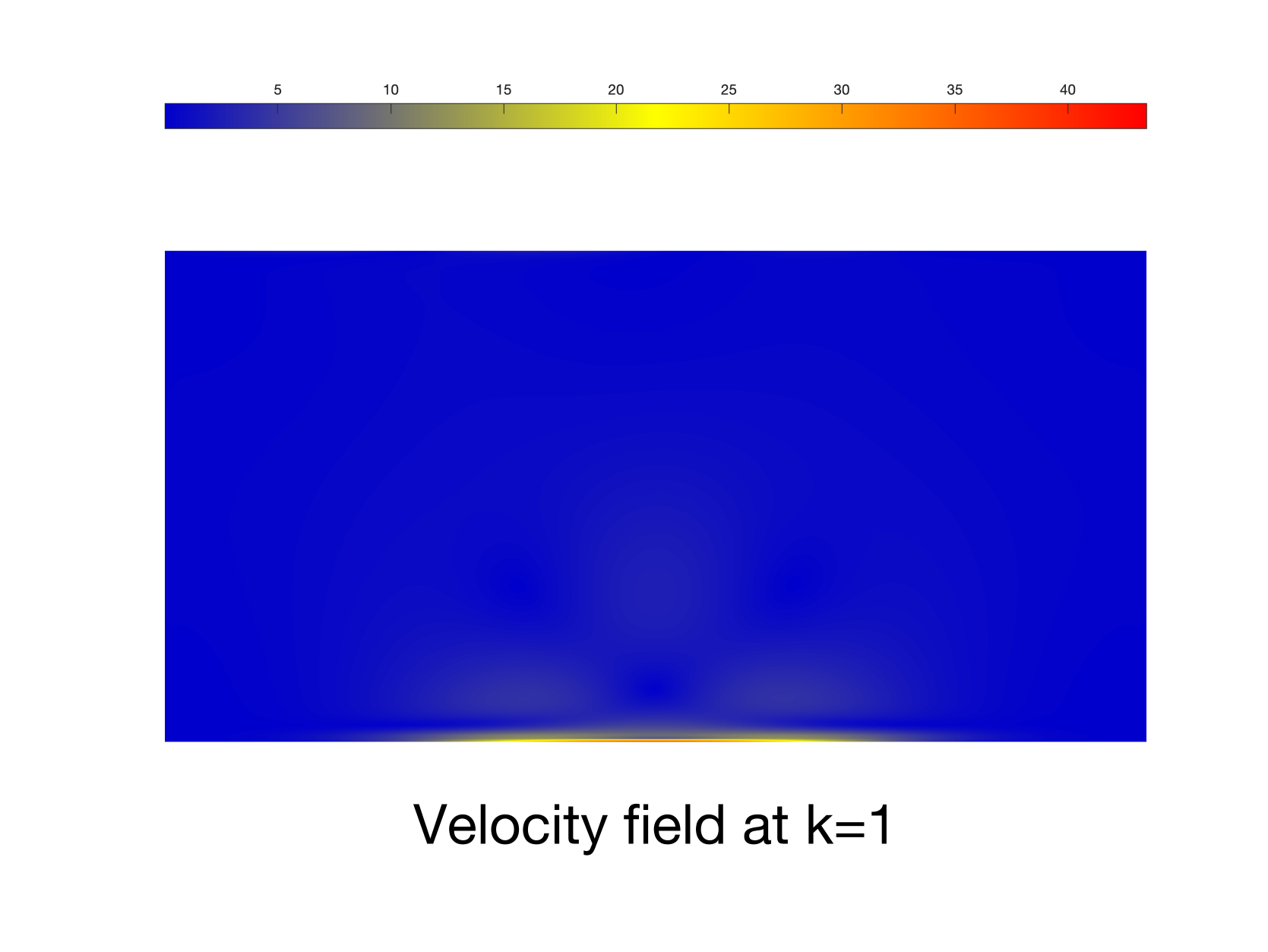}  
	\end{minipage}  
	\begin{minipage}{0.45\textwidth} 
		\centering  
		\includegraphics[width=\textwidth]{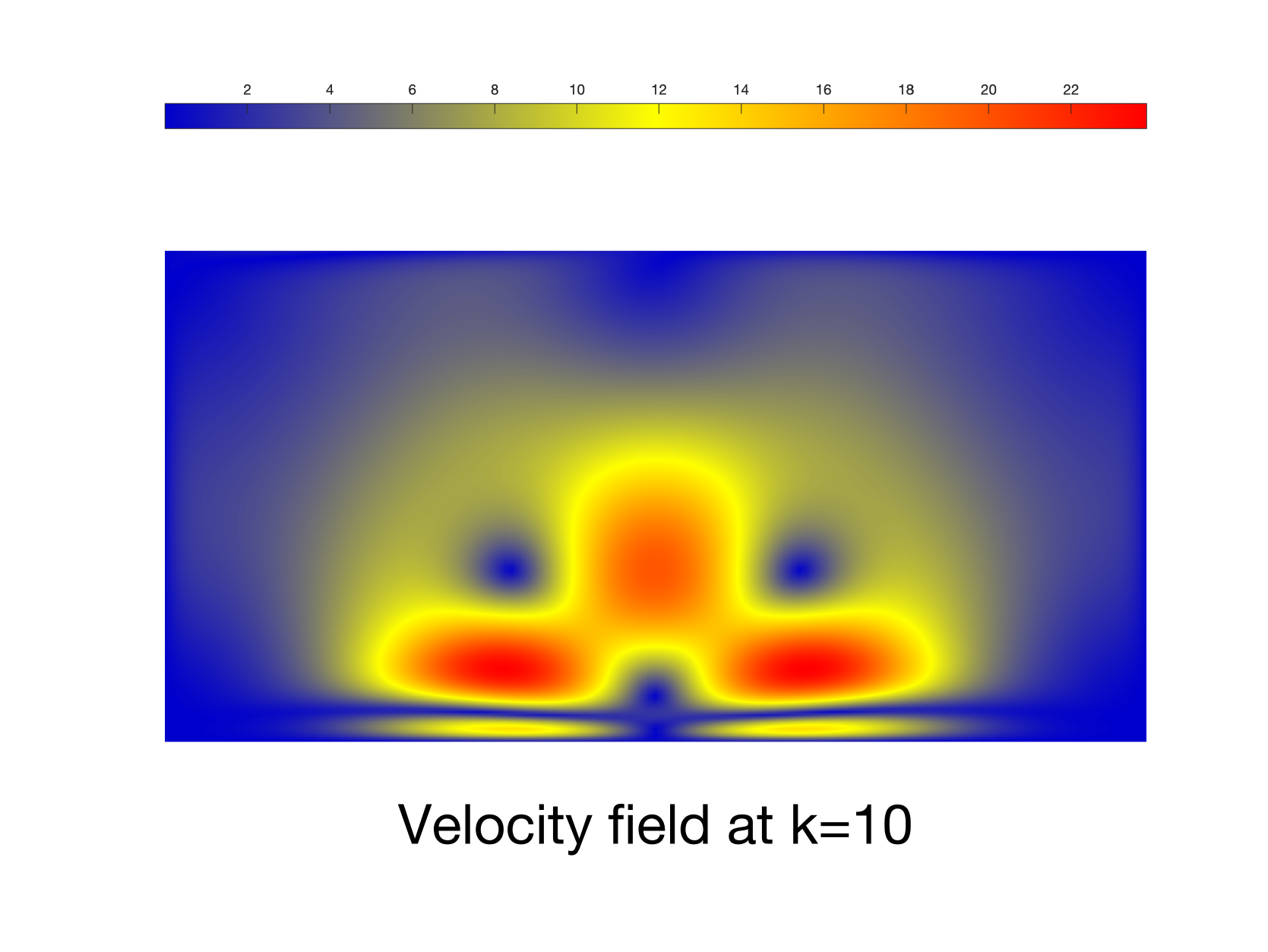} 
	\end{minipage}  
	\qquad
	\begin{minipage}{0.45\textwidth}  
		\centering  
		\includegraphics[width=\textwidth]{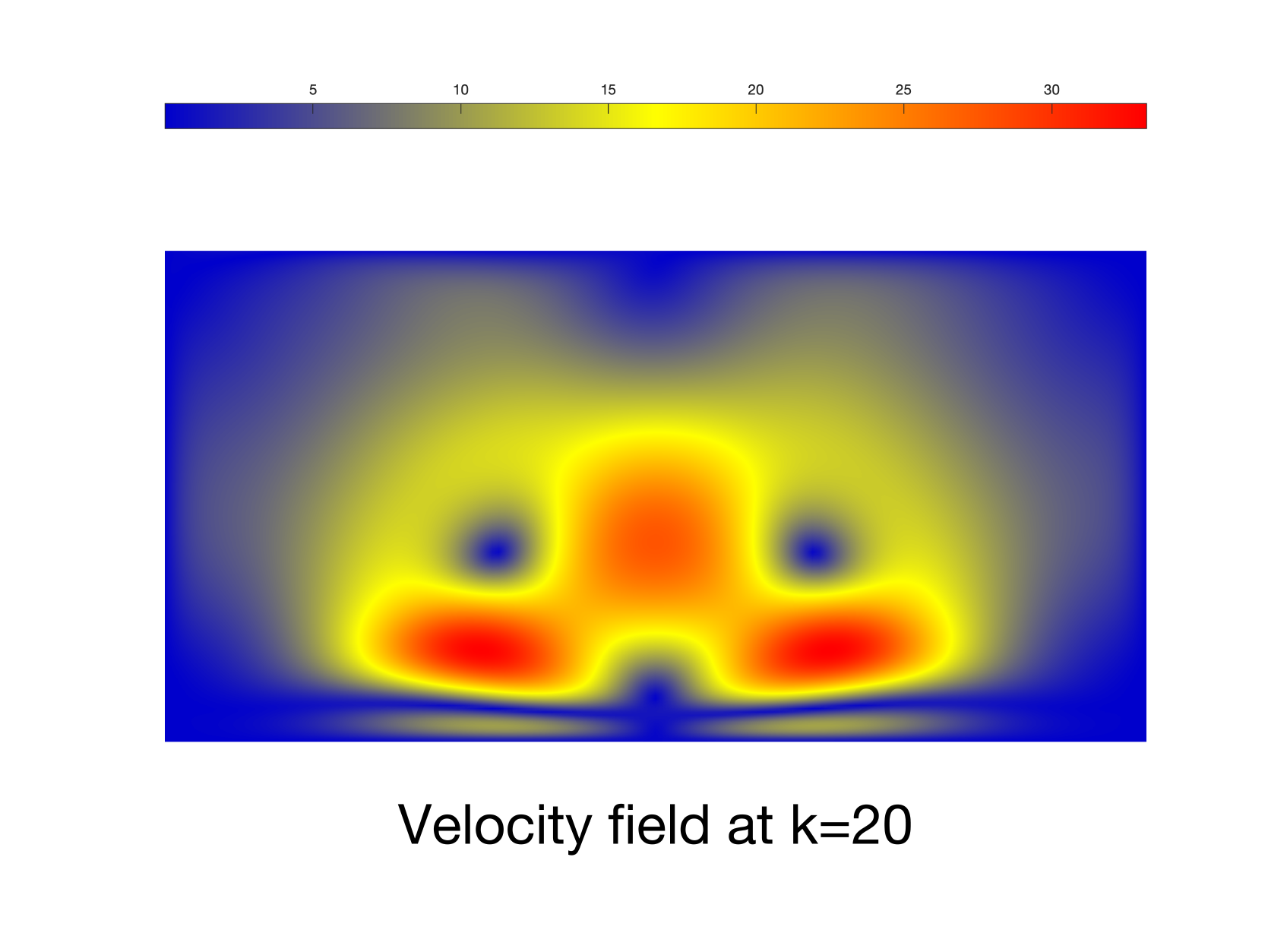}  
	\end{minipage}  	\begin{minipage}{0.45\textwidth} 
		\centering  
		\includegraphics[width=\textwidth]{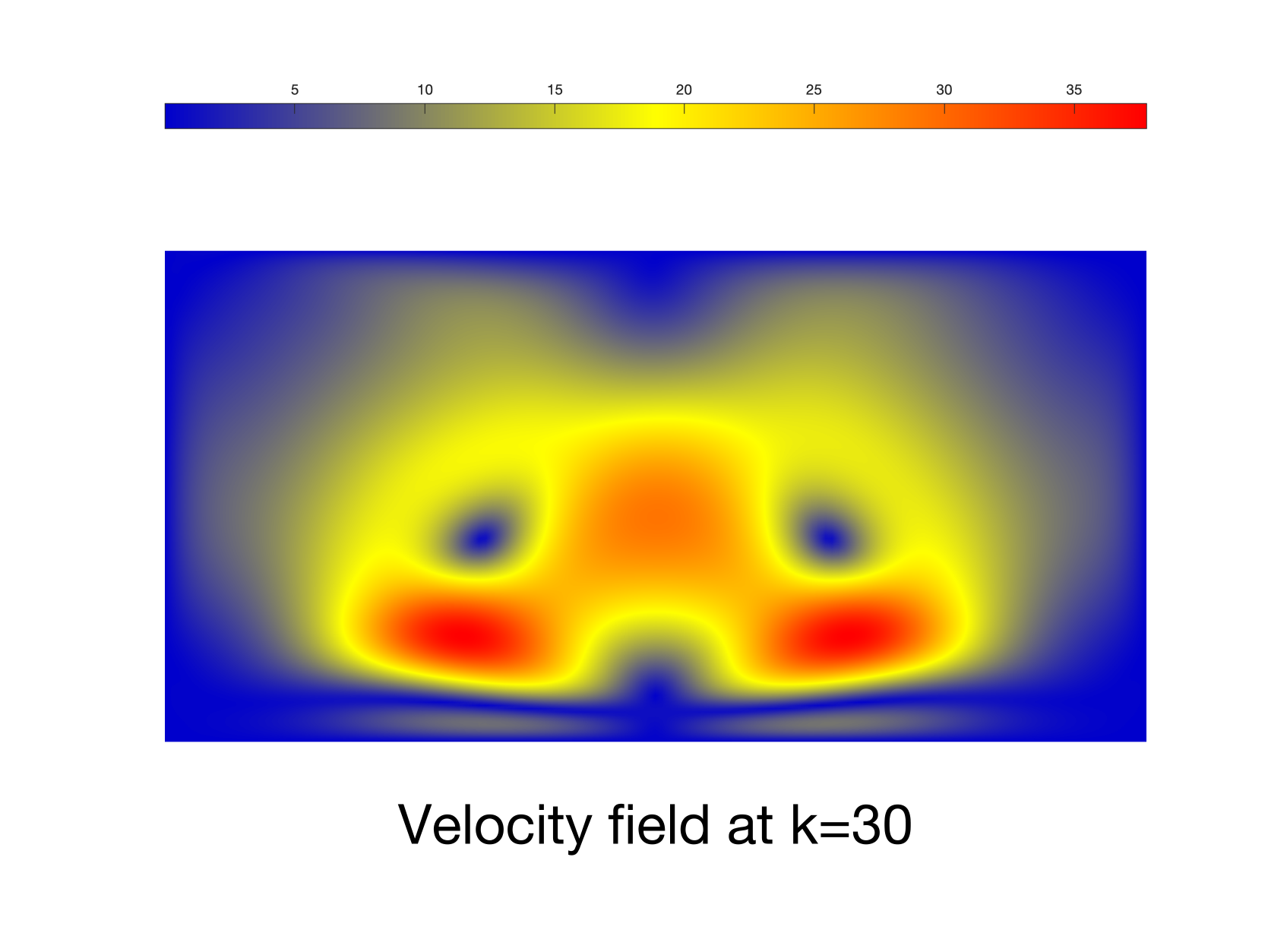}  
	\end{minipage}  
	\qquad
	\begin{minipage}{0.45\textwidth} 
		\centering  
		\includegraphics[width=\textwidth]{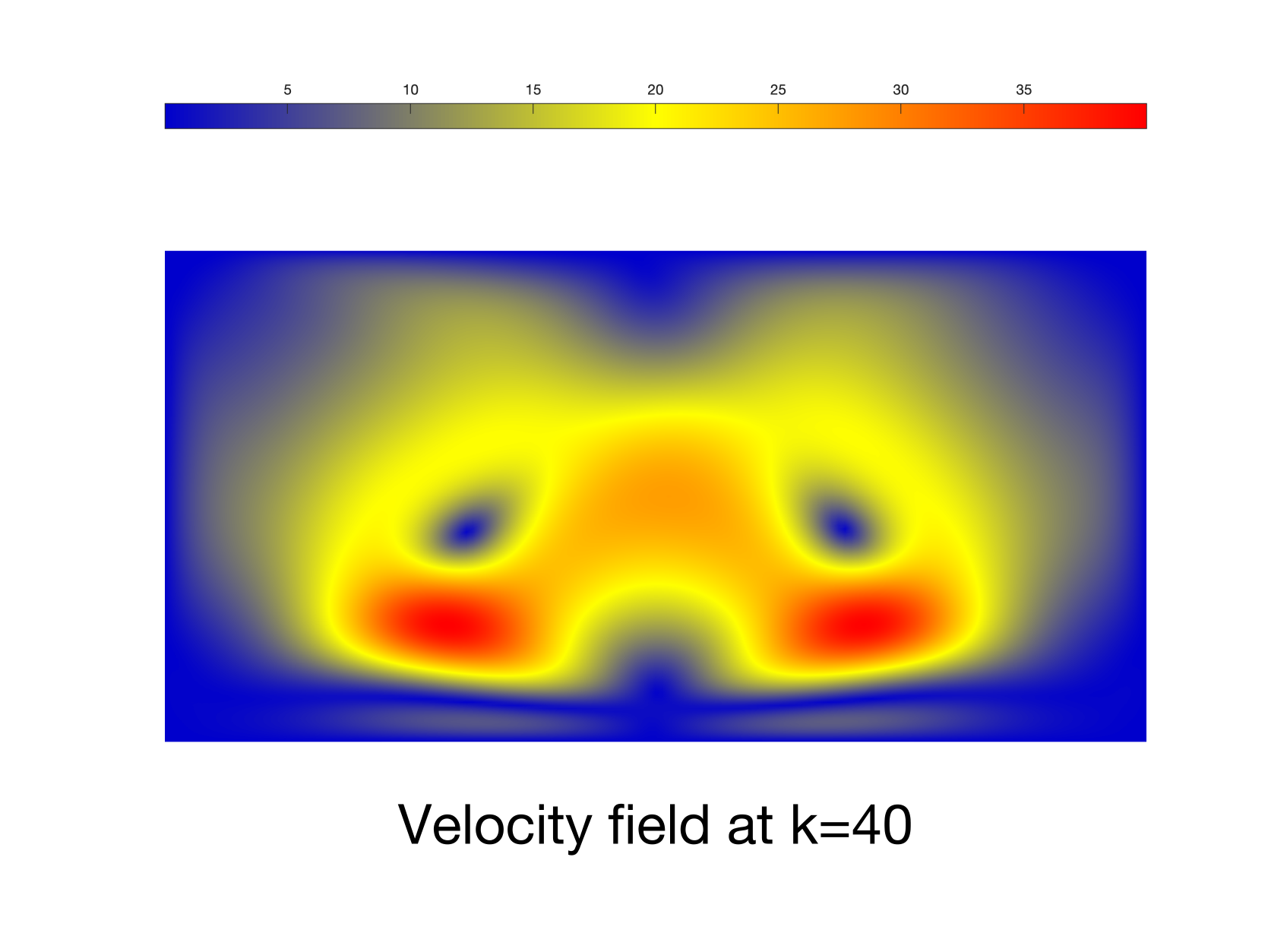} 
	\end{minipage}  
	\begin{minipage}{0.45\textwidth}  
		\centering  
		\includegraphics[width=\textwidth]{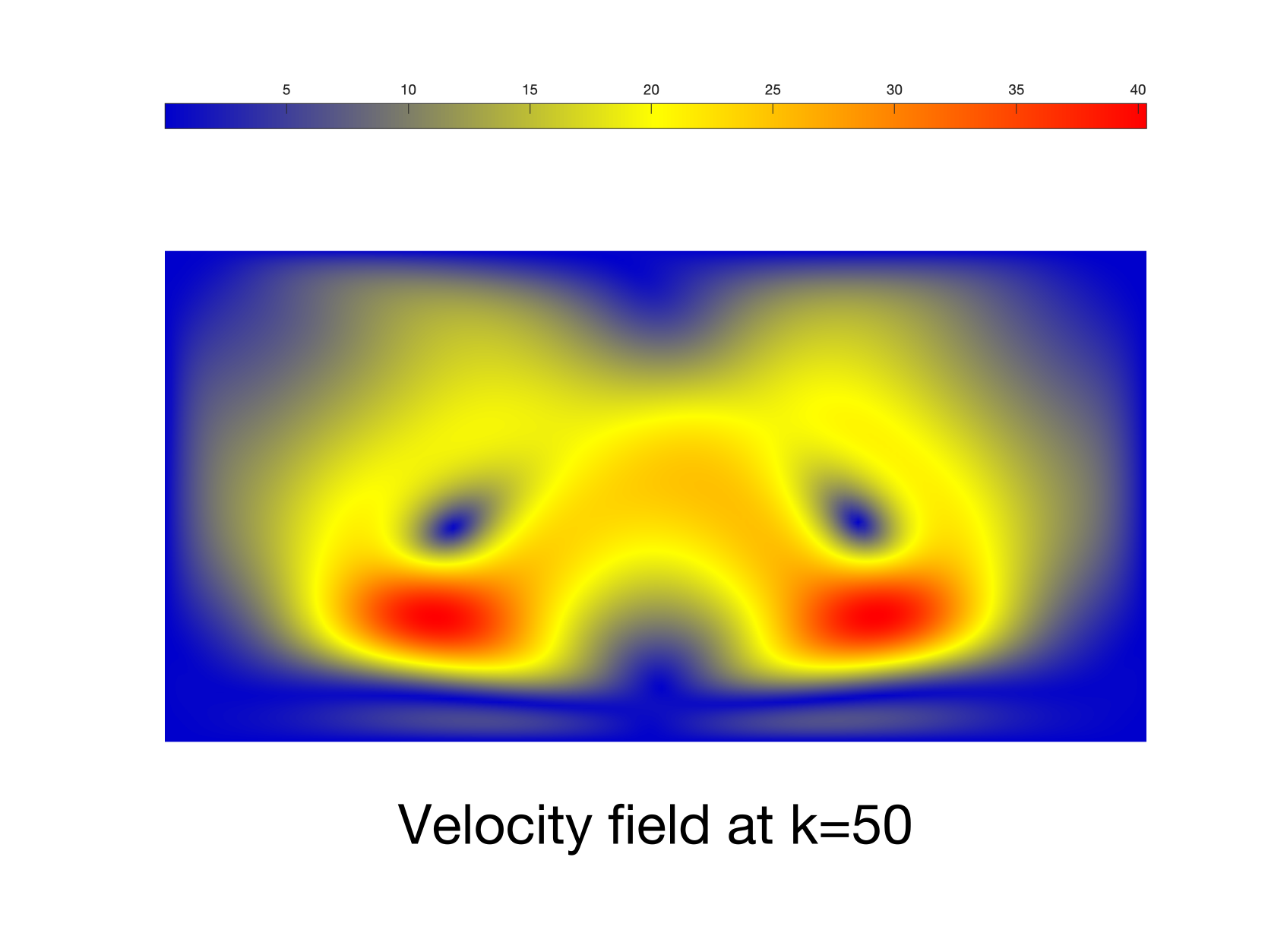}  
	\end{minipage}  
	\caption{Evolution of the velocity field of the fluid at $k=1,10,20,30,40,50$.}  
	\label{biogu-chemo-velocity}  
\end{figure}

\subsection{Microorganisms move towards lower chemical concentrations}
In this numerical example, we consider the computational domain 
$
\Omega = [0,2]\times[0,1],
$
and aim to simulate another behavior of microorganisms moving toward regions of lower chemo-attractant concentration. The initial conditions are prescribed as follows \cite{wang2023}:
\begin{align}
\begin{cases}
	\eta_0(x,y) &= 70\exp\!\big(-8(x-0.4)^2 - 8(y-1)^2\big)\\
	&\quad + 70\exp\!\big(-8(x-0.7)^2 - 8(y-1)^2\big) \\
	&\quad + 70\exp\!\big(-8(x-1.5)^2 - 8(y-1)^2\big), \\
	c_0(x,y) &= 30\exp\!\big(-4(x-1)^2 - 4(y-0.5)^2\big), \\
	\u_0(x,y) &= \textbf{0}, \qquad p_0(x,y)=0.
\end{cases}
\end{align}

The model coefficients are chosen as 
$
\mu_1 = \mu_2 = \mu_3 = 1.
$
The numerical solution is computed using a uniform time step 
$
\tau = 0.0001,
$
and mesh size 
$
h = 1/150.
$
We visualize the numerical solutions at different time step $k$. Figure \ref{biogu-chemo-celldensity} displays the evolution of microorganism density and chemical concentration, it is evident that the chemoattractant concentration is initially highest near the central region of the domain. Microorganisms located in the structured initial distribution begin migrating toward regions of lower chemoattractant concentration.  
Due to the motility of microorganisms, their movement induces fluid motion, which is observable in Figure \ref{biogu-chemo-velocity}. Subsequently, it demonstrates that the chemoattractant begins to diffuse outward under the influence of fluid transport. As the microorganisms respond to the gradient of chemoattractant, stronger directional motion is observed toward regions of lower concentration, leading to an accumulation of microorganism density near the lateral boundaries.  
In summary, the simulations confirm that the microorganism dynamics are predominantly governed by movement along descending chemoattractant gradients.

\section{Conclusion}

In this paper, we have developed a novel first-order fully discrete numerical scheme for the chemo--repulsion--Navier--Stokes system based on the gauge--Uzawa method.
The proposed scheme is rigorously proven to be unconditionally energy stable.
Moreover, we have derived unique solvability and  error estimates for all primary variables, providing a solid theoretical foundation for the accuracy of the method.
Finally, a series of numerical experiments have been conducted, which confirm the accuracy, stability, and efficiency of the proposed scheme. In future research, we aim to develop high-order energy-stable numerical algorithms that can further enhance the accuracy and stability of the proposed framework.
Such extensions are expected to broaden the applicability of our approach to more general chemotaxis--fluid systems and other multi-physics coupling problems.

 \section*{CRediT authorship contribution statement}
 Chenyang Li:
 Writing -- original draft, Visualization, Validation, Software, Methodology, Conceptualization;
Ping Lin:
 Methodology, Conceptualization;
Haibiao Zheng:
 Methodology, Conceptualization;
 
 \section*{Data availability}
 Data will be made available on request.

\section*{Declaration of competing interest}
 The authors declare that they have no known competing financial interests or personal relationships
 that could have appeared to influence the work reported in this paper.
 \section*{Acknowledgments}
 The authors would like to thank the editor and referees for their valuable comments and suggestions
 which helped us to improve the results of this paper. This work was supported by National Natural Science Foundation of China (No. 12471406) and the Science
 and Technology Commission of Shanghai Municipality (Grant Nos. 22JC1400900, 22DZ2229014).

	\bibliographystyle{abbrv}

\end{document}